\documentclass[12pt, reqno]{amsart}
\usepackage{amssymb}
\usepackage{graphicx}
\usepackage{xcolor} 
\usepackage{tensor}

\usepackage{fullpage} 

\usepackage{enumitem}
\setlist[enumerate]{leftmargin=2em, label=(\arabic*)}
\setlist[itemize]{leftmargin=2em}

\usepackage{seqsplit}
\usepackage{hyperref}

\definecolor{green}{rgb}{0,0.8,0} 

\newtheorem{theorem}{Theorem}[section]
\newtheorem{corollary}[theorem]{Corollary}
\newtheorem{claim}[theorem]{Claim}
\newtheorem{lemma}[theorem]{Lemma}
\newtheorem{proposition}[theorem]{Proposition}
\theoremstyle{definition}
\newtheorem{definition}[theorem]{Definition}

\theoremstyle{remark}
\newtheorem{remark}[theorem]{Remark}
\numberwithin{equation}{section}
\newcommand{\nrm}[1]{\Vert#1\Vert}
\newcommand{\abs}[1]{\vert#1\vert}
\newcommand{\brk}[1]{\langle#1\rangle}
\newcommand{\set}[1]{\{#1\}}

\newcommand{\tr}{\mathrm{tr}\,}

\newcommand{\aeq}{\simeq}
\newcommand{\aleq}{\lesssim}
\newcommand{\ageq}{\gtrsim}

\newcommand{\lap}{\Delta}

\newcommand{\ud}{\mathrm{d}}
\newcommand{\rd}{\partial}
\newcommand{\nb}{\nabla}

\newcommand{\imp}{\Rightarrow}

\newcommand{\0}{\emptyset}

\newcommand{\alp}{\alpha}
\newcommand{\bt}{\beta}
\newcommand{\gmm}{\gamma}
\newcommand{\Gmm}{\Gamma}
\newcommand{\dlt}{\delta}
\newcommand{\Dlt}{\Delta}
\newcommand{\eps}{\epsilon}

\newcommand{\kpp}{\kappa}
\newcommand{\lmb}{\lambda}

\newcommand{\sgm}{\sigma}
\newcommand{\Sgm}{\Sigma}

\newcommand{\Tht}{\Theta}

\newcommand{\omg}{\omega}
\newcommand{\Omg}{\Omega}

\newcommand{\zt}{\zeta}

\newcommand{\bfi}{{\bf i}}
\newcommand{\bfj}{{\bf j}}
\newcommand{\bfk}{{\bf k}}

\newcommand{\bfm}{{\bf m}}

\newcommand{\bfD}{{\bf D}}

\newcommand{\bfG}{{\bf G}}

\newcommand{\bbC}{\mathbb C}

\newcommand{\bbN}{\mathbb N}

\newcommand{\bbR}{\mathbb R}
\newcommand{\bbS}{\mathbb S}

\newcommand{\bbZ}{\mathbb Z}

\newcommand{\calA}{\mathcal A}

\newcommand{\calD}{\mathcal D}
\newcommand{\calE}{\mathcal E}

\newcommand{\calG}{\mathcal G}
\newcommand{\calH}{\mathcal H}

\newcommand{\calL}{\mathcal L}

\newcommand{\calO}{\mathcal O}

\newcommand{\frkg}{\mathfrak g}

\newcommand{\weakto}{\rightharpoonup}

\newcommand{\upto}{\nearrow}

\newcommand{\G}{\bfG}
\newcommand{\g}{\frkg}
\newcommand{\covD}{\bfD}

\newcommand{\tA}{\tilde{A}}
\newcommand{\ta}{\tilde{a}}
\newcommand{\te}{\tilde{e}}
\newcommand{\ba}{\bar{a}}
\newcommand{\be}{\bar{e}}

\newcommand{\pfstep}[1]{\vskip.5em \noindent {\bf #1.}}

\newcommand{\rcsp}{\underline{r}_{c}}		
\newcommand{\Rcsp}{\underline{R}_{c}}		

\newcommand{\rc}{r_{c}}		

\newcommand{\srd}{\mbox{$\not \hskip-.25em \rd$}}
\newcommand{\scovD}{\mbox{$\not \!\! \covD$}}
\newcommand{\Egs}{E_{GS}}				
\newcommand{\spE}{\calE_{e}}				

\newcommand{\ch}{\boldsymbol{\chi}}

\newcommand{\h}{\mathfrak{h}}

\begin{document}

\title[]{The hyperbolic Yang--Mills equation for \\ connections in an
  arbitrary topological class}
\author{Sung-Jin Oh}%
\address{Department of Mathematics, UC Berkeley, Berkeley, CA 94720 and KIAS, Seoul, Korea 02455}%
\email{sjoh@math.berkeley.edu}%

\author{Daniel Tataru}%
\address{Department of Mathematics, UC Berkeley, Berkeley, CA 94720}%
\email{tataru@math.berkeley.edu}%


\begin{abstract}
  This is the third part of a four-paper sequence, which establishes
  the Threshold Conjecture and the Soliton-Bubbling vs.~Scattering
  Dichotomy for the energy critical hyperbolic Yang--Mills equation in
  the $(4+1)$-dimensional Minkowski space-time. This paper provides
  basic tools for considering the dynamics of the hyperbolic
  Yang--Mills equation in an arbitrary topological class at an optimal
  regularity.

  We generalize the standard notion of a topological class of
  connections on $\bbR^{d}$, defined via a pullback to the one-point
  compactification $\bbS^{d} = \bbR^{d} \cup \set{\infty}$, to rough
  connections with curvature in the critical space
  $L^{\frac{d}{2}}(\bbR^{d})$.  Moreover, we provide excision and
  extension techniques for the Yang--Mills constraint (or Gauss)
  equation, which allow us to efficiently localize Yang--Mills initial
  data sets.
  Combined with the results in the previous paper \cite{OTYM2}, we
  obtain local well-posedness of the hyperbolic Yang--Mills equation
  on $\bbR^{1+d}$ $(d \geq 4)$ in an arbitrary topological class at
  optimal regularity in the temporal gauge (where finite speed of
  propagation holds). In addition, in the energy subcritical case $d =
  3$, our techniques provide an alternative proof of the classical
  finite energy global well-posedness theorem of Klainerman--Machedon
  \cite{KlMa2}, while also removing the smallness assumption in the
  temporal-gauge local well-posedness theorem of Tao \cite{TaoYM}.

  Although this paper is a part of a larger sequence, the materials
  presented in this paper may be of independent and general interest.
  For this reason, we have organized the paper so that it may be read
  separately from the sequence.
\end{abstract}
\maketitle

\tableofcontents

\section{Introduction}
The subject of this paper is the $(d+1)$-dimensional hyperbolic
Yang--Mills equation with compact noncommutative structure group. Our
goal is two-fold:
\begin{itemize}
\item To describe, topologically and analytically, the Yang--Mills
  initial data sets at the optimal $L^{2}$-Sobolev regularity;
\item To provide a good local theory for solutions at the optimal
  $L^{2}$-Sobolev regularity.
\end{itemize}
In each case, we consider two model base spaces: Either a ball $B_{R}
= \set{x \in \bbR^{d} : \abs{x} < R}$ or the whole space $\bbR^{d}$
for the first goal, and (suitable time restrictions of) their
respective domains of dependence $\calD(B_{R}) = \set{(t, x) \in
  \bbR^{1+d} : \abs{t} + \abs{x} < R}$ and $\calD(\bbR^{d}) =
\bbR^{1+d}$ for the second goal.

The main results of this paper may be classified into three classes:
\begin{enumerate}
\item {\it Good global gauge and topological class of rough connections.} Motivated by the optimal regularity theory for the hyperbolic Yang--Mills equation, we consider locally-defined connections on a subset of $\bbR^{d}$ with $L^{\frac{d}{2}}$-curvature. Patching together the local gauges, we show that we can always produce good global gauges in the two model base spaces above (Theorems~\ref{thm:goodrep-ball} and \ref{thm:goodrep}). Moreover, in whole space case, we use the asymptotics of the good global gauge potential to extend the notion of topological classes of connections to the rough setting. (Definition~\ref{def:top-class}).

\item {\it Initial data surgery.}
We provide techniques for excising and extending Yang--Mills initial data sets, which are subject to the nonlinear Yang--Mills constraint (or Gauss) equation (Theorems~\ref{thm:ext-id} and \ref{thm:excise}). These are based on a sharp
  solvability result for the covariant divergence equation
  $\covD^{\ell} e_{\ell} = h$ which preserves physical space support
  property (Theorem~\ref{thm:gauss-0}).

\item {\it Large data local theory.}
Using the ideas of initial data surgery and patching solutions, we show how to extend a small data well-posedness result in the temporal gauge to arbitrarily large data; the key is that causality (or finite speed of propagation) holds in the temporal gauge.
Combined with the optimal regularity temporal gauge small data global well-posedness theorem proved in \cite{OTYM2}, we prove local well-posedness of the hyperbolic Yang--Mills equation in the temporal gauge for arbitrary critical Sobolev initial data in $d \geq 4$ (Theorem~\ref{thm:local-temp}). In $d = 3$, we obtain a generalization of a low regularity result of Tao \cite{TaoYM}, as well as an alternative proof of the classical result of Klainerman--Machedon \cite{KlMa2}.
\end{enumerate}

In addition, in the last section we provide a review of the theory of harmonic Yang--Mills equation on $\bbR^{4}$ using the topological framework developed in this paper. A particular emphasis is given to the recent sharp energy lower bound for non-instanton solutions due to Gursky--Kelleher--Streets \cite{GKS}, which clarifies the threshold energy for the energy critical hyperbolic Yang--Mills equation (and the Yang--Mills heat flow); namely, it is twice the ground state energy.

\begin{remark} \label{rem:series} When restricted to the energy
  critical dimension $d = 4$, the results in this paper constitute
  the third part of a four-paper sequence, whose principal aim is to
  prove the Threshold Theorem for the energy critical hyperbolic
  Yang--Mills equation. The four installments of the series are
  concerned with
  \begin{enumerate}
  \item the \emph{caloric gauge} for the hyperbolic Yang--Mills
    equation, \cite{OTYM1}.
  \item large data \emph{energy dispersed} caloric gauge solutions,
    \cite{OTYM2}.
  \item \emph{topological classes} of connections and large data local
    well-posedness, present article.
  \item \emph{soliton bubbling} vs. scattering dichotomy for large
    data solutions, \cite{OTYM3}.
  \end{enumerate}
  A short overview of the whole sequence is provided in the survey
  paper \cite{OTYM0}.

  The present paper is mostly independent of the other papers in the
  series; the only exception is the small data well-posedness result
  for the hyperbolic Yang--Mills equation from \cite{OTYM2} ($d \geq
  4$), which is used here as a black-box.
\end{remark}

This paper is structured as follows. In the remainder of the
introduction, we present the basic definitions and main results of
this paper. For the notation and conventions that are not explained in
the course of exposition, we refer the reader to
Section~\ref{sec:notation}. In
Sections~\ref{sec:rough-conn}--\ref{sec:threshold}, we elaborate and
provide proofs of the results stated in the introduction.

\subsection{Connections on a vector bundle with structure group
\texorpdfstring{  $\G$}{G}} \label{subsec:vb} Here we give a quick review of the basic
theory of connections on vector bundles, and at the same time fix some
notation and conventions. For a textbook treatment of these materials,
we recommend \cite{MR1393940, MR1393941, MR0440554}.

Let $\G$ be a compact Lie group with Lie algebra $\g$. We denote
the adjoint action of $\G$ on $\g$ by $Ad(O) A = O A O^{-1}$, and the
corresponding action of $\g$ by $ad(A) B = [A, B]$. We endow $\g$
  with an inner product $\brk{\cdot, \cdot}$ which is $Ad$-invariant
  (or bi-invariant), i.e.,
  \begin{align*}
    \brk{A, B} = \brk{Ad(O) A, Ad(O) B} \qquad A, B \in \g, \ O \in
    \G.
  \end{align*}
  Such an $Ad$-invariant inner product always exists if $\G$ is
  compact. Indeed, from any inner product $\brk{\cdot, \cdot}'$, we
  may construct an $Ad$-invariant inner product by applying $Ad(O)$ to
  each input and averaging in $O \in \G$.

The main objects we consider are connections $\covD$ on a vector
bundle on some smooth base manifold $X$ with structure group $\G$.
Here we recall the standard local definition of a vector bundle in the
smooth and continuous cases, which will be most useful later:
\begin{definition} \label{def:vb} A $C^{\infty}$ [resp. $C^{0}$]
  vector bundle $\eta$ on a smooth manifold $X$ with fibers modeled on
  a vector space $V$ consists of the following objects:
  \begin{itemize}
  \item An open cover $\set{U_{\alp}}$ of $X$;
  \item For each pair $U_{\alp}, U_{\bt}$, a $C^{\infty}$
    [resp. $C^{0}$] \emph{transition map} $O_{(\alp \bt)} : U_{\alp}
    \cap U_{\bt} \to Aut(V)$, which satisfy the following
    \emph{cocycle properties}:
    \begin{enumerate}
    \item $O_{(\alp \alp)} = I \quad \hbox{ on } U_{\alp} (= U_{\alp}
      \cap U_{\alp})$,
    \item $O_{(\alp \gmm)} = O_{(\alp \bt)} O_{(\bt \gmm)} \quad
      \hbox{ on } U_{\alp} \cap U_{\bt} \cap U_{\gmm}$.
    \end{enumerate}
  \end{itemize}
  Suppose that a Lie group $\G$ acts on $V$, in the sense that there
  exists a smooth representation $\rho : \G \to Aut(V)$.  We say that
  \emph{$\eta$ has structure group $\G$} if the transition functions
  may be lifted to $C^{\infty}$ [resp. $C^{0}$] $\G$-valued cocyles,
  i.e.,
  \begin{equation*}
    O_{(\alp \bt)} = \rho \circ \tilde{O}_{(\alp \bt)} \quad \hbox{ for some } \tilde{O}_{(\alp \bt)} : U_{\alp} \cap U_{\bt} \to \G
  \end{equation*}
  so that $\set{\tilde{O}_{(\alp \bt)}}$ satisfy the cocycle property.

  For simplicity, throughout the paper we omit the representation
  $\rho$ and denote the lifted cocycles $\tilde{O}_{(\alp \bt)}$ by
  $O_{(\alp \bt)}$.
\end{definition}

In the local formulation, vector bundles with structure group $\bfG$
defined by the data sets $\set{U_{\alp}, O_{(\alp \bt)}}$ and
$\set{U'_{\alp'}, O'_{(\alp' \bt')}}$ are \emph{isomorphic} if and
only if there exists a common refinement $\set{V_{\gmm}}$ of
$\set{U_{\alp}}$ and $\set{U'_{\alp'}}$, so that $V_{\gmm} \subseteq
U_{\alp(\gmm)} \cap U_{\alp'(\gmm)}$ and $C^{\infty}$ [resp. $C^{0}$]
functions $P_{(\gmm)} : V_{\gmm} \to \G$ so that
\begin{equation*}
  P_{(\gmm)} O_{(\alp(\gmm) \alp(\dlt))} = O'_{(\alp'(\gmm) \alp'(\dlt))} P_{(\dlt)} \quad \hbox{ on } V_{\gmm} \cap V_{\dlt}.
\end{equation*}
By the \emph{topological} or \emph{isomorphism class} of a vector
bundle $\eta$, we mean the class of all vector bundles isomorphic to
$\eta$.

The open cover $\set{U_{\alp}}$ in Definition~\ref{def:vb} provides
subsets on which $\eta$ is isomorphic to the trivial bundle $U_{\alp}
\times V$, and the transition maps $\set{O_{(\alp \bt)}}$ describe how
these local trivial bundles are patched together.  We call an
isomorphism $\eta \restriction_{U_{\alp}} \to U_{\alp} \times V$ a
\emph{local gauge} (or local trivializations), and refer to $O_{(\alp
  \bt)}$, viewed as an isomorphism between two trivial bundles
$U_{\alp} \times V$, as a \emph{local gauge transformation}. Moreover,
we use the term \emph{global gauge} for a global isomorphism from
$\eta \to X \times V$ (if it exists), and \emph{global gauge
  transformation} for a $\G$-valued function on $X$, viewed as an
isomorphism between such trivial bundles.

Let $\eta$ be a $C^{\infty}$ vector bundle with structure group $\G$,
defined by the data $\set{U_{\alp}, O_{(\alp \bt)}}$.  A \emph{section
  $s$ of $\eta$} consists of local data $s_{(\alp)}$ (the local
expression for $s$ in the local gauge on $U_{\alp}$), which are smooth
functions $s_{(\alp)} : U_{\alp} \to V$ satisfying the compatibility
condition
\begin{equation*}
  s_{(\alp)} = O_{(\alp \bt)} s_{(\bt)} \quad \hbox{ on } U_{\alp} \cap U_{\bt}.
\end{equation*}
A \emph{connection $\covD$} on $\eta$ consists of local data $\ud +
A_{(\alp)}$, where each $A_{(\alp)}$ is a smooth $\g$-valued 1-form on
$U_{\alp}$ satisfying the compatibility condition:
\begin{equation*}
  A_{(\alp)} = Ad(O_{(\alp \bt)})A_{(\bt)} - \rd O_{(\alp \bt)} O^{-1}_{(\alp \bt)} \quad \hbox{ on } U_{\alp} \cap U_{\bt}.
\end{equation*}
We call $A_{(\alp)}$ a \emph{gauge potential} for $\covD$ in the local
gauge $U_{\alp}$.

Observe that $\covD$ defines a first order differential on the space
of smooth sections of $\eta$, in the sense that $\covD (f s) = \ud f s
+ f \covD s$ for any function $f$ and any section $s$. The space of
all connections is denoted by $\calA(\eta)$. As is well-known,
$\calA(\eta)$ has the structure of an affine space, in the sense that
the difference of two connections $\covD$ and $\covD'$ is a 1-form
taking values in the adjoint bundle $ad(\eta)$ (defined with the same
data as $\eta$, but where $V = \g$ and $O_{(\alp \bt)}$ acts on $V$ on
the left by the adjoint action).

The curvature $2$-form of $\covD$ is defined by the relation
\begin{equation*}
  F[\covD](X, Y) \cdot s = \covD_{X} \covD_{Y} s - \covD_{Y} \covD_{X} s  - \covD_{[X, Y]} s
\end{equation*}
Locally, it takes the form
\begin{equation*}
  F_{(\alp)} = \ud A_{(\alp)} + \frac{1}{2} [A_{(\alp)} \wedge A_{(\alp)}] \quad \hbox{ on } U_{\alp},
\end{equation*}
and different local data are related to each other by
\begin{equation*}
  F_{(\alp)} = Ad(O_{(\alp \bt)}) F_{(\bt)} \quad \hbox{ on } U_{\alp} \cap U_{\bt}.
\end{equation*}
In other words, $F$ is an $ad(\eta)$-valued 2-form on $X$.

Finally, we introduce the notion of the associated \emph{principal
  $\bfG$-bundle}, which is the bundle with data the $\set{U_{\alp},
  O_{(\alp \bt)}}$ and with the fibers modeled on the group $\bfG$,
where the transition functions $O_{(\alp \bt)}$ act on $\bfG$ by right
multiplication. From the local viewpoint, it is simply a way to
encapsulate the data $\set{U_{\alp}, O_{(\alp \bt)}}$ without
reference to any vector space $V$. Principal bundles may serve as an
alternative starting point for developing the theory of vector bundles
(cf. Kobayashi--Nomizu \cite{MR1393940,MR1393941}).

\subsection{Global gauges and topological classes of 
\texorpdfstring{$C^{\infty}$}{C-infty}  connections} 
\label{subsec:smth} In the following few subsections,
we specialize to the cases $X = B_{R}$ (a ball of radius $R$ in
$\bbR^{d}$) or $\bbR^{d}$. Eventually, we aim to give a suitable
definition of connections at the optimal regularity, and introduce the
notion of topological classes of such connections. Before we embark on
these goals, we first review the simple case of a $C^{\infty}$
connection with a compactly supported curvature.

We start with the case $X = B_{R}$. Since $B_{R}$ is contractible, all
$C^{\infty}$ vector bundles over $B_{R}$ are trivial; more precisely,
a global gauge (or trivialization) of $\eta$ on $B_{R}$ can be
constructed by parallel transportation with respect to $\covD$ along
each ray starting from the center $x_{0}$ of $B_{R}$. We obtain a
representative $A$ of $\covD$ on $B_{R}$ such that
\begin{equation} \label{eq:goodrep-smth-ball} A \in C^{\infty}(B_{R};
  \g).
\end{equation}
Moreover, $(x - x_{0})^{j}A_{j} = 0$ by the parallel transport
condition.

Next, we consider the case $X = \bbR^{d}$. Since $\bbR^{d}$ is
contractible, too, all $C^{\infty}$ vector bundles over $\bbR^{d}$ are
trivial. However, when the vector bundles is endowed with a compactly
supported curvature, we may define their topological class by viewing
them as bundles on the compactification $\bbR^{d} \cup \set{\infty}$,
which is homeomorphic to $\bbS^{d} = \set{X \in \bbR^{d+1} : \abs{X} =
  1}$. More precisely, consider the stereographic projection
\begin{equation} \label{eq:st-proj} \boldsymbol{\Sgm} : \bbS^{d} \to
  \bbR^{d}, \quad (X^{1}, \ldots, X^{d+1}) \mapsto \left(
    \frac{X^{1}}{1 - X^{d+1}}, \ldots, \frac{X^{d}}{1 - X^{d+1}}
  \right).
\end{equation}
Note that the pullback of $(\eta, \covD)$ along $\boldsymbol{\Sgm}$,
which we denote by $(\boldsymbol{\Sgm}^{\ast} \eta,
\boldsymbol{\Sgm}^{\ast} \covD)$, obeys $F[\boldsymbol{\Sgm}^{\ast}
\covD] = 0$ on $U'_{\infty} = \set{X \in \bbS^{d} : 0 < X^{d+1} < 1} =
\boldsymbol{\Sgm}^{-1}(\bbR^{d} \setminus B_{1})$. Since $U_{\infty}'$
is simply connected, the pullback bundle $\boldsymbol{\Sgm}^{\ast}
\eta$ is isomorphic to the trivial bundle $U'_{\infty} \times V$
\cite[Corollary~9.2]{MR1393940}, which may be easily extended to
$U_{\infty} = \set{X \in \bbS^{d} : X^{d+1} > 0}$. Therefore,
$\boldsymbol{\Sgm}^{\ast} \eta$ extends to a smooth vector bundle on
$\bbS^{d}$.  The \emph{topological class} of $(\eta, \covD)$ may be
defined to be that of the extended bundle on $\bbS^{d}$.

Since $\bbS^{d}$ is covered by with two contractible open sets, namely
$U_{0} = \bbS^{d} \setminus \set{(0, \ldots, 0, 1)}$ and $U_{\infty} =
\bbS^{d} \setminus \set{(0, \ldots, 0, -1)}$, the topological class of
the bundle on $\bbS^{d}$ is determined by the transition map
in-between. At the level of $\eta$, it is the transition map $O$
between $\bbR^{d}$, on which there exists a local representative
$\covD = \ud + A$ with $A (0)=0$ and $x^{j} A_{j} = 0$ (parallel
transport along radial rays from $0$), and $\bbR^{d} \setminus B_{1}$,
on which $\covD = \ud$. On $\bbR^{d} \setminus B_{1}$, we have
\begin{equation*}
  A = - \rd_{x} O O^{-1}.
\end{equation*}
Moreover, since $x^{j} A_{j} = 0$, it follows that $x^{j} \rd_{j} O =
0$ on $\bbR^{d} \setminus B_{1}$, i.e., $O(x) = O(\frac{x}{\abs{x}})$
for $\abs{x} \geq 1$. Defining $O_{(\infty)} : \bbR^{d} \setminus
\set{0} \to \G$, $O_{(\infty)}(x) = O(\frac{x}{\abs{x}})$ and
introducing a smooth function $\chi$ such that $1-\chi$ is compactly
supported, we arrive at:
\begin{theorem} \label{thm:goodrep-smth} Let $\covD$ be a $C^{\infty}$
  connection on a $C^{\infty}$ vector bundle $\eta$ on $\bbR^{d}$,
  whose curvature is compactly supported. Then there exists a global
  gauge for $\eta$ in which the global gauge potential $A = \covD -
  \ud$ admits a decomposition of the form
  \begin{equation} \label{eq:goodrep-smth} A = - \chi O_{(\infty); x}
    + B
  \end{equation}
  where $O_{(\infty)} (x)$ is a smooth $0$-homogeneous map into $\G$
  and $B \in C^{\infty}_{c}(\bbR^{d}; \g)$.
\end{theorem}

It is not difficult to see that $O_{(\infty)}$, which we call a
\emph{gauge at infinity for $A$}, is defined uniquely up to homotopy
(cf. Proposition~\ref{prop:goodrep-homotopy-0}). The \emph{homotopy
  class $[O_{(\infty)}]$}, which is defined intrinsically without
reference to the pullback procedure, \emph{determines the topological
  class}\footnote{Strictly speaking, $O_{(\infty)}$ in
  Theorem~\ref{thm:goodrep-smth} directly determines only the smooth
  isomorphism class, which in turn determines the topological (i.e.,
  $C^{0}$) isomorphism class by a density argument.} \emph{of the
  extended pullback bundle on $\bbS^{d}$}. Hence, any topological
invariants of the extended pullback bundle depend only on
$[O_{(\infty)}]$.

  \emph{Characteristic classes} are important invariants of a vector
  (or principal) $\G$-bundle. On $\bbS^{d}$, by the Chern--Weil theory
  \cite[Chapter~XII]{MR1393941}, these may be defined in terms of a
  connection $\covD$ as follows. Given any symmetric $Ad$-invariant
  $k$-linear function $f$ on $\g$, we call the $2k$-form
  \begin{equation*}
    f(F[\covD], \ldots, F[\covD])  = f(F_{j_{1} j_{2}}, \ldots, F_{j_{d-1} j_{d}}) \ud x^{j_{1}} \wedge \ud x^{j_{2}} \wedge \cdots \wedge \ud x^{j_{d}}
  \end{equation*}
  the \emph{characteristic class} associated to $f$. This $2k$-form is
  closed and is invariant, up to an exact form, in the choice of a
  connection $\covD$ on the bundle; hence it defines a cohomology
  class in $H^{2k}(\bbS^{d})$, which depends only on the isomorphism
  class of the bundle. Moreover, when $d = 2k$, the integral
  \begin{equation*}
    \ch_{f} = \int_{\bbS^{d}} f(F[\covD], \ldots, F[\covD]),
  \end{equation*}
  called the \emph{characteristic number}, is also an invariant of the
  bundle.

  Now, as an application of Theorem~\ref{thm:goodrep-smth}, consider a
  $C^{\infty}$ connection $\covD$ on $\bbR^{d}$ with compactly
  supported curvature. Then $\ch_{f}$ of the pullback bundle equals
  \begin{equation} \label{eq:ch-no-smth} \ch_{f} = \int_{\bbR^{d}}
    f(F[\covD], \ldots, F[\covD]),
  \end{equation}
  and depends only on $[O_{(\infty)}]$ in
  Theorem~\ref{thm:goodrep-smth}.

  An important special case of the above theory is when $d = 4$ and
  $\G = SU(2)$, and we take $f(A, B) = \frac{1}{8 \pi^{2}} \tr(AB)$. The
  corresponding characteristic number, given by the integral formula
  \begin{equation*}
    c_{2} = \frac{1}{8 \pi^{2}} \int_{\bbR^{4}} \tr(F \wedge F),
  \end{equation*}
  is called the \emph{second Chern number}. It is always an integer,
  and it classifies the topological classes of $SU(2)$-bundles. For
  more on characteristic classes, we refer the reader to
  \cite{MR0440554}.

\subsection{Global gauges for rough
\texorpdfstring{  $\G$}{G}-bundles} \label{subsec:rough-conn} We are now ready to describe
our first set of results. Motivated by the desire to study the
hyperbolic Yang--Mills equation (cf. Section~\ref{subsec:ym}) at the
optimal scaling-invariant regularity, our aim here is to sharpen
\eqref{eq:goodrep-smth-ball} and \eqref{eq:goodrep-smth} in two ways:
\begin{enumerate}
\item To obtain quantitative bounds for $A$ in a ``good global gauge''
  in terms of $F$;
\item To relax the condition for $F$ to the scaling-invariant
  condition $F \in L^{\frac{d}{2}}(X)$.
\end{enumerate}
In what follows, we restrict to $d \geq 3$ (which, for instance,
avoids the case $L^{\frac{d}{2}} = L^{1}$).

To set up the scene, we start with the definition of connections with
$L^{\frac{d}{2}}_{loc}$ curvature. Let $X$ be an open subset of
$\bbR^{d}$. For $k \in \bbR$ and $p \in [1, \infty]$, we introduce
\begin{align}
  \calG^{k, p}_{loc}(X) = & \set{O \in W^{k, p}_{loc}(X; \bbR^{N
      \times N}) : O (x) \in \bfG \hbox{ for a.e. } x \in X}.
\end{align}
The relevant regularity class is $\calG^{2, \frac{d}{2}}_{loc}$, which
turns to be closed under multiplication and inverse (see
Lemmas~\ref{lem:mult}, \ref{lem:inv} and \ref{lem:d-inv} below). In
parallel to Section~\ref{subsec:vb}, we define \emph{a $\calG^{2,
    \frac{d}{2}}_{loc}$ (principal) $\G$-bundle on $X \subseteq
  \bbR^{d}$} by the data:
\begin{itemize}
\item An open cover $\set{U_{\alp}}$ of $X$;
\item A transition function $O_{(\alp \bt)} \in \calG_{loc}^{2,
    \frac{d}{2}} (U_{\alp} \cap U_{\bt})$ for every $\alp, \bt$,
  obeying the \emph{cocycle conditions}:
  \begin{enumerate}
  \item $O_{(\alp \alp)} = id$ on each $U_{\alp}$;
  \item $O_{(\alp \bt)} \cdot O_{(\bt \gmm)} = O_{(\alp \gmm)}$ on
    each $U_{\alp} \cap U_{\bt} \cap U_{\gmm}$.
  \end{enumerate}
\end{itemize}
An open cover $\set{V_{\gmm}}$ is a \emph{refinement} of
$\set{U_{\alp}}$ if there exists a function $\alp = \alp(\gmm)$ such
that $V_{\gmm} \subseteq U_{\alp(\gmm)}$. We say that two data sets
$\set{U_{\alp}, O_{(\alp \bt)}}$ and $\set{U'_{\alp'}, O'_{(\alp'
    \bt')}}$ define an \emph{equivalent $\calG^{2, \frac{d}{2}}_{loc}$
  bundle} if there exists a common refinement $V_{\gmm}$ of the open
covers and $P_{(\gmm)} \in \calG^{2, \frac{d}{2}}_{loc}(V_{\gmm})$
such that
\begin{equation*}
  P_{(\dlt)} \cdot O_{(\alp(\dlt) \alp(\gmm))} = O'_{(\alp'(\dlt) \alp'(\gmm))} \cdot P_{(\gmm)} \quad \hbox{ on } V_{\gmm} \cap V_{\dlt}.
\end{equation*}
A \emph{$W^{1, \frac{d}{2}}_{loc}$ connection} $\covD$ on the bundle
defined by $\set{U_{\alp}, O_{(\alp \bt)}}$ is given by the local
data:
\begin{itemize}
\item A 1-form $A_{(\alp)} \in W^{1, \frac{d}{2}}_{loc}(U_{\alp}; \g)$
  for each $\alp$, called the \emph{local representative} of $\covD$
  on $U_{\alp}$, satisfying the compatibility condition
  \begin{equation*}
    A_{(\alp)} = Ad (O_{(\alp \bt)}) A_{(\bt)} - O_{(\alp \bt); x} \quad \hbox{ on each } U_{\alp} \cap U_{\bt}.
  \end{equation*}
\end{itemize}
Given a $W^{1, \frac{d}{2}}_{loc}$ connection $\covD$, we define its
\emph{curvature 2-form} $F = F[\covD]$ by the local data:
\begin{equation*}
  F_{(\alp)} = \ud A_{(\alp)} + \frac{1}{2} [A_{(\alp)} \wedge A_{(\alp)}] \quad \hbox{ on each } U_{\alp}.
\end{equation*}
We denote by $\calA^{1, \frac{d}{2}}_{loc}(X)$ the space of all $W^{1,
  \frac{d}{2}}_{loc}$ connections on all $\calG^{2,
  \frac{d}{2}}_{loc}$ bundles on $X$. By the compatibility property of
$F_{(\alp)}$ (algebraically the same as  in the smooth case), note that
\begin{equation*}
  \abs{F} = \abs{F_{(\alp)}} = \sqrt{\brk{F_{(\alp)}, F_{(\alp)}}} \quad \hbox{ on each } U_{\alp}
\end{equation*}
is a well-defined element of $L^{\frac{d}{2}}_{loc}(X)$.

Consider the case $X = B_{R}$. In order to state quantitative bounds
for the gauge potential in a ``good gauge'', we introduce the
\emph{inner ($L^{\frac{d}{2}}$-)concentration scale} with threshold
$\eps_{\ast}$ of a connection $\covD$, defined as follows:
\begin{equation*}
  \rcsp^{\eps_{\ast}} [\covD] = \sup \set{r > 0 : \nrm{F[\covD]}_{L^{\frac{d}{2}}(B_{r}(x) \cap X)} \leq \eps_{\ast} \ \hbox{ for all } x \in X}.
\end{equation*}

\begin{theorem} [Good gauge on a ball] \label{thm:goodrep-ball} Let
  $\covD \in \calA^{1, \frac{d}{2}}_{loc}(B_{R})$ satisfy $F[\covD]
  \in L^{\frac{d}{2}}(B_{R})$ and $\rcsp^{\eps_{\ast}} [\covD] \geq r$,
  for some $r > 0$ and a sufficiently small $\eps_{\ast} >0$. Then
  there exists a global gauge in which the gauge potential $A$ for
  $\covD$ satisfies
  \begin{equation} \label{eq:goodrep-ball-est} \nrm{A}_{\dot{W}^{1,
        \frac{d}{2}}(B_{R})} \aleq_{\eps_{\ast}, \frac{R}{r}} 1.
  \end{equation}
  If, in addition, $\covD^{(n)} F \in L^{p}(B_{R})$ for some
  nonnegative integer $n$ and $p \in (1, \infty)$ such that $p \geq
  \frac{d}{n+2}$, then $A \in W^{n+1, p}(B_{R})$.
\end{theorem}
Theorem~\ref{thm:goodrep-ball} tells us that given any connection on a
ball with $L^{\frac{d}{2}}$-curvature, there exists a good gauge in
which the a-priori bound \eqref{eq:goodrep-ball-est} holds. When
$\nrm{F[\covD]}_{L^{\frac{d}{2}}(B_{R})}$ is sufficiently small (with
the threshold depending on $d$), Theorem~\ref{thm:goodrep-ball} is the
classical result of Uhlenbeck \cite{MR648356}. The general case is
proved by appropriately patching up local applications of Uhlenbeck's
lemma.

Next, we consider the case $X = \bbR^{d}$. To proceed, we need an
additional concept. We define the \emph{outer
  ($L^{\frac{d}{2}}$-)concentration radius} with threshold
$\eps_{\ast}$ of a connection $\covD$ to be
\begin{equation*}
  \Rcsp^{\eps_{\ast}} [\covD] = \inf \set{r > 0 : \nrm{F[\covD]}_{L^{\frac{d}{2}}(\bbR^{d} \setminus  B_{r}(x))} \leq \eps_{\ast} \ \hbox{ for some } x \in \bbR^{d}}.
\end{equation*}
Let $1- \chi \in C^{\infty}_{c}(\bbR^{d})$ be fixed.
\begin{theorem} [Good global gauge on $\bbR^{d}$] \label{thm:goodrep}
  Let $\covD \in \calA^{1, \frac{d}{2}}_{loc}(\bbR^{d})$ satisfy
  $F[\covD] \in L^{\frac{d}{2}}(\bbR^{d})$, as well as
  $\rcsp^{\eps_{\ast}}[\covD] \geq r$ and $\Rcsp^{\eps_\ast}[\covD]
  \leq R$ for some $0 < r \leq R$ and a universal small constant
  $\eps_{\ast} > 0$. Then there exists exists a global gauge on
  $\bbR^{d}$, in which the gauge potential $A \in \dot{W}^{1,
    \frac{d}{2}}_{loc}(\bbR^{d})$ for $\covD$ admits a decomposition
  of the form
  \begin{equation} \label{eq:goodrep-0} A = - \chi (\cdot / R)
    O_{(\infty); x} + B
  \end{equation}
  where $O_{(\infty)} (x)$ is a smooth $0$-homogeneous map into $\G$
  and $B \in \dot{W}^{1, \frac{d}{2}}(\bbR^{d}; \g)$. Moreover,
  \begin{equation} \label{eq:goodrep-bnd} \nrm{B}_{\dot{W}^{1,
        \frac{d}{2}}} \aleq_{\eps_{\ast}, \frac{R}{r}} 1, \qquad
    \nrm{O_{(\infty)}}_{C^{N}(\bbS^{d-1})} \aleq_{\eps_{\ast},
      \frac{R}{r}, N} 1 \quad \hbox{ for all } N \geq 0.
  \end{equation}
  If, in addition, $\covD^{(n)} F \in L^{p}(B_{R})$ for some
  nonnegative integer $n$ and $p \in (1, \infty)$ such that $p \geq
  \frac{d}{n+2}$, then $B \in \dot{W}^{n+1, p}(\bbR^{d})$.
\end{theorem}
Thanks to Theorems~\ref{thm:goodrep-ball} and \ref{thm:goodrep}, we
may identify any connection $\covD \in \calA^{1, \frac{d}{2}}(X)$ with
a gauge potential $A \in W^{1, \frac{d}{2}}_{loc}(X)$ in a good global
gauge. In the rest of the introduction, we adopt the convention of
referring to a connection $\covD$ on $B_{R}$ or $\bbR^{d}$ by its
global gauge potential $A$.

\subsection{Topological classes of rough
  connections} \label{subsec:top-class} Given a $W^{1,
  \frac{d}{2}}_{loc}$ connection $A$ on $\bbR^{d}$, we call a pair
$(O_{(\infty)}, B)$ of a smooth $0$-homogeneous map into $\G$ and an
element in $\dot{W}^{1, \frac{d}{2}}(\bbR^{d}; \g)$ a \emph{good
  representative of $A$} if $A = - \chi O_{(\infty); x} + B$ for some
$1-\chi \in C^{\infty}_{c}(\bbR^{d})$. We furthermore call
  $O_{(\infty)}$ a \emph{gauge (transformation) at infinity for $A$}.
Theorem~\ref{thm:goodrep} insures that a good representative always
exists provided that $F[A] \in L^{\frac{d}{2}}$.

Recall that when the curvature is smooth and compactly supported, the
topological class of $A$ is classified by the homotopy class of its
gauge at infinity $O_{(\infty)}$. We extend the definition of
the topological class to a rough connections on $\bbR^{d}$ with
$L^{\frac{d}{2}}$-curvature using this classification. We need the
following preliminary results:

\begin{proposition} \label{prop:goodrep-homotopy-0} Let $A \in
  \calA_{loc}^{1, \frac{d}{2}}(\bbR^{d})$ satisfy $F[A] \in
  L^{\frac{d}{2}}(\bbR^{d})$, and let $(O_{(\infty)}, B)$ be a good
  representative of $A$.
  \begin{enumerate}
  \item If $(O'_{(\infty)}, B')$ is another good representation of
    $A$, then $O_{(\infty)}$ is homotopic to $O'_{(\infty)}$.
  \item Conversely, given any smooth $O'_{(\infty)} : \bbS^{d-1} \to
    \G$ homotopic to $O_{(\infty)}$, there exists another good
    representation $(O'_{(\infty)}, B')$ of $A$.
  \end{enumerate}
\end{proposition}

\begin{remark} \label{rem:chi-indep-0} For completeness, we make the
  trivial observation that the homotopy class of $O_{(\infty)}$ is
  independent of the choice of $\chi$, too.
\end{remark}

Theorem~\ref{thm:goodrep}, Proposition~\ref{prop:goodrep-homotopy-0}
and Remark~\ref{rem:chi-indep-0} lead to the following:
\begin{definition} \label{def:top-class} Given an
  $L^{\frac{d}{2}}$-curvature connection $A$, we define the
  \emph{topological class} $[A]$ of $A$ to be the homotopy class of
  $O_{(\infty)} : \bbS^{d-1} \to \G$ of a good representative
  ({i.e., a gauge at infinity for $A$}). If the topological class
  of $A'$ is $[A]$, then we write $A' \in [A]$.
\end{definition}

Observe that the addition of a 1-form $B$ in $\dot{W}^{1,
  \frac{d}{2}}(\bbR^{d}; \g)$ does not change the topological class of
$A$, i.e.,
\begin{equation*}
  A + B \in [A].
\end{equation*}
In particular, by mollifying and cutting off $B$, we can easily find
approximations by smooth connections with compactly supported
curvature in the same topological class with respect to the distance
$d_{\dot{W}^{1, \frac{d}{2}}}(A, A') = \nrm{A - A'}_{\dot{W}^{1,
    \frac{d}{2}}(\bbR^{d}; \g)}$. Moreover, good representations of
two connections with the same $O_{(\infty)}$ are path-connected with
respect to the $d_{\dot{W}^{1, \frac{d}{2}}}$. By
Proposition~\ref{prop:goodrep-homotopy-0}, it follows that each
topological class is path-connected with respect to $d_{\dot{W}^{1,
    \frac{d}{2}}}$ up to global gauge transformations in $\calG^{2,
  \frac{d}{2}}_{loc}(\bbR^{d})$.

Observe also that topological class is determined by the part of the
connection where the $L^{\frac{d}{2}}$ norm of $F$ is
concentrated. More precisely, we have:
\begin{proposition} \label{prop:top-class-outer} Let $A, A' \in
  \calA^{1, \frac{d}{2}}_{loc}(\bbR^{d})$ satisfy $F[A], F[A'] \in
  L^{\frac{d}{2}}(\bbR^{d})$. Assume moreover that $A$ and $A'$ are close
  in $L^{d}(B_{5R})$, and have small $L^{\frac{d}{2}}$ curvature
  outside $B_{R}$, i.e.,
  \begin{equation*}
    \nrm{A - A'}_{L^{d}(B_{5R})} \leq \eps_{\ast}, \quad
    \nrm{F[A]}_{L^{\frac{d}{2}}(\bbR^{d} \setminus B_{R})} \leq \eps_{\ast}, \quad
    \nrm{F[A']}_{L^{\frac{d}{2}}(\bbR^{d} \setminus B_{R})} \leq \eps_{\ast},
  \end{equation*}
  where $\eps_{\ast} > 0$ is sufficiently small universal
  constant. Then $[A] = [A']$.
\end{proposition}

We now discuss some simple consequences of the above
results. Given an $L^{\frac{d}{2}}$-curvature connection $A$, let
  $A^{n}$ be an approximation of $A$ in $d_{\dot{W}^{1,
      \frac{d}{2}}}$, such that each $A^{n}$ is smooth and $F[A^{n}]$
  is compactly supported. For any symmetric $Ad$-invariant $k$-linear
  function $f$ on $\g$, the associated characteristic classes of the
  pullback bundles $(\boldsymbol{\Sgm}^{\ast} \eta,
  \boldsymbol{\Sgm}^{\ast} A^{n})$ are independent of $n$ (for
  sufficiently large $n$), as well as of the approximating
  sequence. Moreover, when $d = 2k$, the characteristic numbers obey
  \begin{equation*}
    \ch_{f} = \int_{\bbR^{d}} f(F[A^{n}], \ldots, F[A^{n}]) \to \int_{\bbR^{d}} f(F[A], \ldots, F[A])
  \end{equation*}
  by continuity of the integral with respect to $\nrm{A -
  A'}_{\dot{W}^{1, \frac{d}{2}}(\bbR^{d}; \g)}$. Hence we recover the
following result of Uhlenbeck \cite{MR815194}:
\begin{corollary}\label{cor:ch-class}
  The characteristic numbers $\ch_{f}$, defined as in
  \eqref{eq:ch-no-smth}, depend only on $[A]$.  In particular,
    they vanish for $[0]$.
\end{corollary}

As another corollary of Theorem~\ref{thm:goodrep}, we obtain a
characterization of the topologically trivial class (i.e., the
topological class of the trivial connection $A = 0$):
\begin{corollary} \label{cor:top-triv} The space of topologically
  trivial connections with finite $L^{\frac{d}{2}}$ curvature
  correspond exactly to
  \begin{equation*}
    \calA^{1, \frac{d}{2}}_{0}(\bbR^{d}) = \set{\covD = \ud + A : A \in \dot{W}^{1, \frac{d}{2}}(\bbR^{d}; \g)}.
  \end{equation*}
  All characteristic numbers associated to a connection $A$ in
    $\calA^{1, \frac{d}{2}}_{0}(\bbR^{d})$ vanish.
\end{corollary}
\begin{remark}
  The preceding corollary implies that given any connection $A$ in the
  topologically trivial class, there exists a global representative
  $\tA$ in the space $\dot{W}^{1, \frac{d}{2}}(\bbR^{d}; \g)$. Note,
  however, that no quantitative bound on $\nrm{\tA}_{\dot{W}^{1,
      \frac{d}{2}}}$ is claimed; such a bound would rely on
  quantitative bounds on a homotopy of $O_{(\infty)}$ to the identity
  in terms of scaling-invariant bounds on $O_{(\infty)}$.
\end{remark}

\subsection{Hyperbolic Yang--Mills equation} \label{subsec:ym} The
remainder of the introduction concerns the hyperbolic Yang--Mills
equation. The purpose of this subsection is to provide a brief
introduction to this equation.

Let $\bbR^{1+d}$ denote the $(d+1)$-dimensional Minkowski space, which
is equipped with the Minkowski metric $\bfm_{\mu \nu} =
\mathrm{diag}(-1, +1, \ldots, +1)$ in the rectangular coordinates
$(x^{0}, x^{1}, \ldots, x^{d})$.  We will often write $t = x^{0}$, to
emphasize the role of $x^{0}$ as (a choice of) a time function.
Throughout this paper, we will use the usual convention of raising and
lowering indices using the Minkowski metric, as well as summing up
repeated upper and lower indices.

Consider a connection $\covD$ on a vector bundle on $\bbR^{1+d}$ with
structure group $\G$. By topological triviality of $\bbR^{d}$ (or
Theorem~\ref{thm:goodrep} at low regularity), $\covD$ at each $t$ may
be identified with a global gauge potential $A$. The \emph{hyperbolic
  Yang--Mills equation} on $\bbR^{1+d}$ for $A$ is the Euler--Lagrange
equation associated with the formal Lagrangian action functional
\begin{equation*}
  \calL(A) = \frac{1}{2} \int_{\bbR^{1+d}} \brk{F_{\alp \bt}, F^{\alp \bt}} \, \ud x \ud t,
\end{equation*}
which takes the form
\begin{equation} \label{eq:ym} \covD^{\alp} F_{\alp \bt} = 0.
\end{equation}
Clearly, \eqref{eq:ym} is invariant under (smooth) gauge
transformations. This equation possesses a conserved energy, given by
\begin{equation*}
  \calE_{\set{t} \times \bbR^{d}}(A) = \int_{\set{t} \times \bbR^{d}} \sum_{\alp < \bt} \abs{F_{\alp \bt}}^{2} \, \ud x.
\end{equation*}
Furthermore, \eqref{eq:ym} is invariant under the scaling
\begin{equation*}
  A(t, x) \mapsto \lmb A (\lmb t, \lmb x) \qquad (\lmb > 0).
\end{equation*}
The scaling-invariant $L^{2}$-Sobolev norm is $\nrm{A(t,
  \cdot)}_{\dot{H}^{\frac{d-2}{2}}}$. In particular, \eqref{eq:ym} is
\emph{energy critical} when $d = 4$, in the sense that the conserved
energy (which scales like $\nrm{A(t, \cdot)}_{\dot{H}^{1}}$) is
invariant under the scaling.

We are interested in the initial value problem for \eqref{eq:ym} at
the scaling-invariant $L^{2}$-Sobolev regularity.
For this purpose we first formulate a gauge-covariant notion of
initial data sets. We say that a pair $(a, e)$ of a gauge potential
$a$ and a $\g$-valued 1-form $e$ on $\bbR^{d}$ is an initial data set
for a solution $A$ to \eqref{eq:ym} if
\begin{equation*}
  (A_{j}, F_{0j}) \restriction_{\set{t = 0}} = (a_{j}, e_{j}).
\end{equation*}
Here and throughout this paper, the roman letters stand for the
spatial coordinates $x^{1}, \ldots, x^{d}$. Note that \eqref{eq:ym}
with $\bt = 0$ imposes the condition that
\begin{equation} \label{eq:YMconstraint} \covD^{j} e_{j} = \rd^{j}
  e_{j} + [a^{j}, e_{j}] = 0.
\end{equation}
This equation is the \emph{Gauss} (or the \emph{constraint})
\emph{equation} for \eqref{eq:ym}.

It turns out that \eqref{eq:YMconstraint} characterizes precisely
those pairs $(a, e)$ which can arise as an initial data set. Thus we
make the following definition:
\begin{definition} \label{def:ym-id} An $\calH^{\sgm}(\calO)$
  (resp. $\dot{\calH}^{\sgm}(\calO)$ or $\calH^{\sgm}_{loc}(\calO)$)
  \emph{initial data set} for the Yang-Mills equation is a pair $(a,e)
  \in H^{\sgm} \times H^{\sgm-1}(\calO)$ (resp. $\dot{H}^{\sgm} \times
  \dot{H}^{\sgm-1}(\calO)$ or $H^{\sgm}_{loc} \times
  H^{\sgm-1}_{loc}(\calO)$) that satisfies the constraint equation
  \eqref{eq:YMconstraint}.
\end{definition}

Due to invariance under gauge transformations, \eqref{eq:ym} is not
even formally well-posed when viewed as a PDE for $A$. In order to
analyze \eqref{eq:ym} at the level of $A$, this invariance must be
removed by fixing a representative (or a gauge). A simple and useful
way is to require that
\begin{equation} \label{eq:temporal} A_{0} = 0.
\end{equation}
The gauge thus chosen is called \emph{temporal}. In this gauge,
\eqref{eq:ym} becomes a coupled system of wave and transport equations
for the curl and divergence of $A$, respectively, and local
well-posedness for regular data is easily follows. Moreover, in the
regular case it is also easy to verify the finite speed of propagation
property, in the sense that $A$ vanishes on the domain of dependence
of the zero-set of the data.

The aforementioned coupled wave-transport system in the temporal gauge
becomes difficult to analyze in the low regularity
setting. Nonetheless, in \cite{OTYM2}, global well-posedness of
\eqref{eq:ym} under \eqref{eq:temporal} was proved for small data at
the optimal $L^{2}$-Sobolev regularity (for dimensions $d \geq 4$), by
first working in a gauge with more favorable structure (caloric
gauge), and then estimating the gauge transformation to the temporal
gauge.

At this point, one may imagine upgrading the small data result to
large data local well-posedness by the following procedure:
\begin{enumerate}
\item Constructing local-in-spacetime solutions from the small data
  result applied to suitable localizations of the initial data;
\item Patch the local-in-spacetime solutions together by finite speed
  of propagation.
\end{enumerate}
Though this strategy eventually works (see Section~\ref{subsec:local}
below), this is not trivial. The primary reason is because the Gauss
equation \eqref{eq:YMconstraint} is nonlocal, and thus initial data
sets cannot be freely cut off. The next subsection is devoted to
resolving this issue.

\subsection{Excision and extension of Yang--Mills initial
  data} \label{subsec:excise} In this subsection we present the second
set of results of this paper, which eventually lead to a useful
excision-and-extension technique for Yang--Mills initial data. The
first and main result is solvability of the inhomogeneous Gauss
equation
\begin{equation} \label{eq:gauss-inhom} (\covD^{(a)})^{\ell} e_{\ell}
  = h
\end{equation}
while keeping good physical space support properties.
 
\begin{theorem} \label{thm:gauss-0} Let $d \geq 4$ and $a \in
  \dot{H}^{\frac{d-2}{2}}(\bbR^{d})$. Given any convex open set $K$,
  there exists a solution operator $T_{a}$ for \eqref{eq:gauss-inhom}
  satisfying the following conditions:
  \begin{enumerate}
  \item (Boundedness) We have
    \begin{equation} \label{eq:gauss-0-bnd}
      \nrm{T_{a}[h]}_{\dot{H}^{\frac{d-4}{2}}}
      \aleq_{\nrm{a}_{\dot{H}^{\frac{d-2}{2}}}, L(K)}
      \nrm{h}_{\dot{H}^{\frac{d-6}{2}}},
    \end{equation}
    where $L(K)$ is a scaling-invariant quantity (i.e., $L(\lmb K)$ is
    independent of $\lmb > 0$) defined in \eqref{eq:lip-K}.
  \item (Exterior support property) If $h$ is supported outside the
    set
    \begin{equation*}
      \lmb K = \set{\lmb (x - x_{K}) \in \bbR^{d} : \hbox{$x_{K}$ is the barycenter of $K$}}
    \end{equation*}
    for some $\lmb > 0$, then so is $T_{a}[h]$.
  \item (Higher regularity) If $h$ and $a$ are smooth, so is
    $T_{a}[h]$.
  \end{enumerate}
\end{theorem}
\begin{remark} \label{rem:gauss-low-d} In $d \leq 3$, our proof does
  not apply at the critical regularity $e \in
  \dot{H}^{\frac{d-4}{2}}$, since the possible error of
  \eqref{eq:YMconstraint} belongs only to the ill-behaved space
  $\dot{H}^{-\frac{3}{2}}$. However, under an extra smallness
  assumption for $\nrm{a}_{\dot{H}^{\frac{d-2}{2}}}$, the conclusion
  of Theorem~\ref{thm:gauss-0} holds for $h \in \dot{H}^{\sgm-1}$ and
  $e \in \dot{H}^{\sgm}$ for the subcritical regularities $\sgm > 1-
  \frac{d}{2}$; see Proposition~\ref{prop:gauss-small} below.
\end{remark}

As a consequence of Theorem~\ref{thm:gauss-0}, we have the following
extension result for the Yang--Mills initial data sets.
\begin{theorem} \label{thm:ext-id} For $d \geq 4$, let $K$ be a convex
  domain in $\bbR^{d}$, and let $(a, e)$ be an $\calH^{\frac{d-2}{2}}$
  Yang--Mills initial data set on $2K \setminus \overline{K}$. Then
  there exists an $\calH^{\frac{d-2}{2}}$ Yang--Mills initial data set
  $(\ba, \be)$ on $\bbR^{d} \setminus \overline{K}$ that coincides
  with $(a, e)$ on $2K \setminus \overline{K}$ and obeys
  \begin{align}
    \nrm{\ba}_{\dot{H}^{\frac{d-2}{2}}(\bbR^{d} \setminus
      \overline{K})} & \aleq_{L(K)}
    \nrm{a}_{\dot{H}^{\frac{d-2}{2}}(2K \setminus \overline{K})}, \label{eq:ext-id-a} \\
    \nrm{\be}_{\dot{H}^{\frac{d-4}{2}}(\bbR^{d} \setminus
      \overline{K})} & \aleq_{\nrm{a}_{\dot{H}^{\frac{d-2}{2}}(2K
        \setminus \overline{K})}, L(K)}
    \nrm{e}_{\dot{H}^{\frac{d-4}{2}}(2K \setminus
      \overline{K})}. \label{eq:ext-id-e}
  \end{align}
  It can be arranged so that the association $(a, e) \mapsto (\ba,
  \be)$ is equivariant under constant gauge transformations, i.e.,
  $(Ad(O) a, Ad(O) e) \mapsto (Ad(O) \ba, Ad(O) \be))$ for each $O \in
  \G$. Moreover, if $(a, e)$ is smooth, then so is $(\ba, \be)$.
\end{theorem}

At this point, it is useful to introduce a suitable generalization of
local energy for initial data sets at the optimal $L^{2}$-Sobolev
regularity. For $d \geq 4$ even, we make a gauge-invariant definition
\begin{equation*}
  \calE^{\frac{d-2}{2}}_{U}(a, e) = \nrm{(\covD^{(a)})^{(\frac{d-2}{2})} (F[a], e)}_{L^{2}(U)}^{2} + \nrm{(F[a], e)}_{L^{\frac{d}{2}}(U)}^{2}.
\end{equation*}
Note that this is equivalent to the energy when $d = 4$. For $d \geq
4$ odd, there is a nuisance that the optimal $L^{2}$-Sobolev
regularity involves a fractional derivative. Here, we take an easy way
out, and make a gauge-dependent definition in this case:
\begin{equation*}
  \calE^{\frac{d-2}{2}}_{U}(a, e) = \nrm{(a, e)}_{\dot{H}^{\frac{d-2}{2}} \times \dot{H}^{\frac{d-4}{2}}(U)}^{2}.
\end{equation*}

Let $\eps_{\ast} > 0$. For $X = B_{R}$ or $\bbR^{d}$, we define the
notion of the (inner) critical $L^{2}$-Sobolev concentration scale
with threshold $\eps_{\ast}$ as follows:
\begin{align}
  \rc^{\eps_{\ast}} =& \rc^{\eps_{\ast}}[a, e] = \sup \set{r > 0:
    \calE^{\frac{d-2}{2}}_{X \cap B_{r}(x)}(a, e) \leq \eps_{\ast}^{2}
    \hbox{ for all } x \in X}, \label{eq:conc-scale-id}
\end{align}
When $d = 4$, we call $\rc^{\eps_{\ast}}$ the \emph{energy
  concentration scale} with threshold $\eps_{\ast}$.

Combining Theorem~\ref{thm:ext-id} with Uhlenbeck's lemma, we also
obtain the following excision-and-extension result.

\begin{theorem} \label{thm:excise} Let $(a, e)$ be an
  $\calH^{\frac{d-2}{2}}_{loc}$ Yang--Mills initial data set on $X =
  B_{R}$ (resp. $X = \bbR^{d}$) with critical $L^{2}$-Sobolev
  concentration scale (with threshold $\eps_{\ast}$) at most
  $\rc$. Consider a ball $B_{r}(x)$ with radius $r < 10 \rc$ and
  $x \in X$. For $\eps_{\ast} > 0$ sufficiently small (as a universal
  constant), the following statements hold.
  \begin{enumerate}
  \item To $(a, e)$, we associate $(\ta, \te, O) \in
    \calH^{\frac{d-2}{2}}(\bbR^{d}) \times
    \calG^{\frac{d}{2}}(B_{r}(x) \cap X)$ such that $(\ta, \te)$ is
    gauge equivalent to $(a, e)$ on $B_{r}(x) \cap X$, i.e.,
    \begin{equation*}
      (\ta, \te) = (Ad(O) a - O_{;x}, Ad(O) e) \quad \hbox{ in } B_{r}(x) \cap X.
    \end{equation*}
    Moreover, $(\ta, \te)$ and $O$ obey the bounds
    \begin{align}
      \nrm{(\ta, \te)}_{\dot{H}^{\frac{d-2}{2}} \times
        \dot{H}^{\frac{d-4}{2}}}^{2}
      + r^{-(d-2)} \nrm{\ta}_{L^{2}}^{2} + r^{-(d-4)} \nrm{\te}_{L^{2}}^{2} \aleq & \calE^{\frac{d-2}{2}}_{B_{r}(x) \cap X}(a, e), \label{eq:excise-a} \\
      \nrm{O_{;x}}_{\dot{H}^{\frac{d-2}{2}}(B_{r}(x) \cap X)} \aleq &
      \nrm{a}_{\dot{H}^{\frac{d-2}{2}}(B_{r}(x) \cap
        X)}. \label{eq:excise-O}
    \end{align}
    When $d$ is odd, $O$ is a constant gauge transformation. If $(a,
    e)$ is smooth, then so are $(\ta, \te)$ and $O$.

  \item Let $\set{(a^{n}, e^{n})}$ be a sequence of
    $\calH^{\frac{d-2}{2}}$ Yang--Mills initial data sets on $B_{r}(x)
    \cap X$ such that $(a^{n}, e^{n}) \to (a, e)$ in
    $H^{\frac{d-2}{2}} \times H^{\frac{d-4}{2}}(B_{r}(x) \cap X)$. Let
    $(\ta^{n}, \te^{n}, O^{n})$ be given\footnote{Note that the
      hypothesis on the critical $L^{2}$-Sobolev concentration scale
      is satisfied for large enough $n$.} by (1) from $(a^{n},
    e^{n})$. Then after passing to a subsequence and suitably
    conjugating each $(\ta^{n}, \te^{n}, O^{n})$ with a constant gauge
    transformation, we have
    \begin{align*}
      (\ta^{n}, \te^{n}) \to (\ta, \te) \hbox{ in } H^{\frac{d-2}{2}}
      \times H^{\frac{d-4}{2}}(\bbR^{d}), \qquad O^{n} \to O \hbox{ in
      } H^{\frac{d}{2}}(B_{r}(x) \cap X).
    \end{align*}
  \end{enumerate}
\end{theorem}
\begin{remark}
  Theorems~\ref{thm:ext-id} and \ref{thm:excise} have a similar flavor
  to the so-called \emph{initial data gluing} procedure in general
  relativity \cite{ChDe, Co, CoSch}, which is a method to remove an
  error in the constraint equation while keeping physical space
  localization properties. See \cite{OT1} for an adaptation of this
  procedure for the Maxwell--Klein--Gordon constraint equation at the
  critical regularity, which had a similar role as
  Theorems~\ref{thm:ext-id} and \ref{thm:excise} in the present
  paper. We also note that an initial data extension theorem,
  analogous to Theorem~\ref{thm:ext-id}, was recently proved for the
  vacuum Einstein equation at the $L^{2}$-curvature regularity
  \cite{Czi1, Czi2}.
\end{remark}
As is evident from (2), it is natural to view the association $(a, e)
\mapsto (\ta, \te, O)$ in (1) as defined up to a constant gauge
transformation.

\subsection{Local theory in an arbitrary topological
  class} \label{subsec:local} We present the third set of results of
this paper, which concern local theory of \eqref{eq:ym} for arbitrary
$\calH^{\frac{d-2}{2}}_{loc}$ initial data set. The main local
well-posedness results in the temporal gauge
(Theorems~\ref{thm:local-temp} and \ref{thm:local-temp-sub}) are
proved as consequences of the finite speed of propagation property of
\eqref{eq:ym}, the results in Section~\ref{subsec:excise} and small
data well-posedness results \cite{OTYM2, TaoYM}.

We start with a (rather general) basic definition of a solution.
\begin{definition} \label{def:ym-sol-rough}
  \begin{enumerate}
  \item An \emph{$\calH^{\frac{d-2}{2}}_{loc}$ connection} in an open
    set $\calO \subseteq \bbR^{1+d}$ is a connection $\covD = \ud + A$
    satisfying
    \begin{equation*}
      (A, \rd_{t} A) \in C_{t} H^{\frac{d-2}{2}}_{loc} \times C_{t} H^{\frac{d-4}{2}}_{loc}(\calO).
    \end{equation*}

  \item An \emph{$\calH^{\frac{d-2}{2}}$ solution for the hyperbolic
      Yang--Mills equation \eqref{eq:ym}} in $\calO$ is an
    $\calH^{\frac{d-2}{2}}_{loc}$ connection $\covD = \ud + A$ in
    $\calO$ which is the limit of regular solutions in the topology
    $C_{t} H^{\frac{d-2}{2}}_{loc} \times C_{t}
    H^{\frac{d-4}{2}}_{loc}(\calO)$.
  \end{enumerate}
\end{definition}
It is straightforward to see that the set of
$\calH^{\frac{d-2}{2}}_{loc}$ solutions is closed with respect to the
$C_{t} H^{\frac{d-2}{2}}_{loc} \times C_{t} H^{\frac{d-4}{2}}_{loc}$
topology.

Next, we formulate the notion of gauge covariance of
$\calH^{\frac{d-2}{2}}_{loc}$ connections, as follows:
\begin{definition} \label{def:ym-sol-rough-gt}
  \begin{enumerate}
  \item A \emph{regular gauge transformation} in an open set $\calO
    \subseteq \bbR^{1+d}$ is a map $O : \calO \to \G$ with the
    regularity properties $O_{;t, x} \in C_{t} H^{N}_{loc}$.
  \item An \emph{admissible gauge transformation} in $\calO$ is a map
    $O : \calO \to \G$ with the regularity properties $O_{;t, x} \in
    C_{t} H^{\frac{d-2}{2}}_{loc}$.
  \item We say that two $\calH^{\frac{d-2}{2}}$ connections $A^{(1)}$
    and $A^{(2)}$ in $\calO$ are gauge equivalent if there exists an
    admissible gauge transformation $O$ in $\calO$ such that
    $A^{(2)}_{j} = Ad(O) A^{(1)}_{j} - O_{;j}$.
  \end{enumerate}
\end{definition}

Any admissible gauge transformation may be approximated by regular
gauge transformations in $C_{t} H^{\frac{d}{2}}_{loc}$ (the proof is a
straightforward variant of Lemma~\ref{lem:part-approx} below, and is
left to the reader). As a consequence, if $A$ and $A'$ are gauge
equivalent $\calH^{\frac{d-2}{2}}$ connections in $\calO$, $A$ is a
$\calH^{\frac{d-2}{2}}$ solution to \eqref{eq:ym} if and only if $A'$
is. Moreover, the class of gauge-equivalent connections is closed:
\begin{proposition} \label{prop:closed-class} The class $[A]$ of
  gauge-equivalent $\calH^{\frac{d-2}{2}}$ connections is closed in
  the topology $C_{t} H^{\frac{d-2}{2}}_{loc} \times C_{t}
  H^{\frac{d-4}{2}}_{loc}(\calO)$
\end{proposition}

With the basic notion of a solution in our hands, we are ready to
discuss the local theory of \eqref{eq:ym} for
$\calH^{\frac{d-2}{2}}_{loc}$ initial data sets. Given a subset $X$ of
$\bbR^{d}$ and a time interval $I$, denote by $\calD_{I}(X)$ the
future domain of dependence of $X$, intersected with $I \times
\bbR^{d}$:
\begin{equation*}
  \calD_{I}(X) = \set{(t, x) \in [0, \infty) \times \bbR^{d} : B_{t}(x) \subseteq X} \cap I \times \bbR^{d}.
\end{equation*}

In \cite{OTYM2}, global well-posedness of \eqref{eq:ym} in the
temporal gauge for small $\dot{\calH}^{\frac{d-2}{2}}$ data on
$\bbR^{d}$ was proved for dimensions\footnote{The exposition of
  \cite{OTYM2} is focused on the case $d = 4$, but the proof extends
  in a straightforward manner to $d \geq 4$.} $d \geq 4$ (see
Theorem~\ref{thm:small-temp} below). Combined with the
excision-and-extension result in Section~\ref{subsec:excise} and the
finite speed of propagation property in the temporal gauge, we obtain:

\begin{theorem}[Local well-posedness at optimal regularity, $d \geq
  4$] \label{thm:local-temp} For $d \geq 4$, there exists a
  dimensional constant $\eps_{\ast} > 0$ such that the Yang--Mills
  equation in the temporal gauge is locally well-posed on the time
  interval of length $\rc^{\eps_\ast} = \rc^{\eps_\ast} [a, e]$ for initial
  data $(a, e) \in \calH^{\frac{d-2}{2}}_{loc}(X)$ for $X = B_{R}$ or
  $\bbR^{d}$. More precisely, the following statements hold.
  \begin{enumerate}
  \item (Regular data) Let $(a, e) $ be a smooth Yang--Mills initial
    data set on $X$. Then there exists a unique smooth solution
    $A_{t,x}$ to the Yang--Mills equation in the temporal gauge on
    $\calD_{[0, \rc)}(X)$ such that $(A_{j}, F_{0j})
    \restriction_{\set{t = 0}} = (a_{j}, e_{j})$.

  \item (Rough data) Let $\calH^{\frac{d-2}{2}}_{loc, \, \rc}(X)$ be
    the class of $\calH^{\frac{d-2}{2}}_{loc}(X)$ Yang--Mills initial
    data sets with concentration scale $\geq \rc$, topologized with
    the norm
    \begin{equation*}
      \nrm{(a, e)}_{\calH^{\frac{d-2}{2}}_{loc, \, \rc}(X)} = \sup_{x \in X} \nrm{(a, e)}_{\dot{H}^{\frac{d-2}{2}} \times \dot{H}^{\frac{d-4}{2}} (B_{\rc}(x) \cap X)}.
    \end{equation*}
    Then the data-to-solution map admits a continuous extension
    \begin{equation} \label{eq:local-temp-cont}
      \calH^{\frac{d-2}{2}}_{loc, \, \rc}(X) \ni (a, e) \mapsto
      (A_{x}, \rd_{t} A_{x}) \in C_{t} \calH^{\frac{d-2}{2}}_{loc, \,
        \rc}(\calD_{[0, \rc)}(X)).
    \end{equation}
  \item (A-priori bound) The solution defined as above obeys the
    a-priori bound
    \begin{equation} \label{eq:local-temp-bnd} \nrm{(A, \rd_{t}
        A)}_{L^{\infty} (H^{\frac{d-2}{2}} \times H^{\frac{d-4}{2}})
        (\calD_{[0, \rc)}(B_{R'}(x)))} \aleq \nrm{(a,
        e)}_{H^{\frac{d-2}{2}} \times H^{\frac{d-4}{2}}(B_{R'}(x))}
    \end{equation}
    for any $B_{R'}(x) \subseteq X$.
  \end{enumerate}
\end{theorem}

The temporal gauge solution given by Theorem~\ref{thm:local-temp}
represents any $\calH^{\frac{d-2}{2}}_{loc}$ solution in the sense of
Definition~\ref{def:ym-sol-rough}.
\begin{theorem} \label{thm:equiv-temp} Any
  $\calH^{\frac{d-2}{2}}_{loc}$ solution to the hyperbolic Yang--Mills
  equation in $\calD_{I}(X)$ (where $X = B_{R}$ or $\bbR^{d}$) can be
  put into the temporal gauge.
\end{theorem}

When $X = \bbR^{d}$, we say that $A$ is a $\calH^{\frac{d-2}{2}}$
solution to the hyperbolic Yang--Mills equation in $I \times \bbR^{d}$
if it is an $\calH^{\frac{d-2}{2}}_{loc}$ solution, and moreover
satisfies the following condition for every $t \in I$:
\begin{equation} \label{eq:tail}
  \calE^{\frac{d-2}{2}}_{\bbR^{d}}(A_{x}(t) , F_{0x}(t)) < \infty.
\end{equation}
By Uhlenbeck's lemma and Theorem~\ref{thm:local-temp}.(3),
\eqref{eq:tail} holds for every $t \in I$ if it holds for its data
$(a, e)$ at some $t \in I$.  For such a solution, the topological
class of $A_{x}(t)$ is preserved under the hyperbolic Yang--Mills
evolution.
\begin{proposition} \label{prop:top-class-ym} Let $A$ be an
  $\calH^{\frac{d-2}{2}}$ solution to \eqref{eq:ym} in $I \times
  \bbR^{4}$. Then $[A_{x}(t)]$ is constant in $t$.
\end{proposition}

The temporal gauge is convenient in order to deal with causality, but
it lacks good dispersive bounds in contrast to the caloric gauge
\cite{OTYM2} (cf. also the small data result in the Coulomb gauge in
\cite{KT}). In a different global gauge, the caloric gauge regularity
may be patched up, as the following sample result demonstrates:

\begin{theorem} \label{thm:imp-reg} Let $A$ be an
  $\calH^{\frac{d-2}{2}}_{loc}$ solution to \eqref{eq:ym} in
  $\calD_{[0, \rc)}(B_{R})$, whose initial data set has critical
  $L^{2}$-Sobolev concentration scale $\geq \rc$ with sufficiently
  small $\eps_{\ast} > 0$. In a suitable global gauge in $D = [0,
  \rc) \times B_{R - 4 \rc}$, the solution obeys
  \begin{equation} \label{eq:imp-reg} \nrm{\nb A_{x}}_{L^{\infty}
      \dot{H}^{\frac{d-4}{2}}(D)} + \nrm{\Box A_{x}}_{\ell^{1} L^{2}
      \dot{H}^{\frac{d-5}{2}}(D)} + \nrm{\nb A_{0}}_{\ell^{1} L^{2}
      \dot{H}^{\frac{d-3}{2}}(D)} \aleq_{\eps_{\ast}, \frac{R}{\rc}}
    1.
  \end{equation}
\end{theorem}
\begin{remark}
  The restriction to $[0, \rc) \times B_{R - 4 \rc}$ instead of
  $\calD_{[0, \rc)(B_{R})}$ is enforced merely to avoid technical
  issues near the boundary, and may be removed if desired. We do not
  pursue this improvement, since Theorem~\ref{thm:imp-reg} suffices
  for our application in \cite{OTYM3}.
\end{remark}
Finally, we discuss application of our techniques to the case of $d =
3$. For $X = B_{R}$ or $\bbR^{3}$, we topologize the space
$\calH^{\sgm}_{loc}(X)$ with the norm
\begin{equation*}
  \nrm{(a, e)}_{\calH^{\sgm}_{loc}(X)} = \sup_{x \in X} \nrm{(a, e)}_{H^{\sgm} \times H^{\sgm-1}(B_{1}(x) \cap X)}.
\end{equation*}
From the small data local well-posedness result of Tao \cite{TaoYM},
we obtain the following large data result:
\begin{theorem}[Local well-posedness in the temporal gauge, $d =
  3$] \label{thm:local-temp-sub} Let $\sgm > \frac{3}{4}$. The
  Yang-Mills equation in the temporal gauge is locally well-posed for
  initial data $(a, e) \in \calH^{\sgm}_{loc}(\bbR^{3})$ on a time
  interval of length $\geq T(\nrm{(a, e)}_{\calH^{\sgm}_{loc}})$.
\end{theorem}

Moreover, the techniques of this paper lead to an alternative proof
of the classical result of Klainerman--Machedon \cite{KlMa2}:
\begin{theorem} \label{thm:KM} The Yang--Mills equation in the
  temporal gauge is globally well-posed for initial data $(a, e) \in
  \calH^{1}_{loc}(\bbR^{3})$.
\end{theorem}
An advantage of the present approach is that the delicate issue of
boundary values on spacetime cones (i.e., the domains of dependence of
balls) is avoided by the robust excision-and-extension procedure. We
note that yet another proof of Theorem~\ref{thm:KM} relying on a
global gauge defined by the Yang--Mills heat flow (a subcritical
version of the \emph{caloric gauge} we use in the present series
\cite{OTYM1, OTYM2, OTYM3}) was given by the first author \cite{Oh1,
  Oh2}.

\subsection{Topological classes, instantons and harmonic Yang--Mills
  connections on \texorpdfstring{$\bbR^{4}$}{R4}} \label{subsec:threshold} In this
subsection, we restrict to the energy critical dimension $d = 4$, and
discuss the relationship between the topological class of a connection
$a$ on $\bbR^{4}$ and its \emph{static energy}
\begin{equation} \label{eq:ym-har-en} \spE(a) = \calE_{\bbR^{4}}(a, 0)
  = \frac{1}{2} \int_{\bbR^{4}} \brk{F_{jk}[a], F^{jk}[a]} \, \ud x.
\end{equation}

Recall that each topological class $[a]$ of finite energy connections
form a path-connected component in the $\dot{H}^{1}$ distance up to
gauge transformations (Section~\ref{subsec:top-class}). We may
therefore look for an absolute minimizer of $\spE(a)$ in each
topological class; such a connection is called an
\emph{instanton}\footnote{Usually, one also distinguishes between an
  instanton and an anti-instanton, depending on whether the curvature
  is self- or anti-self-dual. Here, we make no such
  distinction.}. More generally, we refer to a critical point of
\eqref{eq:ym-har-en} as a \emph{harmonic Yang--Mills connection}.

Such connections are clearly static solutions to both the Yang--Mills
heat flow and the hyperbolic Yang--Mills equation, and hence
obstructions to convergence of solutions to the trivial connection (as well as
scattering). Moreover, these connections may also arise as ``bubbles''
near the singularity of a dynamic solution. Therefore, knowledge of the
energies of the harmonic Yang--Mills connections is necessary for
determining the precise threshold energy in the Threshold Theorem,
both for the Yang--Mills heat flow \cite{OTYM1} and for the hyperbolic
Yang--Mills equation \cite{OTYM3}.

We open our discussion with the important special case $\G =
  SU(2)$. The corresponding Lie algebra $\g = su(2)$ consists of $2
  \times 2$ complex anti-hermitean matrices with zero trace. We
  furthermore assume that the $Ad$-invariant inner product on $\g$
  takes the form
  \begin{equation*}
    \brk{A, B} = - \tr (A B).
  \end{equation*}
  In fact, as all $Ad$-invariant inner products on $\g$ are positive
  multiples of each other, there is no loss of generality.

  In this case, the topological classes of finite energy connections
  are classified by the second Chern number $c_{2}$, which takes the
  explicit form (via the Chern--Weil theory)
  \begin{equation} \label{eq:chern-no} c_{2} = \frac{1}{8 \pi^{2}}
    \int_{\bbR^{4}} \tr (F[a] \wedge F[a]).
  \end{equation}
  For any finite energy connection $a$, the second Chern number
  $c_{2}$ is an integer; in fact, it equals the degree of the
  $0$-homogeneous map $O$ (defined using the homeomorphism $SU(2)
  \simeq \bbS^{3}$) in Theorem~\ref{thm:goodrep}. A simple algebraic
  manipulation using the Hodge star operator\footnote{To define
    $\star$, we use the standard inner product on $2$-forms such that
    $\set{\ud x^{j} \wedge \ud x^{k} : j < k}$ is an orthonormal
    basis.} $\star$ shows that
  \begin{align*}
    \brk{F_{jk}[a], F^{jk}[a]} =& - \star 2 \tr(F \wedge \star F)  \\
    = & - \star \tr((F \pm \star F) \wedge \star (F \pm \star F)) \pm 2 \star \tr
    (F \wedge F) \\
    = & \frac{1}{2} \brk{F \pm \star F, F \pm \star F} \pm 2 \star \tr
    (F \wedge F) .
  \end{align*}
  Note that the first term on the last line is nonnegative. Integrating over $\bbR^{4}$, we obtain the \emph{Bogomoln'yi bound}
  \begin{equation} \label{eq:bog} \spE(a) \geq 8 \pi^{2} \abs{c_{2}}.
  \end{equation}
  The equality holds (in which case, $a$ is an instanton) if and only
  if $F = \mp \star F$, where $\pm$ is the sign of $c_{2}$. We call
  such a connection \emph{anti-self} or \emph{self dual},
  respectively. There is a beautiful theory due to
  Atiyah--Drinfeld--Hitchin--Manin \cite{ADHM}, which gives an explicit
  construction of all anti-self dual (resp. self-dual) connections
  with $c_{2} > 0$ (resp. $c_{2} < 0$).  In particular, we have:
  \begin{theorem}[\cite{ADHM}]\label{thm:instanton-SU2}
    For any $\kpp \in \bbZ$, there exists an instanton with $c_{2} = -
    \kpp$ and energy $8 \pi^{2} \abs{\kpp}$.
  \end{theorem}

  However, the instantons do not tell the full story. It is known that
  there also exist nontrivial harmonic Yang--Mills connections which
  are not self or anti-self dual \cite{SSU, Bor, SaSe,
    Parker}. Nevertheless, by the recent result of
  Gursky--Kelleher--Streets \cite{GKS}, they must have energy at least
  $16 \pi^{2}$ more than the Bogomoln'yi bound\footnote{Note that
    \cite[Corollary~1.2]{GKS} is stated on $\bbS^{4}$, but the same
    conclusion holds on $\bbR^{4}$ by conformal invariance of the
    harmonic Yang--Mills equation and $\spE$. Moreover, to compare the
    results, recall that $\spE(a) = \frac{1}{2}
    \nrm{F[a]}_{L^{2}}^{2}$.}:
  \begin{theorem}[{\cite[Corollary~1.2]{GKS}}] \label{thm:GKS-SU2} Any
    harmonic Yang--Mills connection on $\bbR^{4}$ either has energy
    equal to $8 \pi^{2} \abs{c_{2}}$, or has energy at least $8
    \pi^{2} \abs{c_{2}} + 16 \pi^{2}$.
  \end{theorem}
  In conclusion, we see that: \emph{Any nontrivial harmonic $SU(2)$
    Yang--Mills connection either has energy at least $16 \pi^{2}$, or
    it is an instanton with $c_{2} = \pm 1$ (a \emph{first instanton})
    with energy $8 \pi^{2}$.} We call the first instanton alternatively as the \emph{ground state} (as it has the lowest nontrivial energy), and refer to its energy as the \emph{ground state energy} $\Egs$.

  We now turn to the general case when $\G$ is a compact Lie group,
  for which our goal is to establish a similar conclusion.  Consider
  $f_{2} (\cdot, \cdot) = - \brk{\cdot, \cdot}$, which is a symmetric
  $Ad$-invariant bilinear function, and the corresponding
  characteristic class (cf. Section~\ref{subsec:smth}).
  \begin{equation} \label{eq:ch-cl} - \brk{F[a] \wedge F[a]} = -
    \brk{F_{i j}[a], F_{k \ell}[a]} \, \ud x^{i} \wedge \ud x^{j}
    \wedge \ud x^{k} \wedge \ud x^{\ell}.
  \end{equation}
  The characteristic number
  \begin{equation} \label{eq:ch-no} \ch = \int_{\bbR^{4}} - \brk{F[a]
      \wedge F[a]}
  \end{equation}
  is determined by the topological class $[a]$, by
  Corollary~\ref{cor:ch-class}. Moreover, the same algebra as in
  \eqref{eq:bog} leads to:
  \begin{lemma} \label{lem:bog-G} Let $\G$ be a compact Lie group. For
    any finite energy connection $a$ on a $\G$-bundle on $\bbR^{4}$,
    we have the pointwise bound
    \begin{equation} \label{eq:bog-G} \frac{1}{2} \brk{F_{jk}[a],
        F^{jk}[a]} \geq \abs{\brk{F[a] \wedge F[a]}},
    \end{equation}
    and the corresponding integrated bound
    \begin{equation*}
      \spE(a) \geq \abs{\ch}.
    \end{equation*}

  \end{lemma}

  Note that when $\G$ is commutative, then the harmonic Yang--Mills
  connections are nothing else than the harmonic $2$-forms; thus no
  nontrivial finite energy harmonic Yang--Mills connections exist. In
  the noncommutative case, we prove:
  \begin{theorem} \label{thm:thr} Let $\G$ be a noncommutative compact Lie
    group. Let
    \begin{equation*}
      \Egs = \inf \set{\spE(a) : \hbox{$a$ is a nontrivial harmonic Yang--Mills connection on a $\G$-bundle on $\bbR^{4}$}}.
    \end{equation*}
    Then the following statements hold.
    \begin{enumerate}
    \item There exists a nontrivial harmonic Yang--Mills connection
      $a$ such that $\spE(Q) = \Egs < \infty$.
    \item Let $a$ be any nontrivial harmonic Yang--Mills
      connection. Then either $\spE(a) \geq 2 \Egs$, or
      \begin{equation*}
	\abs{\ch} = \spE(a) \geq \Egs.
      \end{equation*}

    \end{enumerate}
  \end{theorem}
  We call $\Egs$ the \emph{ground state energy}, and a harmonic Yang--Mills connection $Q$ attaining this energy a \emph{ground state}.
  
  The proof of Theorem~\ref{thm:thr} combines well-known results
  concerning the structure of a compact Lie group and the preceding
  analysis in the case $\G = SU(2)$; it is provided in
  Section~\ref{sec:threshold}.

\addtocontents{toc}{\protect\setcounter{tocdepth}{-1}}
\subsection*{Acknowledgments} 
S.-J. Oh was supported by the Miller Research Fellowship from the Miller Institute, UC Berkeley and the TJ Park Science Fellowship from the POSCO TJ Park Foundation. D. Tataru was partially
supported by the NSF grant DMS-1266182 as well as by a Simons
Investigator grant from the Simons Foundation.

\addtocontents{toc}{\protect\setcounter{tocdepth}{2}}

\section{Notation and conventions} \label{sec:notation} Here we
collect some notation and conventions used in this paper.
\begin{itemize}
\item We employ the usual asymptotic notation $A \aleq B$ to denote $A
  \leq C B$ for some implicit constant $C > 0$. The dependence of $C$
  on various parameters is specified by subscripts.

\item Throughout the paper, we omit the dependence of constants on the
  dimension $d$. In particular, by a universal constant, we mean a
  constant that depends only on $d$.

\item We call a bounded open subset $U$ of $\bbR^{d}$ a
  \emph{domain}. For $\lmb > 0$, $\lmb U$ is defined to be rescaling
  of $U$ by $\lmb$ centered at the barycenter of $U$. For any $r > 0$
  and $x \in \bbR^{d}$, $B_{r}(x)$ is the ball of radius $r$ centered
  at $x$. When $(x)$ is omitted, the center is taken to be the origin
  $0$.

\item We use the notation $\rd$ (without sub- or superscripts) for the
  spatial gradient $\rd = (\rd_{1}, \rd_{2}, \ldots, \rd_{d})$, and
  $\nb$ for the spacetime gradient $\nb = (\rd_{0}, \rd_{1}, \ldots,
  \rd_{d})$. We write $\rd^{(n)}$ (resp. $\nb^{(n)}$) for the
  collection of $n$-th order spatial (resp. spacetime) derivatives,
  and $\rd^{(\leq n)}$ (resp. $\nb^{(\leq n)}$) for those up to order
  $n$.

\item The $n$-th homogeneous $L^{p}$-Sobolev space for functions from
  $\bbR^{d}$ into a normed vector space $V$ is denoted by $\dot{W}^{n,
    p}(\bbR^{d}; V)$. In the special case $p = 2$, we write
  \[
\dot{H}^{n}(\bbR^{d}; V) = \dot{W}^{n, 2}(\bbR^{d}; V).
\]
 The
  inhomogeneous counterparts are denoted by $W^{n, p}(\bbR^{d}; V)$
  and $H^{n}(\bbR^{d}; V)$, respectively. The Lebesgue spaces (i.e.,
  when $n = 0$) are denoted by $L^{p}(\bbR^{d}; V)$.

\item The mixed spacetime norm $L^{q}_{t} \dot{W}^{n, r}_{x}$ of
  functions on $\bbR^{1+d}$ is often abbreviated as $L^{q} \dot{W}^{n,
    r}$.

\item Given a function space $X$ (on either $\bbR^{d}$ or
  $\bbR^{1+d}$), we define the space $\ell^{p} X$ by
  \begin{equation*}
    \nrm{u}_{\ell^{p}X}^{p} = \sum_{k} \nrm{P_{k} u}^{p}_{X}
  \end{equation*}
  (with the usual modification for $p = \infty$), where $P_{k}$ $(k
  \in \bbZ)$ are the usual Littlewood--Paley projections to dyadic
  frequency annuli.

\item Generally, a function space on an open subset $U \subseteq
  \bbR^{d}$ is defined by restriction, i.e., $\nrm{u}_{X(U)} = \inf
  \set{\nrm{\tilde{u}}_{X} : \tilde{u} \in X, \ \tilde{u}
    \restriction_{U} = u}$. A similar convention applies for a function
  space on an open subset $\calO \subseteq \bbR^{1+d}$.

  According to this convention, the restriction of the homogeneous
  Sobolev norm $\dot{W}^{n, p}$ for $n \in \bbN$, $1 < p <
  \frac{d}{n}$ for a locally Lipschitz domain $U$ is characterized by
  \begin{equation*}
    \nrm{u}_{\dot{W}^{n, p}(U)} \simeq_{U} \nrm{\rd^{(n)} u}_{L^{p}(U)} + \nrm{u}_{L^{p^{\ast}}(U)}, \quad \hbox{ where } \frac{d}{p^{\ast}} = \frac{d}{p} - n.
  \end{equation*}
  Note, importantly, that the implicit constant is invariant under
  scaling. To distinguish this norm from the usual homogeneous Sobolev
  semi-norm, we introduce the notation $\mathring{W}^{n, p}(U)$ for a
  nonnegative integer $n$ and $p \in [1, \infty]$, and define
  $\nrm{u}_{\mathring{W}^{n, p}(U)} = \nrm{\rd^{(n)} u}_{L^{p}(U)}$.
\item The local function space $X_{loc}(U)$ is defined as
  \begin{equation*}
    X_{loc}(U) = \bigcap_{B_{x}(r): \overline{B}_{x}(r) \subseteq U} X(B_{x}(r)).
  \end{equation*}
\end{itemize}

\section{Connections with
  \texorpdfstring{$L^{\frac{d}{2}}$}{Ld/2}-curvature} \label{sec:rough-conn} In this section,
we prove the good global gauge theorems
Theorems~\ref{thm:goodrep-ball} and \ref{thm:goodrep}.  Throughout the
section, we let $d \geq 3$.

\subsection{\texorpdfstring{$\G$}{G}-valued functions at critical
  regularity} \label{subsec:rough-gt} We start by collecting some
basic analytic facts concerning $\G$-valued functions at regularity
$W^{k, \frac{d}{k}}$.

In what follows, we assume that $\G$ is a group of orthogonal matrices
in $\bbR^{N \times N}$, equipped with the usual inner product $\brk{A,
  B} = \tr A B^{\dagger}$. Recall the standard fact that any compact
Lie group $\G$ may be realized as such a matrix group, and the inner
product on $\g = T_{Id} \G$ is equivalent to the one induced from
$\bbR^{N \times N}$.

Let $U \subseteq \bbR^{d}$ be an open set, $k \in \bbR$ and $p \in [1,
\infty]$. In Section~\ref{subsec:rough-conn}, we introduced
\begin{align*}
  \calG^{k, p}(U) = & \set{O \in W^{k, p}(U; \bbR^{N \times N}) : O
    (x) \in \bfG \hbox{ for a.e. } x \in U}.
\end{align*}
Since $\G$ is compact, any $O \in \calG^{k, p}(U)$ belongs to
$L^{\infty}(U)$. When $U$ is a domain with locally Lipschitz boundary,
an element $O \in \calG^{k, p}(U)$ may be extended\footnote{We
  emphasize, however, that $\tilde{O}(x) \not \in \bfG$ for $x \not
  \in U$ in general.} to $\tilde{O} \in W^{k, p} \cap
L^{\infty}(\bbR^{d})$; see \cite[\S VI.3]{MR0290095}. For a general
irregular open set $U$, we instead use
\begin{align*}
  \calG^{k, p}_{loc}(U) = & \set{O \in W^{k, p}_{loc}(U; \bbR^{N
      \times N}) : O (x) \in \bfG \hbox{ for a.e. } x \in U},
\end{align*}
for which the following extension property holds: For any ball $B
\subseteq U$, there exists ${}^{(B)}\tilde{O} \in W^{k, p} \cap
L^{\infty}(\bbR^{d})$ such that ${}^{(B)}\tilde{O}(x) = O(x)$ for
a.e. $x \in B$.

In view of the applications to the hyperbolic Yang--Mills equation at
the critical regularity, we consider the scale-invariant case $p =
\frac{d}{k} > 1$, which is subtle due to the fact that $H^{k,
  \frac{d}{k}} \not \hookrightarrow L^{\infty}$, and thus $H^{k,
  \frac{d}{k}}$ is not an algebra. Nevertheless, as we will see, basic
operations needed to define a $\G$-bundle are still well-defined. To
avoid technical issues, we focus on the case when $k$ is a positive
integer. Of special importance is when $k = 2$, which correspond to
local gauge transformations in a bundle admitting a connection with
$L^{\frac{d}{2}}$ curvature.

As a quick consequence of the extension properties mentioned above, we
have the following multiplication lemma.
\begin{lemma} \label{lem:mult} Let $k$ be a positive integer, and let
  $U \subseteq \bbR^{d}$ be an open set. Then the pointwise
  multiplication map
  \begin{equation*}
    \calG^{k, \frac{d}{k}}_{loc}(U) \times \calG^{k, \frac{d}{k}}_{loc}(U) \ni (O_{1}, O_{2}) \mapsto O_{1} \cdot O_{2} \in \calG^{k, \frac{d}{k}}_{loc}(U)
  \end{equation*}
  is continuous. If $U$ is a domain with a locally Lipschitz boundary,
  then the same conclusion holds for the space $\calG^{k,
    \frac{d}{k}}(U)$.
\end{lemma}
Although multiplication is continuous, we remark that it utterly fails
to be any more regular. This is in sharp contrast with the subcritical
case $\calG^{k, p}$ with $p > \frac{d}{k}$, in which multplication is
smooth.
\begin{proof}
  It suffices to consider the case when $U$ is a domain with a locally
  Lipschitz boundary (the other case follows by taking $U$ to be
  balls). Let $O_{1}, O_{2} \in \calG^{k, \frac{d}{k}}(U)$, and
  consider their usual extensions outside $U$. Note that $O_{1} \cdot
  O_{2}$ is an $L^{1}_{loc}$ function with values in $\bfG$ for
  a.e. $x \in U$, and belongs to $W^{k, \frac{d}{k}}(U)$ by the whole
  space estimate
  \begin{equation*}
    \nrm{O_{1} \cdot O_{2}}_{W^{k, \frac{d}{k}}} \aleq \nrm{O_{1}}_{L^{\infty}} \nrm{O_{2}}_{W^{k, \frac{d}{k}}} + \nrm{O_{1}}_{W^{k, \frac{d}{k}}} \nrm{O_{2}}_{L^{\infty}}.
  \end{equation*}
  To prove continuity, consider sequences $O^{n}_{1} \to O_{1}$ and
  $O^{n}_{2} \to O_{2}$ in $\calG^{k, \frac{d}{k}}(U)$. We extend
  $O^{n}_{1}$ and $O^{n}_{2}$ to the whole space using the same
  extension operator as before, which insures $O^{n}_{1} \to O_{1}$
  and $O^{n}_{2} \to O_{2}$ in $W^{k, \frac{d}{k}}(\bbR^{d}; \bbR^{N
    \times N})$. By the Leibniz rule and the Sobolev inequality, for
  any multi-index $\alp$ of order $k$, we may show that
  \begin{equation*}
    \rd^{\alp} (O^{n}_{1} \cdot O^{n}_{2}) - (\rd^{\alp} O^{n}_{1}) \cdot O^{n}_{2} - O^{n}_{1} \cdot \rd^{\alp} O^{n}_{2}
    \to \rd^{\alp} (O_{1} \cdot O_{2}) - (\rd^{\alp} O_{1}) O_{2} - O_{1} \rd^{\alp} O_{2} \hbox{ in } L^{\frac{d}{k}}.
  \end{equation*}
  By symmetry, it only remains to prove that $(\rd^{\alp} O^{n}_{1})
  \cdot O^{n}_{2} \to (\rd^{\alp} O_{1}) \cdot O_{2}$ in
  $L^{\frac{d}{k}}$. Since $O^{n}_{2}$ is uniformly bounded, the
  problem is further reduced to showing that
  \begin{equation*}
    \nrm{\rd^{\alp} O_{1} \cdot (O^{n}_{2} - O_{2})}_{L^{\frac{d}{k}}} \to 0.
  \end{equation*}
  If this limit were not true, then there would exist a subsequence
  with no further subsequence converging to zero.  However, $O^{n}_{2}
  \to O_{2}$ in $W^{k, \frac{d}{k}}$ implies a.e. convergence along a
  subsequence, along which the above limit holds by the dominated
  convergence theorem. \qedhere
\end{proof}

It is well-known that if $U$ is an open set with piecewise smooth
boundary, then any $O \in \calG^{2, \frac{d}{2}}(U)$ can be
approximated by a sequence $O^{n} \in C^{\infty}(U; \G)$ in the $W^{2,
  \frac{d}{2}}(U; \bbR^{N \times N})$-topology \cite{MR710054}. We
state here a technical refinement which allows us to localize the
region where we perform the approximation (essentially from
\cite{MR815194}). This version will be helpful for handling the
extension problem to a $\G$-valued map (not $\bbR^{N \times
  N}$-valued).

\begin{lemma} \label{lem:part-approx} Let $k$ be a positive
  integer. Let $U \subseteq \bbR^{d}$ be a domain with locally
  Lipschitz boundary, and let $O \in \calG^{k, \frac{d}{k}}(U)$. If
  $V, W$ are (possibly empty) open sets in $\bbR^{d}$ such that
  $\overline{V} \cup \overline{W} \subseteq \overline{U}$ and
  $\overline{V} \cap \overline{W} = \0$, then for every $\eps > 0$
  there exists $O' \in \calG^{k, \frac{d}{k}}(U)$ such that $O'
  \restriction_{V} = O \restriction_{V}$, $O' \in C^{\infty}(W; \G)$
  and $\nrm{O' - O}_{W^{k, \frac{d}{k}}(U; \bbR^{N \times N})} <
  \eps$.
\end{lemma}
We recover the usual approximation result by setting $V = \0$ and $W =
U$. As a consequence, for a general open set $U$, any $O \in \calG^{k,
  \frac{d}{k}}_{loc}(U)$ can be approximated by $O^{n} \in
C^{\infty}(B; \G)$ in the $W^{k, \frac{d}{k}}(B; \bbR^{N \times
  N})$-topology for any open ball $B \subseteq U$.
\begin{proof}
  We may assume that $W \neq \0$, as otherwise we may set $O^{\eps} =
  O$. By standard Sobolev extension, there exists $\tilde{O} \in W^{k,
    \frac{d}{k}}(\bbR^{d}; \bbR^{N \times N})$ such that $\tilde{O}
  \restriction_{U} = O$. We introduce $\dlt > 0$ to be fixed later,
  and let $h : U \to [0, 1]$ be a smooth function such that $h = 0$ on
  $V$ and $h = 1$ on $W$ (smooth Urysohn's lemma). Fix a smooth
  function $\zt$ supported in the unit ball satisfying $\int \zt =
  1$. We define $\tilde{O}^{\dlt} : \bbR^{d} \to \bbR^{N \times N}$ by
  inhomogeneous mollification:
  \begin{equation*}
    \tilde{O}^{\dlt} (x) = \int \zt(y) \tilde{O}(x - \dlt h(x) y) \, \ud y.
  \end{equation*}
  It is straightforward to verify that
  $\nrm{\tilde{O}^{\dlt}-\tilde{O}}_{W^{k, \frac{d}{k}}(U)} \to 0$ as
  $\dlt \to 0$, and also that $\tilde{O}^{\dlt}$ is smooth on
  $W$. However, $\tilde{O}^{\dlt} (x) \not \in \G$ in general. To
  rectify this, we proceed as in \cite{MR710054}.

  Let $\tilde{\G} \subseteq \bbR^{N \times N}$ be a tubular
  neighborhood of $\G$, on which the nearset-point projection
  $\pi_{\G} : \tilde{\G} \to \G$ is well-defined as a smooth map. For
  $x \in U$, we wish to ensure that $\tilde{O}^{\dlt}(x) \in
  \tilde{\G}$ for $\dlt$ sufficiently small. Since $O(y) \in \G$ for
  a.e. $y \in U$, we have
  \begin{align*}
    d(\tilde{O}^{\dlt}(x), \G)^{d} \leq \frac{1}{\abs{U \cap B_{\dlt
          h(x)}(x)}} \int_{U \cap B_{\dlt h(x)}(x)}
    \abs{\tilde{O}^{\dlt}(x) - O(y)}^{d} \ \ud y .
  \end{align*}
  By boundedness and the locally Lipschitz condition, $\abs{U \cap
    B_{r}(x)} \ageq_{U, d} r^{d}$ for every $x \in \overline{U}$ and
  sufficiently small $r > 0$. Moreover, by the Poincar\'e inequality
  $\nrm{f}_{L^{d}(B_{r}(x))} \aleq_{\zt} r \nrm{\rd
    f}_{L^{d}(B_{r}(x))}$ for $f$ satisfying $\int \zt(y) f(x + r y)
  \, \ud y = 0$, we have
  \begin{equation*}
    d(\tilde{O}^{\dlt}(x), \G)^{d} \aleq_{U, \zt, d} \int_{B_{\dlt h(x)}(x)} \abs{\rd \tilde{O}(y)}^{d} \ \ud y.
  \end{equation*}
  By compactness of $\overline{U}$, the RHS goes to $0$ uniformly as
  $\dlt \to 0$, so that $\tilde{O}^{\dlt}(x) \in \tilde{\G}$.

  Define $O' = \pi_{\G} \circ \tilde{O}^{\dlt} \restriction_{U}$. It
  is now straightforward to show that $O'$ obeys the desired
  properties once we fix $\dlt > 0$ small enough (depending on
  $\eps$).
\end{proof}

As a consequence of the approximation property, we now show that
pointwise inversion is well-defined as a continuous map $\calG^{k,
  \frac{d}{k}}_{loc}(U) \to \calG^{k, \frac{d}{k}}_{loc}(U)$.
\begin{lemma} \label{lem:inv} Let $k$ be a positive integer, and let
  $U \subseteq \bbR^{d}$ be an open set. Then the pointwise inversion
  map
  \begin{equation*}
    \calG^{k, \frac{d}{k}}_{loc}(U) \ni O \mapsto O^{-1} \in \calG^{k, \frac{d}{k}}_{loc}(U)
  \end{equation*}
  is continuous. Moreover, the usual differentiation rule $\rd_{x}
  O^{-1} = - O^{-1} \rd_{x} O O^{-1}$ holds for $O \in \calG^{k,
    \frac{d}{k}}_{loc}(U)$. If $U$ is a domain with a locally
  Lipschitz boundary, then the same conclusion holds for the space
  $\calG^{k, \frac{d}{k}}(U)$.
\end{lemma}
\begin{proof}
  As before, we only consider the case when $U$ is a domain with a
  locally Lipschitz boundary. For simplicity, we only treat the case
  $k = 1$; the higher $k$'s are handled similarly. Given $O \in
  \calG^{1, d}(U)$, let $O^{n} \to O$ be a smooth approximation
  sequence in $\calG^{1, d}(U)$ given by Lemma~\ref{lem:part-approx}
  (with $V = \0$ and $W = U$). By passing to a subsequence, we may
  assume that $O^{n} \to O$ a.e. in $U$ as well; hence $(O^{n})^{-1}
  \to O^{-1}$ in $U$. Moreover, by the usual differentiation formula
  in the smooth case,
  \begin{equation*}
    \rd_{x} (O^{n})^{-1} = - (O^{n})^{-1} \rd_{x} O^{n} (O^{n})^{-1}.
  \end{equation*}
  By the dominated convergence theorem, $\rd_{x} (O^{n})^{-1}$ is
  Cauchy in $W^{1, d}(U; \bbR^{N \times N})$, so that $O^{-1} \in
  \calG^{1, d}(U)$. Moreover, the formula
  \begin{equation*}
    \rd_{x} O^{-1} = - O^{-1} \rd_{x} O O^{-1}
  \end{equation*}
  is justified for $O \in \calG^{1, d}(U)$. By a similar argument
  using the dominated convergence theorem applied to an arbitrary
  sequence $O^{n} \to O$ in $\calG^{1, d}(U)$. the continuity property
  also follows.
\end{proof}

Next, from the approximation property and Lemma~\ref{lem:inv}, it
follows that the usual operations involving $\calG^{k,
  \frac{d}{k}}_{loc}(U)$ and $W^{k', \frac{d}{k}}_{loc}(U; \g)$ are
continuous.
\begin{lemma} \label{lem:d-inv} Let $k$ be a positive integer, and let
  $U \subseteq \bbR^{d}$ be an open set.
  \begin{enumerate}
  \item The operations $O \mapsto O_{;x} = \rd_{x} O O^{-1}$ and $O
    \mapsto O^{-1}_{;x} = - O^{-1} \rd_{x} O$ are continuous as
    mappings $\calG^{k, \frac{d}{k}}_{loc}(U) \to W^{k-1,
      \frac{d}{k}}_{loc}(U; \g)$.
  \item For any integer $0 \leq k' \leq k$, the operation $(O, B)
    \mapsto Ad(O) B = O B O^{-1}$ is continuous as a mapping
    $\calG^{k, \frac{d}{k}}_{loc}(U) \times W^{k',
      \frac{d}{k}}_{loc}(U; \g) \to W^{k', \frac{d}{k}}_{loc}(U; \g)$.
  \item If $O, O_{1}, O_{2} \in \calG^{k, \frac{d}{k}}_{loc}(U)$ and
    $B \in W^{k', \frac{d}{k}}_{loc}(U; \g)$, then the following
    Leibniz rules hold:
    \begin{align*}
      (O_{1} O_{2})_{;x} =& O_{1;x} + Ad(O_{1}) O_{2;x}, \\
      \rd_{x} (Ad(O) B) =& Ad(O) \rd_{x} B + Ad(O) [O_{;x}, B].
    \end{align*}
  \end{enumerate}
  If $U$ has a locally Lipschitz boundary, then the same conclusion
  holds for the spaces $\calG^{k, \frac{d}{k}}(U)$ and $W^{k',
    \frac{d}{k}}(U; \g)$.
\end{lemma}
As before, the fact  that these operations map into the right space is justified
by using a smooth approximating sequence (Lemma~\ref{lem:part-approx}),
and then their continuity properties are proved in a similar
manner. We omit the proof.

We end with an auxiliary lemma concerning the construction of a
$\G$-valued function on an annulus with a prescribed normal derivative
on the outer boundary.
\begin{lemma} \label{lem:Ar=0} Let $B$ be a $\g$-valued function in
 $H^{\frac{d-3}{2}}(\bbS^{d-1})$. There exists $O \in
  L^{\infty} \cap H^{\frac{d}{2}}(B_{1})$, which depends continuously on $B$,
  such that
  \begin{equation*}
    (O, O_{;r}) \restriction_{\set{r = 1}} = (Id, B),
  \end{equation*}
  where $O_{;r} = \frac{x^{j}}{\abs{x}} O_{;j}$.
  A similar construction can be done in the exterior region $\bbR^{d}
  \setminus B_{1}$.
\end{lemma}
\begin{proof}
  We first work on the annulus $B_{1} \setminus \overline{\frac{1}{2} B_{1}}$,
  which we view as the product space $(\frac{1}{2}, 1)_{r} \times
  \bbS^{d-1}_{\Tht}$ (note that the corresponding Lebesgue and Sobolev
  spaces are equivalent). We define $\varphi(r, \Tht)$ to be Poisson
  semigroup $\varphi(r, \Tht) = e^{\sqrt{-\lap_{\Tht}} (r-1)} B$, and
  define
  \begin{equation*}
    \Psi(r, \Tht) = (r - 1) \varphi(r, \Tht), \qquad
  \end{equation*}
  By the properties of the Poisson semigroup, observe that
  \begin{equation*}
    \Psi(r, \Tht) = (r-1) B(\Tht) + o_{r \to 1}(r-1) \quad \hbox{ in } H^{\frac{d-3}{2}}(\bbS^{d-1}).
  \end{equation*}
  Moreover, $\Psi(r, \Tht) \in L^{\infty} \cap
  H^{\frac{d}{2}}((\frac{1}{2}, 1) \times \bbS^{d-1})$ and
  \begin{equation*}
    \nrm{\Psi(r, \cdot)}_{L^{\infty}(\bbS^{d-1})} = o_{r \to 1}(1)
  \end{equation*}
  where the rate depends only on the right tail of the
  $H^{\frac{d-3}{2}}$ frequency envelope of $B$. Let
  \begin{equation*}
    O(r, \Tht) = \exp (\chi \Psi(r, \Tht)).
  \end{equation*}
  where $\chi = \chi(r)$ is a smooth radial function such that $\chi =
  0$ in $\set{r < \frac{2}{3}}$ and $\chi = 1$ in $\set{r >
    \frac{5}{6}}$.  Since $L^{\infty} \cap H^{\frac{d}{2}}$ is an
  algebra, and since $O = Id$ in $\set{r < \frac{2}{3}}$, it may be
  checked that $O \in L^{\infty} \cap H^{\frac{d}{2}}(B_{1})$. Moreover,
  $\rd_{r} O(r, \Tht) O^{-1}(r, \Tht) \restriction_{\set{r=1}} =
  \rd_{r} \Psi(r, \Tht) \restriction_{\set{r = 1}} = B(\Tht)$, as
  desired. \qedhere
\end{proof}

\subsection{Patching procedures} \label{subsec:patching} Here we
describe procedures for patching together local gauges to a global
gauge, which is one of the main ingredients of the proof of the good
global gauge theorems.

We consider three scenarios:
\begin{enumerate}
\item Local gauges given on small (round) cubes $Q_{(\alp)}$ covering
  a large (round) cube $Q_{R}$;
\item Local gauges given on small balls $B_{(\alp)}$ covering a large
  ball $B_{R}$;
\item Local gauges given on concentric balls $B_{R_{n}}$ covering $X =
  B_{R}$ or $\bbR^{d}$.
\end{enumerate}
In all three scenarios, the patching procedure depends only on the
trivial topology and differentiable structure of the base.

\subsubsection*{Scenario~(1): Large cubes covered by smaller cubes}
We first consider a covering consisting of (round) cubes, which admits
simple intersection properties.

Let $Q_{R}$ be a smooth domain in $\bbR^{d}$, and consider a covering
$\set{Q_{\alp}}_{\alp \in \Gmm}$ of $Q_{R}$ by smooth domains
$Q_{\alp}$ indexed by a subset $\Gmm$ of the lattice $\bbZ^{d}$. We
equip $\bbZ^{d}$ with two norms: $\abs{\alp}_{\infty}= \sup_{k}
\abs{\alp_{k}}$ and $\abs{\alp}_{1}= \left( \sum_{k}
  \abs{\alp_{k}}^{2} \right)^{1/2}$. We say that two indices are
adjacent if $\abs{\alp - \alp'}_{\infty} \leq 1$. If $\abs{\alp -
  \alp'}_{1} \leq 1$, we say that $\alp$ and $\alp'$ are
\emph{face-adjacent}; if $\abs{\alp - \alp'}_{\infty} = 1$ but
$\abs{\alp - \alp'}_{1} > 1$, then we say that $\alp$ and $\alp'$ are
\emph{corner-adjacent}. We say that the covering $\set{Q_{\alp}}_{\alp
  \in \Gmm}$ is \emph{good} if the following properties hold:
\begin{enumerate}[label=(\alph*)]
\item The index set $\Gmm$ is of the form $\Gmm = \set{\alp \in
    \bbZ^{d} : \abs{\alp}_{\infty} < R_{\Gmm}}$ for some $R_{\Gmm} >
  0$.
\item For each $\alp$, there exist a sequence of shrinking domains
  $Q_{\alp} = Q^{(0)}_{\alp} \supset Q^{(1)}_{\alp} \supset \cdots$,
  such that, for each $n \geq 0$,
  \begin{equation*}
    Q_{R} \subseteq \bigcup_{\alp \in \Gmm} Q^{(n)}_{\alp}, \qquad
    \overline{Q^{(n+1)}_{\alp}} \cap Q_{R} \subseteq Q^{(n)}_{\alp} \cap Q_{R}. 
  \end{equation*}

\item Two domains $Q^{(n)}_{\alp}$ and $Q^{(n')}_{\alp'}$ intersect if
  and only if their indices are adjacent.
\item Consider any $\alp \in \Gmm$ and a subfamily $\Gmm' \subseteq
  \Gmm$ of adjacent indices with the property that (i) the
  face-adjacent indices in $\Gmm'$ are adjacent to each other and (ii)
  each corner-adjacent index in $\Gmm'$ is adjacent to some
  face-adjacent index in $\Gmm'$. Then for each $n \geq 1$ there
  exists a diffeomorphism $\Phi^{(n)}_{\Gmm'}$ from $Q^{(n)}_{\alp}$
  into $\tilde{F}^{(n)}_{\Gmm'} = \left(\cup_{\alp' \in \Gmm'}
    Q^{(n-1)}_{\alp'}\right) \cap Q^{(n)}_{\alp}$, which equals the
  identity in $F^{(n)}_{\Gmm'} = \left(\cup_{\alp' \in \Gmm'}
    Q^{(n)}_{\alp'}\right) \cap Q^{(n)}_{\alp}$.
\end{enumerate}

Given any cube $Q_{R}$ of sidelength $R > 1$, we construct a good
covering of $Q_{R}$ by round cubes (i.e., with rounded edges, so that
they are smooth) with roughly unit sidelength (more precisely, between
$1/2$ and $4$) as follows. Rescaling by a factor $\simeq 1$ (say
between $1/2$ and $2$), we may assume that $R$ is an
integer. Partition $Q_{R}$ into unit cubes $\tilde{Q}_{\alp}$ with
integer vertices, indexed in an obvious manner by $\Gmm \subseteq
\bbZ^{d}$ as in (a). Rounding off the edges (uniformly in $\alp$), we
may replace each $\tilde{Q}_{\alp}$ by a round cube, such that
$\set{1.1 \tilde{Q}_{\alp}}$ still covers $Q_{R}$. Fix a sequence $2 >
\lmb^{(0)} > \lmb^{(1)} > \cdots > 1.1$, and define $Q^{(n)}_{\alp}$
to be the enlargement $\lmb^{(n)}\tilde{Q}_{\alp}$. It is then
straightforward to verify that (b)--(d) hold for
$\set{Q^{(n)}_{\alp}}$.
\begin{remark}\label{rem:Q-cover}
  We make the simple but crucial observation that the preceding
  construction of a good covering may be \emph{fixed} depending only
  on the size $R$ of the large cube. Also, $Q_{R}$ may be taken to be
  a round cube as well; it does not affect the properties (a)--(d), as
  long as the edges are rounded off at a scale much smaller than $1$.
\end{remark}

Let $\set{Q_{\alp}}_{\alp \in \Gmm}$ be a good covering of $Q_{R}$,
and suppose that a local data set $\set{Q_{\alp}, O_{(\alp \bt)}}$ for
a $\G$-bundle (with arbitrary regularity) is given. Our goal is to
patch the local gauges up to form a global gauge on $Q_{R}$. More
concretely, we find a gauge transformation $P_{(\alp)}$ on each
$Q^{(N)}_{(\alp)}$, where $N = \# \Gmm$, such that
\begin{equation*}
  P_{(\bt)} = P_{(\alp)} \cdot O_{(\alp \bt)} \quad \hbox{ in } Q^{(N)}_{\alp} \cap Q^{(N)}_{\bt}.
\end{equation*}

To start the construction, we endow $\Gmm$ with the lexicographic
ordering (i.e., $\alp < \alp'$ if $\alp_{i} < \alp'_{i}$, where $i$ is
the first index where the components differ); we denote by $[\alp]$
the ordinality of $\alp$ in this covering (thus $1 \leq [\alp] \leq
N$). The simple key observation is that such an ordering insures that
each $\alp$ and $\Gmm' = \set{\alp' : \alp' \hbox{ is adjacent to } \alp, \, \alp' < \alp}$ satisfy the condition of
(d).

We proceed inductively on $[\alp]$, and construct $P_{(\alp)}$ on
$Q^{([\alp])}_{\alp}$ such that
\begin{equation*}
  P_{(\bt)} = P_{(\alp)} \cdot O_{(\alp \bt)} \quad \hbox{ in } Q_{\alp}^{([\bt])} \cap Q_{\bt}^{([\bt])}, \hbox{ for } \alp \leq \bt.
\end{equation*}
For the first element $[\alp] = 1$, we simply take $P_{(\alp)} = Id$
on $Q^{(1)}_{\alp}$. Now assume that $P_{(\alp')}$ has been
constructed on $Q^{([\alp'])}_{\alp'}$ for $\alp' < \alp$, where
$[\alp] = n > 1$. Define $\tilde{P}_{(\alp)}$ in
$\tilde{F}^{(n)}_{\set{\alp' < \alp}} = \left(\cup_{\alp' < \alp}
  Q^{(n-1)}_{\alp'} \right) \cap Q^{(n)}_{\alp}$ by
\begin{equation} \label{eq:patching-tP} \tilde{P}_{(\alp)} =
  P_{(\alp')} \cdot O_{(\alp' \alp)} \quad \hbox{ on }
  Q^{(n-1)}_{\alp'} \cap Q^{(n)}_{\alp} \hbox{ for each } \alp' <
  \alp.
\end{equation}
By construction, these expressions match on the
intersections. Applying (d) in the definition of a good covering, we
find a diffeomorphism $\Phi^{(n)}_{\set{\alp' < \alp}}$ from
$Q^{(n)}_{\alp}$ into $\tilde{F}^{(n)}_{\set{\alp' < \alp}}$, which
equals the identity in $F^{(n)}_{\set{\alp' < \alp}} =
\left(\cup_{\alp' < \alp} Q^{(n)}_{\alp'} \right) \cap
Q^{(n)}_{\alp}$. We simply define $P_{(\alp)}$ in $Q^{(n)}_{\alp}$ by
the pullback
\begin{equation} \label{eq:patching-P} P_{(\alp)} = \tilde{P}_{(\alp)}
  \circ \Phi^{(n)}_{\set{\alp' < \alp}}.
\end{equation}

Next, suppose that local data for a connection $\set{A_{(\alp)}}$ are
also given. Then the gauge potential $A$ in the global gauge
constructed above is described in terms of $A_{(\alp)}$ and
$P_{(\alp)}$ as follows: Given a partition of unity $\chi_{(\alp)}$
subordinate to $\set{Q^{(N)}_{(\alp)}}$, we have
\begin{equation} \label{eq:patching-A} A = \sum \chi_{(\alp)} \left(
    Ad(P_{(\alp)}) A_{(\alp)} - P_{(\alp); x} \right).
\end{equation}

The advantage of this patching procedure is that it relies only on the
properties (a)--(d) of the good covering $\set{Q_{\alp}}_{\alp \in
  \Gmm}$, and is universal in the data $\set{O_{(\alp \bt)}}$ or
$\set{A_{(\alp)}}$. Moreover, it is straightforward to infer
properties of $P_{(\alp)}$ and $A$ from those of $\set{O_{(\alp
    \bt)}}$ and $\set{A_{(\alp)}}$. Indeed, in the above construction,
observe that $\set{P_{(\alp)}}$ is constructed from $\set{O_{(\alp
    \bt)}}$ using only the operations of (i) pointwise multplication,
(ii) pullback by a diffeomorphism, (iii) restriction to a smooth
subdomain and (iv) patching up local expressions which are consistent
on the intersections. Any property of $\set{O_{(\alp \bt)}}$ invariant
under these operations transfers to $P_{(\alp)}$. In particular, for
any $k \geq 1$ and $p \geq \frac{d}{k}$,
\begin{equation*}
  O_{(\alp \bt)} \in \calG^{k, p}_{loc}(Q_{\alp} \cap Q_{\bt}) \ \forall \alp, \bt \imp P_{(\alp)} \in \calG^{k, p}_{loc}(Q^{(N)}_{\alp}) \ \forall \alp.
\end{equation*}
Regarding bounds for $A$, it is useful to introduce the following
definition:
\begin{definition} \label{def:patching-adm} We say that a norm $Y$ on
  $\bbR^{d}$ is \emph{(patching-)admissible} if:
  \begin{itemize}
  \item $Y$ is invariant under pullback by any diffeomorphism;
  \item $Y$ is invariant under any smooth cutoff;
  \item If $A \in Y$ and $O_{;x} \in Y$, then $Ad(O) A \in Y$ with
    $\nrm{Ad(O) A}_{Y} \aleq_{\nrm{A}_{Y}, \nrm{O_{;x}}_{Y}} 1$.
  \end{itemize}
\end{definition}
From the preceding observation regarding the construction of
$P_{(\alp)}$, as well as the explicit formula \eqref{eq:patching-A},
we see that:
\begin{equation*}
  O_{(\alp \bt); x} \in Y(Q_{\alp} \cap Q_{\bt}) \ \forall \alp, \bt \hbox{ and }
  A_{(\alp)} \in Y(Q_{\alp})\ \forall \alp \imp \nrm{A}_{Y(Q_{R})} \aleq 1,
\end{equation*}
where the implicit constant depends only on the good covering (which,
in turn, may be fixed depending only on $R$;
cf. Remark~\ref{rem:Q-cover}), $\sup_{\alp}
\nrm{A_{(\alp)}}_{Y(Q_{\alp})}$ and $\sup_{\alp, \bt} \nrm{O_{(\alp
    \bt)}}_{Y(Q_{\alp} \cap Q_{\bt})}$.

\subsubsection*{Scenario~(2): Large ball covered by small balls}
Here, we wish to patch up local data for a $\G$-bundle and a
connection given on small balls centered inside $B_{R}$; this is the
case we encounter in our applications. The idea is to reduce to
Scenario~(1) by a suitable diffeomorphism.

Consider a covering $\set{B_{\alp} \cap B_{R}}$ of $B_{R}$ by finitely
many balls.  Let $\Phi$ be a bi-Lipschitz isomorphism from the cube
$Q_{\lmb_{0} R}$ to $B_{R}$, where $\lmb_{0} \in (0, \infty)$ is to be
fixed below. Let $\set{Q_{\alp}}_{\alp \in \Gmm}$ be a good covering
of $Q_{\lmb_{0} R}$ as in Scenario~(1). We wish to insure that the
image of each $Q_{\alp}$ under $\Phi_{R}$ is contained in a unit
ball. Indeed, observe that, by scaling-invariance, the Lipschitz
constant of $\Phi_{R}$ is independent of $R$, but decreases in
$\lmb_{0}$. Hence, for any $\dlt > 0$, by choosing $\lmb_{0}$
sufficiently large (independent of $R$) we may insure that
\begin{equation} \label{eq:cube-ball} \Phi_{R}(Q_{\alp}) \subseteq
  B_{\dlt}(x) \quad \hbox{ for some } x \in B_{R}.
\end{equation}
By Lebesgue's covering lemma, this ensures that $\Phi(Q_{\alp})$ is
contained in some ball $B_{\alp}$ in the covering.  Finally, by
rounding off the edges of $Q_{\lmb_{0} R}$, we may replace
$Q_{\lmb_{0} R}$ by a round cube, and $\Phi$ by a diffeomorphism with
uniform bounds. Note that this can be done while not disturbing the
Lipschitz constant much (and thus \eqref{eq:cube-ball} still holds),
while the uniform bounds of higher derivatives would depend on $R$.

\begin{remark}\label{rem:B-cover}
  In the above procedure, note that $\lmb_{0}$ depends only on
  Lebesgue constant $\dlt > 0$ of the covering $\set{B_{\alp}}$. In
  particular, if $B_{\alp}$'s are unit balls which are uniformly
  separated, so that the Lebesgue constant is $\aeq 1$, $\lmb_{0}$ may
  be fixed independent of $R$. The remaining components of the
  construction may be \emph{fixed} depending only on the radius $R$
  (recall also Remark~\ref{rem:Q-cover}).
\end{remark}
We now apply the patching procedure in Scenario~(1) to the pulled-back
data $\set{Q_{\alp}, O_{(\alp \bt)} \circ \Phi, \Phi^{\ast}
  A_{(\alp)}}$, which are well-defined since each $\Phi(Q_{\alp})$ is
contained in some ball $B_{\alp}$ in the covering. Then we return to
$B_{R}$ via $\Phi^{-1}$. As a result, we obtain a refinement
$\set{B'_{\alp} = \Phi(Q^{(N)}_{\alp})}$ of the covering
$\set{B_{\alp}}$ (the index sets are different, but we abuse the
notation and denote both by $\alp$), as well as a gauge transform
$P_{\alp'}$ in each $B'_{\alp'}$, such that
\begin{equation} \label{eq:patching-P-ball} P_{(\alp)} = P_{(\alp')}
  \cdot O_{(\alp' \alp)} \quad \hbox{ on } B'_{\alp'} \cap B'_{\alp}.
\end{equation}
Moreover, given a partition of unity subordinate to $\set{B'_{\alp}}$,
the global gauge potential $A$ takes the form
\begin{equation} \label{eq:patching-A-ball} A = \sum \chi_{(\alp)}
  \left(Ad(P_{(\alp)} A_{(\alp)} - P_{(\alp); x} \right).
\end{equation}
Finally, we obtain the following result:
\begin{proposition} \label{prop:patching-ball} Let $R \geq 1$, and
  consider a covering $\set{B_{\alp} \cap B_{R}}$ of $B_{R}$ by
  uniformly separated unit balls $B_{\alp}$ centered inside
  $B_{R}$. Any $\G$-bundle with $O_{(\alp \bt)} \in \calG^{k,
    \frac{d}{k}}(B_{\alp} \cap B_{\bt} \cap B_{R})$ admits a global
  gauge. Moreover, given any local data $\set{A_{(\alp)}}$ for a
  connection on this $\G$-bundle satisfying $A \in W^{k-1,
    \frac{d}{k}}(B_{\alp} \cap B_{R})$, the global gauge potential
  satisfies $A \in W^{k, \frac{d}{k}}(B_{R})$. More precisely, if
  \begin{equation*}
    \sup_{\alp} \nrm{A_{(\alp)}}_{W^{k-1, \frac{d}{k}}(B_{\alp} \cap B_{R})} \leq M, \quad
    \sup_{\alp, \bt} \nrm{O_{(\alp \bt); x}}_{W^{k-1, \frac{d}{k}}(B_{\alp} \cap B_{\bt} \cap B_{R})} \leq M, 
  \end{equation*}
  for some $M > 0$, then
  \begin{equation*}
    \nrm{A}_{W^{k-1, \frac{d}{k}}(B_{R})} \aleq_{R, M} 1.
  \end{equation*}
\end{proposition}

\subsubsection*{Scenario~(3): $X = B_{R}$ or $\bbR^{d}$ covered by concentric balls}
Finally, we consider the case when local data for a $\G$-bundle and a
connection are given all concentric balls $\set{B_{R_{n}}}_{n= 1, 2,
  \ldots}$ with $R_{n} \nearrow R$ or $\infty$.

Add a smaller ball $B_{R_{0}} \subset B_{R_{1}}$ to the covering. For
$n \geq 2$, let $\Phi_{n}$ be a diffeomorphism from $B_{R_{n}}$ into
$B_{R_{n-1}}$, which equals the identity on $B_{R_{n-2}}$. Define
$P_{(n)}$ on $B_{R_{n}}$ inductively by $P_{(1)} = id$ and
\begin{equation*}
  P_{(n)} = (P_{(n-1)} \cdot O_{((n-1) n)}) \circ \Phi_{n}.
\end{equation*}
Then we restrict the data and $P_{(n)}$ on $B_{R_{n}}$ to
$B_{R_{n-1}}$. It follows by construction that, for $n < m$,
\begin{equation*}
  P_{(m)} = P_{(n)} \cdot O_{(n m)} \quad \hbox{ in } B_{R_{n-1}}.
\end{equation*}
Given some local data $\set{A_{(n)}}$ for a connection, the global
gauge potential is given by
\begin{equation*}
  A = Ad(P_{(n)}) A_{(n)} - P_{(n); x} \quad \hbox{ in } B_{R_{n-1}}.
\end{equation*}
These expressions are consistent in the intersection (i.e., the
smaller ball).  Again, observe that $P_{(n)}$ is constructed by the
same operations (i)--(iv) as in Scenario~(1).

As a consequence this patching procedure, as well as
Proposition~\ref{prop:patching-ball}, we obtain the following soft
result, which is a starting point for the good global gauge theorems.
\begin{proposition} \label{prop:ball-triv} Any $\G$-bundle with
  regularity $\calG^{k, \frac{d}{k}}_{loc}$ on $X = B_{R}$ or
  $\bbR^{d}$ admits a global gauge.  Moreover, for any $\covD \in
  \calA^{k-1, \frac{d}{k}}_{loc}(X)$ on this $\G$-bundle, the global
  gauge potential satisfies $A \in W^{k-1, \frac{d}{k}}_{loc}(X)$.
\end{proposition}
\begin{proof}
  Let $\set{U_{\alp}, O_{(\alp \bt)}}$ be the local data for a
  $\G$-bundle with regularity $\calG^{k, \frac{d}{k}}_{loc}$ on $X =
  B_{R}$ or $\bbR^{d}$, and consider a smaller ball $B_{R'}$ such that
  $\overline{B_{R'}} \subseteq X$. By Lebesgue's covering lemma, there
  exists a refinement of $U_{\alp}$ by balls $\set{B_{\dlt}(x) \cap
    B_{R'}}_{x \in B_{R'}}$ of the same radius $\dlt > 0$ . By
  Proposition~\ref{prop:patching-ball}, we obtain a global gauge on
  $B_{R'}$. Since $R'$ is arbitrary, Scenario~(3) applies to a
  sequence of global gauges on $B_{R'}$ with $R' \nearrow R$ or
  $\infty$, and we obtain a global gauge on $X$. Existence of a
  corresponding global gauge potential for any $\calA^{k-1,
    \frac{d}{k}}_{loc}(X)$ connection is a quick corollary. \qedhere
\end{proof}

\subsection{Uhlenbeck lemmas and elliptic regularity}
Thanks to Proposition~\ref{prop:ball-triv}, we know that any
$\calA^{1, \frac{d}{2}}_{loc}(X)$ connection admits a global gauge
potential in $W^{1, \frac{d}{2}}_{loc}(X)$. This is a natural setting
for Uhlenbeck's lemma, which finds good local gauges under a
gauge-invariant smallness assumption. These good local gauges furnish
another main ingredient of the proof of the good global gauge
theorems.

We start with the case of a ball $B_{1}$.
\begin{theorem} [Uhlenbeck's lemma on a
  ball] \label{thm:uhlenbeck-ball} Consider $\covD \in \calA^{1,
    \frac{d}{2}}_{loc}(B_{1})$ of the form $\covD = \ud + A$ with $A
  \in W^{1, \frac{d}{2}}(B_{1}; \g)$, which satisfies
  \begin{equation} \label{eq:uhlenbeck-ball-hyp}
    \nrm{F[A]}_{L^{\frac{d}{2}}(B_{1})} < \eps_{0}.
  \end{equation}
  \begin{enumerate}
  \item There exists $O \in \calG^{2, \frac{d}{2}}(B_{1})$, unique up
    to multiplication by a constant element of $\G$, such that $\tA =
    Ad(O) A - O_{;x} \in W^{1, \frac{d}{2}}(B_{1}; \g)$ obeys
    \begin{equation*}
      \rd^{\ell} \tA_{\ell} = 0 \hbox{ in } B_{1}, \qquad
      x^{\ell} \tA_{\ell} = 0 \hbox{ on } \rd B_{1}
    \end{equation*}
    and
    \begin{equation*}
      \nrm{\tA}_{W^{1, \frac{d}{2}}(B_{1})} \aleq \nrm{F[A]}_{L^{\frac{d}{2}}(B_{1})}.
    \end{equation*}

  \item Let $A^{n}$ be a sequence of connections such that $A^{n} \to
    A$ in $W^{1, \frac{d}{2}}(B_{1}; g)$. Let $(\tA^{n}, O^{n})$ be
    given by (1) from $A^{n}$. Then passing to a subsequence and
    suitably conjugating each $(\tA^{n}, O^{n})$ with a constant gauge
    transformation, we have
    \begin{equation*}
      \tA^{n} \to \tA \hbox{ in } W^{1, \frac{d}{2}}(B_{1}), \qquad
      O^{n} \to O \hbox{ in } W^{2, \frac{d}{2}}(B_{1}).
    \end{equation*}
  \end{enumerate}
\end{theorem}
\begin{proof}
  For a proof of the existence claim in (1), see
  \cite[Theorem~1.3]{MR648356}. For uniqueness, observe that the gauge
  transformation $\tilde{O} \in \calG^{2, \frac{d}{2}}(B_{1})$ between
  the two possible $\tA$ and $\tA'$ satisfies the a-priori bound
  $\nrm{\tilde{O}_{;x}}_{W^{1, \frac{d}{2}}(B_{1})} \aleq \eps_{0}$,
  and also solves the div-curl system
  \begin{equation*}
    \rd^{\ell} \tilde{O}_{;\ell} = Ad(\tilde{O}) [\tilde{O}_{;\ell}, (\tA')^{\ell}], \qquad
    \rd_{j} \tilde{O}_{;k} - \rd_{k} \tilde{O}_{;j} = - [\tilde{O}_{;j}, \tilde{O}_{;k}],
  \end{equation*}
  with the boundary condition $x^{\ell} \tilde{O}_{;\ell} = 0$ on $\rd
  B_{1}$. It follows that $\tilde{O}_{;x} = 0$, i.e., $\tilde{O}$ is
  constant.

  To prove $(2)$, observe first that the $W^{2, \frac{d}{2}}(B_{1})$
  norm of $O^{n}$ is uniformly bounded, thanks to the formula
  $O^{n}_{;x} = Ad(O^{n}) A^{n} - \tA^{n}$. Thus, after passing to a
  subsequence, $O^{n} \weakto O'$ and $\tA^{n} \weakto \tA'$ in $W^{2,
    \frac{d}{2}}(B_{1})$ and $W^{1, \frac{d}{2}}(B_{1})$,
  respectively. This weak convergence is enough to justify
  \begin{equation*}
    \tA' = Ad(O') A - O'_{;x} \hbox{ in } B_{1}, \quad \rd^{\ell} \tA_{\ell} = 0 \hbox{ in } B_{1}, \quad x^{\ell} \tA_{\ell} = 0 \hbox{ on } \rd B_{1}.
  \end{equation*}
  Hence, by the uniqueness statement in (1), $(\tA', O')$ coincides
  with $(\tA, O')$ up to a constant gauge transformation
  $O_{0}$. Applying $O_{0}$ to the sequence $(\tA^{n}, O^{n})$, we may
  insure that $O^{n} \weakto O$ and $\tA^{n} \weakto \tA$ in $W^{2,
    \frac{d}{2}}(B_{1})$ and $W^{1, \frac{d}{2}}(B_{1})$,
  respectively.

  To upgrade the weak convergence to strong convergence, we use the
  div-curl system for $\tA$. First, by the strong $W^{1, \frac{d}{2}}$
  convergence $A^{n} \to A$ and the weak $W^{2, \frac{d}{2}}$
  convergence $O^{n} \to O$, it follows that
  \begin{equation*}
    F[\tA^{n}] = Ad(O^{n}) F[A^{n}] \to Ad(O) F[A] = F[\tA] \quad \hbox{ in } L^{\frac{d}{2}}(B_{1}).
  \end{equation*}
  Then by the div-curl system
  \begin{equation*}
    \rd^{\ell} \tA_{\ell} = 0, \qquad 
    \rd_{j} \tA_{k} - \rd_{k} \tA_{j} = F[\tA^{n}],
  \end{equation*}
  the weak $W^{1, \frac{d}{2}}$ convergence $\tA^{n} \to \tA$ is
  improved to strong convergence. Finally, by the formula $O_{;x} =
  Ad(O) A - \tA$, the weak $W^{2, \frac{d}{2}}$ convergence $O^{n} \to
  O$ is also improved to strong convergence. \qedhere
\end{proof}

Theorem~\ref{thm:uhlenbeck-ball} was extended in \cite{MR815194} to a
``removal of singularity'' result for connections defined only on a
punctured ball. Let $B'_{r} = \set{x \in \bbR^{d} : 0 < \abs{x} < r}$.

\begin{theorem} [Uhlenbeck's lemma on a punctured
  ball] \label{thm:uhlenbeck-pball} Consider $\covD \in \calA^{1,
    \frac{d}{2}}_{loc}(B_{1+\dlt}')$ for some $\dlt > 0$, which admits
  a representative $\covD = \ud + A$ with $A \in W^{1,
    \frac{d}{2}}_{loc}(B'_{1+\dlt}; \g)$ and satisfies
  \begin{equation*}
    \nrm{F[A]}_{L^{\frac{d}{2}}(B_{1}')} \leq \eps_{0}'.
  \end{equation*}
  Then there exists $O \in \calG^{2, \frac{d}{2}}_{loc}(B_{1}')$ such
  that $\tA = Ad(O) A - O_{;x}$ obeys
  \begin{equation*}
    \rd^{\ell} \tA_{\ell} = 0 \quad \hbox{ in } B_{1}',
  \end{equation*}
  and
  \begin{equation*}
    \nrm{\tA}_{W^{1, \frac{d}{2}}(B_{1}')} \aleq \nrm{F[A]}_{L^{\frac{d}{2}}(B_{1}')}.
  \end{equation*}
\end{theorem}
As a consequence, we see that $\tA$ is the restriction of a $\calA^{1,
  \frac{d}{2}}_{loc}$ connection on the full ball $B_{1+\dlt}$. For a
proof, we refer the reader to \cite{MR815194}.

If $F$ satisfies higher (covariant) regularity bounds, then so does
$\tA$ in the above theorems. This statement is most naturally
formulated as an elliptic regularity result for the nonlinear div-curl
system satisfied by $\tA$ with $\rd^{\ell} \tA_{\ell} = 0$. In what
follows, we omit the tilde for simplicity, and we focus on
quantitative bounds in scaling-invariant spaces.

We start with a simple interior regularity result.
\begin{lemma} \label{lem:div-curl-A-intr} Let $A \in W^{1,
    \frac{d}{2}}(B)$ be a solution to the nonlinear div-curl system
  \begin{equation} \label{eq:div-curl-A}
    \begin{aligned}
      \rd_{j} A_{k} - \rd_{k} A_{j} =& F_{jk} - [A_{j}, A_{k}], \\
      \rd^{\ell} A_{\ell} =& 0.
    \end{aligned}
  \end{equation}
  If $\covD^{(m)} F \in L^{\frac{d}{m+2}}(B)$ with $\frac{d}{m+2} >
  1$, then $\rd^{(m+1)} A \in L^{\frac{d}{n+2}}(\lmb B)$ for any $0
  \leq \lmb < 1$, with a bound depending only on $m$,
  $\nrm{\covD^{(m)} F}_{L^{\frac{d}{m+2}}(B)}$, $\nrm{A}_{L^{d}(B)}$
  and $\lmb$.
\end{lemma}
\begin{proof}
  Since it is a straightforward interior elliptic regularity argument,
  we only sketch the proof.  We proceed by a simple induction on $m$;
  the key point is that $\rd^{(m)} F_{jk}$ and $\rd^{(m)} [A_{j},
  A_{k}]$ in $L^{\frac{d}{m+2}}$ are controlled by $\covD^{(m)} F$ in
  $L^{\frac{d}{m+2}}$ and the inductive bounds for $\rd^{(m'+1)} A$ in
  $L^{\frac{d}{m'+2}}$ ($0 \leq m' \leq m$).
\end{proof}

When Theorem~\ref{thm:uhlenbeck-ball} is applied to a unit ball
$B_{1}(x)$ centered near the boundary $\rd B_{R}$ of a larger ball, it
is of interest to control regularity of $A$ up to the boundary $\rd
B_{R}$. For this purpose, consider normalized angular derivatives
$\srd = \set{\frac{1}{\abs{x}} (x_{j} \rd_{k} - x_{k} \rd_{j})}$ about
the origin (at which $B_{R}$ is centered), and the corresponding
covariant angular derivatives $\scovD = \set{\frac{1}{\abs{x}} (x_{j}
  \covD_{k} - x_{k} \covD_{j})}$. In any unit ball away from the
origin, we show that higher angular regularity of $F$ implies the
corresponding regularity of $A$ in the Coulomb gauge.
\begin{lemma} \label{lem:div-curl-A-tang} Let $B$ be a unit ball in
  $\bbR^{d}$ such that $B \cap B_{1}(0) = \0$, and let $A \in W^{1,
    \frac{d}{2}}(B)$ be a solution to the nonlinear div-curl system
  \eqref{eq:div-curl-A}.  If $\scovD^{(m)} F \in L^{\frac{d}{m+2}}(B)$
  with $\frac{d}{m+2} > 1$, then $\rd \srd^{(m)} A \in
  L^{\frac{d}{m+2}}(\lmb B)$ for any $0 \leq \lmb < 1$, with a bound
  depending only on $m$, $\nrm{\scovD^{(m)}
    F}_{L^{\frac{d}{m+2}}(B)}$, $\nrm{A}_{W^{1, \frac{d}{2}}(B)}$ and
  $\lmb$.
\end{lemma}
\begin{proof}
  This lemma is most simply proved by commuting with the Lie
  derivatives with respect to the normalized rotation vector fields
  $\overline{\Omg}_{jk} = \frac{1}{d(0, B)} \Omg_{jk}$; these are
  isometries and thus exactly commute with the div-curl
  system. Moreover, their lengths are comparable to $1$ (independent
  of $B$), so that $\abs{\calL^{(\leq n)}_{\overline{\Omg}} A}
  \aeq_{n} \abs{\srd^{(\leq n)} A}$.

  As before, when $p = \frac{d}{n+2} > 1$, the statement follows (with
  explicit bounds) by an induction on $n$. By the trace theorem and
  the (angular) Sobolev inequality, observe that
  \begin{equation*}
    \nrm{u}_{L^{\infty}_{r} L^{\frac{d-1}{m+1}}_{\Tht}(B)} \aleq \nrm{u}_{L^{\frac{d}{m+2}}(B)} + \nrm{\rd u}_{L^{\frac{d}{m+2}}(B)}.
  \end{equation*}
  Using this inequality and H\"older, we may control
  $\calL_{\overline{\Omg}}^{(n)} F$ and $\overline{\Omg}^{(n)} [A_{j},
  A_{k}]$ in $L^{\frac{d}{n+2}}$ by $\scovD^{(\leq n)} F$ in
  $L^{\frac{d}{n+2}}$ and the inductive bounds for $\rd \srd^{(\leq
    m)} A$ in $L^{\frac{d}{m+2}}$. Then we may proceed as in the proof
  of Lemma~\ref{lem:div-curl-A-intr}.
\end{proof}

\begin{remark} \label{rem:uhlenbeck-cont} As in
  Theorem~\ref{thm:uhlenbeck-ball}(2), an argument similar to
  Lemma~\ref{lem:div-curl-A-intr}
  (resp. Lemma~\ref{lem:div-curl-A-tang}) for the div-curl system for
  $\tA$ leads to strong convergence of $\rd^{(\leq m+1)} \tA^{n}$ and
  $\rd^{(\leq m+2)} O^{n}$ in $L^{\frac{d}{n+2}}(\lmb B)$
  (resp. $\rd^{(\leq m+1)} \tA^{n}$ and $\rd^{(\leq 2)} \srd^{(\leq
    m)} O^{n}$ in $L^{\frac{d}{n+2}}(\lmb B \cap B_{R})$), provided
  that $A^{n} \to A$ in $W^{m, \frac{d}{m+1}}$. We omit the
  straightforward proof.
\end{remark}

Next, we record a simple interior regularity result for the div-curl
system of $O$.
\begin{lemma} \label{lem:div-curl-O} Let $O \in W^{2, \frac{d}{2}}(B)$
  be a solution to the div-curl system
  \begin{equation} \label{eq:div-curl-O}
    \begin{aligned}
      \rd_{j} O_{;k} - \rd_{k} O_{;j} =& [O_{;j}, O_{;k}] \\
      \rd^{\ell} O_{; \ell} = & H.
    \end{aligned}
  \end{equation}
  If $H \in \ell^{1} L^{\frac{d}{2}}(B)$, then $O_{;x} \in \ell^{1}
  W^{1, \frac{d}{2}}(\lmb B)$ for any $0 \leq \lmb < 1$, with the
  bound
  \begin{equation*}
    \nrm{O_{;x}}_{\ell^{1} \dot{W}^{1, \frac{d}{2}}(\lmb B)} \aleq_{\lmb} \nrm{H}_{\ell^{1} L^{\frac{d}{2}}(B)} + \nrm{O_{;x}}_{W^{1, \frac{d}{2}}(B)}^{2}.
  \end{equation*}
  Moreover, if $(O', H') \in W^{2, \frac{d}{2}}(B) \times \ell^{1}
  L^{\frac{d}{2}}(B)$ is another solution to \eqref{eq:div-curl-O},
  then
  \begin{equation*}
    \nrm{O_{;x} - O'_{;x}}_{\ell^{1} \dot{W}^{1, \frac{d}{2}} (\lmb B)} \aleq_{\lmb} \nrm{H-H'}_{\ell^{1} L^{\frac{d}{2}}(B)} +( \nrm{O_{;x}}_{W^{1, \frac{d}{2}}(B)} + \nrm{O'_{;x}}_{W^{1, \frac{d}{2}}(B)}) \nrm{O_{;x} - O'_{;x}}_{W^{1, \frac{d}{2}}(B)}.
  \end{equation*}
\end{lemma}
The key point is that $[O_{;j}, O_{;k}]$ in $\ell^{1}
L^{\frac{d}{2}}(B)$ can be estimated by $O_{;j}, O_{;k}$ in $W^{1,
  \frac{d}{2}}(B)$. We omit the obvious proof.

The $\ell^{1} \dot{W}^{1, \frac{d}{2}}$ bound on $O_{;x}$ is useful as
it implies continuity of $O$. More precisely, we have the following:
\begin{lemma} \label{lem:ptwise-O} If $O_{;x} \in \ell^{1} W^{1,
    \frac{d}{2}}(B)$, then $O$ is continuous on $B$.
\end{lemma}
\begin{proof}
  Without loss of generality, let $x_{1}$ be farther away from $\rd B$
  than $x_{2}$. As in the proof of Morrey's inequality, we have
  \begin{equation*}
    d(O(x_{1}), O(x_{2})) \aleq \int_{B(x_{1}, 2r)} \frac{\abs{O_{;x}}}{\abs{x - x_{1}}^{d-1}} + \frac{\abs{O_{;x}}}{\abs{x - x_{2}}^{d-1}} \, \ud x.
  \end{equation*}
  The last integral may be estimated in terms of the Besov norm of the
  extension of $O_{;x}$, and vanishes as $x_{1} \to x_{2}$. \qedhere
\end{proof}

\subsection{Good global gauge theorem on the ball}
The goal of this subsection is to prove
Theorem~\ref{thm:goodrep-ball}. The overall proof is divided into two
steps:
\begin{itemize}
\item First, we prove the quantitative statements under the assumption
  that $\covD$ admits a global gauge potential $A \in \dot{W}^{1,
    \frac{d}{2}}(B_{R})$.
\item Next, using softer arguments, we remove the global gauge
  assumption.
\end{itemize}

In the first step, the idea is to produce local gauges on balls
$B_{1}(x)$ centered inside $B_{R}$ using Uhlenbeck's lemma, and then
patch them up to a global gauge on $B_{R}$. To handle balls near the
boundary, the following simple extension procedure is helpful.
\begin{lemma} \label{lem:ext-simple} Let $A \in W^{1,
    \frac{d}{2}}(B_{R})$ with $A_{r} = 0$ on $\rd B_{R}$. Extend $A$
  outside $B_{R}$ by
  \begin{equation*}
    \bar{A}_{r} \left( \frac{R^{2}}{r}, \Tht \right)= - A_{r}(r, \Tht), \quad
    \bar{A}_{\Tht} \left( \frac{R^{2}}{r}, \Tht \right)= A_{\Tht} (r, \Tht).
  \end{equation*}
  Then the extension obeys
  \begin{equation} \label{eq:ext-simple} F[\bar{A}]
    \left(\frac{R^{2}}{r}, \Tht\right) = F[A](r, \Tht) \quad \hbox{
      for } r < R.
  \end{equation}
\end{lemma}
The proof is an easy algebra computation, which we omit.
%
%
We now carry out the first step.
\begin{proposition} \label{prop:goodrep-ball-key}
  Theorem~\ref{thm:goodrep-ball} holds under the additional assumption
  that $A \in \dot{W}^{1, \frac{d}{2}}(B_{R})$.
\end{proposition}
\begin{proof}
  By rescaling, we may set $r = 1$, i.e., $\rcsp^{\eps_\ast}[A] \geq
  1$. Then we need to show that \eqref{eq:goodrep-ball-est} holds with
  an implicit constant depending only on $\eps_{\ast}$ and $R$,
  provided that $\eps_{\ast}$ is sufficiently small compared to a
  universal constant.

  If $R \aleq 1$, then the conclusion of
  Theorem~\ref{thm:goodrep-ball} follows by Uhlenbeck's lemma, so we
  may assume that $R > 10$ (say). Applying Lemma~\ref{lem:Ar=0}, we
  may assume, without loss of generality that $A_{r} = 0$. Then we
  extend $A$ outside $B_{R}$ via Lemma~\ref{lem:ext-simple}. By
  \eqref{eq:ext-simple}, it follows that the extended connection still
  has concentration radius $\ageq 1$ in $B_{R+10}$. Choosing
  $\eps_{\ast}$ sufficiently small, we may insure that Uhlenbeck's
  lemma applies to the extended connection on balls of radius $2$
  centered in $B_{R}$.

  Consider a covering $\set{B_{\alp}}$ of $B_{R}$ by uniformly
  separated unit balls centered in $B_{R}$, and apply Uhlenbeck's
  lemma on each $2 B_{\alp}$ to obtain local data $A_{(\alp)} \in
  W^{1, \frac{d}{2}}(2 B_{\alp})$ and $O_{(\alp \bt)} \in \calG^{2,
    \frac{d}{2}}(2 B_{\alp} \cap 2 B_{\bt})$. By
  Lemma~\ref{lem:div-curl-A-intr}, we see that $A_{(\alp)}$ enjoys
  higher regularity properties in each interior ball $B_{\alp}$ (i.e.,
  $2 B_{\alp} \cap \rd B_{R} = \0$). For a boundary ball $B_{\alp}$,
  i.e., $2 B_{\alp} \cap \rd B_{R} \neq \0$, we first obtain higher
  angular regularity of $A$ in $B_{\alp} \cap B_{R}$ by
  Lemma~\ref{lem:div-curl-A-tang}, and then also regularity in the
  radial direction by the equations
  \begin{equation} \label{eq:div-curl-rad} \rd_{r} A_{r} = -
    \mathrm{div}_{\Tht} A_{\Tht}, \qquad \rd_{r} A_{\Tht} = \rd_{\Tht}
    A_{r} + [A_{r}, A_{\Tht}] + F_{r \Tht},
  \end{equation}
  as well as radial covariant derivative bounds on $F_{r
    \Tht}$. Finally, observe that the desired higher regularity of
  $O_{(\alp \bt)}$ in $B_{\alp} \cap B_{\bt} \cap B_{R}$ follows from
  the equation $O_{(\alp \bt); x} = Ad(O_{(\alp \bt)}) A_{(\bt)} -
  A_{(\alp)}$ and the bounds for $A_{(\alp)}, A_{(\bt)}$.

  As a result, on the covering $\set{B_{\alp} \cap B_{R}}$, we obtain
  local data $O_{(\alp \bt)} \in W^{k, \frac{d}{k}}(B_{\alp} \cap
  B_{\bt} \cap B_{R})$ and $A_{(\alp)} \in W^{k,
    \frac{d}{k}}(B_{(\alp)} \cap B_{R})$, provided that $\covD^{(k)} F
  \in L^{\frac{d}{k}}$ (with $k \geq 1$, $\frac{d}{k} > 1$). We are in
  a position to apply Proposition~\ref{prop:patching-ball}, from which
  the conclusion of Theorem~\ref{thm:goodrep-ball} follows.
\end{proof}

Finally, we remove the global gauge assumption, and thereby complete
the proof of Theorem~\ref{thm:goodrep-ball}.

\begin{proof}[Completion of proof of Theorem~\ref{thm:goodrep-ball}]
  Consider a sequence $R_{n} \upto R$. Apply
  Proposition~\ref{prop:goodrep-ball-key} to each $A
  \restriction_{B_{R_{n}}}$, which gives rise to $\tA^{(n)}$ and
  $O^{(n)}$ such that
  \begin{align*}
    O^{(n)}_{;j} =& Ad(O^{(n)}) A_{j} - \tA^{(n)}_{j} \\
    \rd_{k} O^{(n)}_{;j} =& [O^{(n)}_{;k}, Ad(O^{(n)}) A_{j}] +
    Ad(O^{(n)}) \rd_{k} A_{j} - \rd_{k} \tA^{(n)}_{j}
  \end{align*}
  It follows that $O^{(m)}_{;x}$ is uniformly bounded in $W^{1,
    \frac{d}{2}}$ on each fixed $B_{R'}$.  Therefore, after passing to
  a subsequence, there exists $O \in W^{2, \frac{d}{2}}_{loc}(B_{R};
  \bbR^{N \times N})$ such that $O^{(n)} \weakto O$ in $W^{2,
    \frac{d}{2}}(B_{R'}; \bbR^{N \times N})$ for every $0 < R' < R$
  and $O^{(n)} \to O$ a.e. on $B_{R}$. Hence, $O \in \calG^{2,
    \frac{d}{2}}_{loc}(B_{R})$ and moreover
  \begin{equation*}
    \tA_{j} = Ad(O) A_{j} - O_{;j}
  \end{equation*}
  is the weak limit of $\tA^{(n)}$ in $W^{1,
    \frac{d}{2}}_{loc}$. Since the $\dot{W}^{1,
    \frac{d}{2}}(B_{R_{n}})$ norm of $\tA^{(n)}$ is uniformly bounded
  in $n$, it follows that $\nrm{\tA}_{\dot{W}^{1, \frac{d}{2}}(B_{R})}
  \aleq_{\nrm{F}_{L^{\frac{d}{2}}}} 1$.
\end{proof}

\subsection{Good global gauge theorem on the whole space}
Next, we establish Theorem~\ref{thm:goodrep}.
\begin{proof}[Proof of Theorem~\ref{thm:goodrep}]
  By rescaling, we set $\Rcsp = 1$. Throughout this proof, we work
  with global gauge potentials in $W^{1, \frac{d}{2}}_{loc}(\bbR^{d})$
  for $\covD$, which exists thanks to
  Proposition~\ref{prop:ball-triv}.

  The first main task is to find a good gauge in a suitable exterior
  domain. By hypothesis, and our normalization $\Rcsp = 1$, we have
  $\nrm{F[A]}_{L^{\frac{d}{2}}(\bbR^{d} \setminus \overline{B})} <
  \eps_{\ast}$. Consider the inversion map
  \begin{equation*}
    \iota : x \mapsto y = \frac{x}{\abs{x}^{2}},
  \end{equation*}
  which clearly satisfies $\iota \circ \iota = id$. Under $\iota$, the
  exterior region $\bbR^{d} \setminus \overline{B}$ is the image of
  the punctured unit ball $B'$, and vice versa. The map $\iota$ is a
  conformal isometry, such that
  \begin{equation*}
    (\iota^{\ast} \dlt)^{ij} = \abs{x}^{4} \dlt^{ij}, \quad \iota^{\ast}(\ud y^{1} \wedge \cdots \wedge \ud y^{d}) = \frac{(-1)^{d}}{\abs{x}^{2d}} \ud x^{1} \wedge \cdots \wedge \ud x^{d}.
  \end{equation*}
  In particular, if $T$ is a covariant $2$-tensor on $\iota(U)
  \subseteq \bbR^{d}$, then
  \begin{equation*}
    \int_{\iota(U)} (\sum_{i, j} \abs{T_{y^{i} y^{j}}}^{2} )^{\frac{d}{4}}(y) \, \ud y = \int_{U} (\sum_{i, j} \abs{\iota^{\ast} T_{x^{i} x^{j}}}^{2} )^{\frac{d}{4}}(x) \, \ud x.
  \end{equation*}
  Choosing $\eps_{\ast} < \eps_{0}'$, we have $\nrm{\iota^{\ast}
    F}_{L^{\frac{d}{2}}(B')} = \nrm{F}_{L^{\frac{d}{2}}(\bbR^{d}
    \setminus \overline{B}} < \eps_{0}'$, and we may apply
  Theorem~\ref{thm:uhlenbeck-pball} to find a local gauge in which the
  gauge potential satisfies $\tilde{A}_{(\infty)} \in \dot{W}^{1,
    \frac{d}{2}}(B)$. We define $A_{(\infty)}$ to be the local gauge
  potential of $\covD = \iota^{\ast} \iota^{\ast} \covD$ on $\bbR^{d}
  \setminus \overline{B}$ given by $A_{(\infty)} = \iota^{\ast}
  \tilde{A}_{(\infty)}$. Since $\rd (\iota^{\ast}
  \tilde{A}_{(\infty)}) = \iota^{\ast} (\rd \tilde{A}_{(\infty)})$ and
  $\nrm{\iota^{\ast} \rd
    \tilde{A}_{(\infty)}}_{L^{\frac{d}{2}}(\bbR^{d} \setminus
    \overline{B})} = \nrm{\rd
    \tilde{A}_{(\infty)}}_{L^{\frac{d}{2}}(B')}$, it follows that
  $A_{(\infty)} \in \dot{W}^{1, \frac{d}{2}}(\bbR^{d} \setminus
  \overline{B})$ and
  \begin{equation} \label{eq:goodrep-extr} \nrm{A_{(\infty)}}_{L^{d}
      \cap \dot{W}^{1, \frac{d}{2}}(\bbR^{d} \setminus \overline{B})}
    \aleq \eps_{\ast}.
  \end{equation}

  On the other hand, by Theorem~\ref{thm:goodrep-ball} applied to
  $5B$, we obtain a local gauge potential $A_{(0)} \in $ for such that
  \begin{equation} \label{eq:goodrep-intr} \nrm{A_{(0)}}_{L^{d} \cap
      \dot{W}^{1, \frac{d}{2}}(5B)} \aleq_{\eps_{\ast}, \rcsp^{-1}} 1.
  \end{equation}
  
  By construction there exists $O \in \calG^{2, \frac{d}{2}}_{loc}(5B
  \setminus \overline{B})$ such that
  \begin{equation*}
    A_{(0)} = Ad(O) A_{(\infty)} - O_{;x} \quad \hbox{ in } 5B \setminus \overline{B}.
  \end{equation*}
  By this relation, \eqref{eq:goodrep-extr} and
  \eqref{eq:goodrep-intr}, on $5B \setminus \overline{B}$ we have
  \begin{equation*}
    \nrm{O_{;x}}_{L^{d} \cap \dot{W}^{1, \frac{d}{2}}(5B \setminus \overline{B})}
    \aleq_{\eps_{\ast}, \rcsp^{-1}} 1.
  \end{equation*}
%
%
  Using the partial approximation lemma (Lemma~\ref{lem:part-approx})
  and performing $0$-homogeneous extension outside a suitable sphere,
  it is straightforward to construct a gauge transform
  $\tilde{O}_{(\infty)}$ on $\bbR^{d} \setminus \overline{B}$
  satisfying the following properties:
%
%
  \begin{itemize}
  \item $\tilde{O}_{(\infty)} = O$ in $2B \setminus \overline{B}$;
  \item $\tilde{O}_{(\infty)}(r \Tht) = \tilde{O}_{(\infty)}(4 \Tht)$
    for $\Tht \in \bbS^{d-1}$ and $r \geq 4$
  \item $\displaystyle{\nrm{\tilde{O}_{(\infty)}}_{L^{d} \cap
        \dot{W}^{1, \frac{d}{2}}(5B \setminus \overline{B})}
      \aleq_{\eps_{\ast}, \rcsp^{-1}} 1} $;
  \item $\tilde{O}_{(\infty)} $ is $C^{\infty}$ in $5B \setminus
    \overline{3B}$ with $\nrm{O_{(\infty)}}_{C^{N}(5B \setminus
      \overline{3B})} \aleq_{\eps_{\ast}, \rcsp^{-1}, N} 1$ for all $N
    \geq 0$.
  \end{itemize}
  Using $\tilde{O}_{(\infty)}$ to patch up the local gauges in $2B$
  and $\bbR^{d} \setminus \overline{B}$, we obtain the global gauge
  potential
  \begin{equation*}
    A_{x} = \left\{ 
      \begin{array}{cl}
	A_{(0) x} & \hbox{ on } 2B \\
	Ad(\tilde{O}_{(\infty)}) A_{(\infty) x} - \tilde{O}_{(\infty); x} & \hbox{ on } \bbR^{d} \setminus \overline{B}
      \end{array}
    \right.
  \end{equation*}
  Let $O_{(\infty)}$ be the smooth $0$-homogeneous map on $\bbR^{d}
  \setminus \set{0}$ defined by $O_{(\infty)}(r \Tht) =
  \tilde{O}_{(\infty)}(4 \Tht)$, and define $B_{x} = A_{x} + \chi
  O_{(\infty); x}$. By \eqref{eq:goodrep-extr},
  \eqref{eq:goodrep-intr} and the preceding bounds for
  $\tilde{O}_{(\infty)}$, the desired bounds \eqref{eq:goodrep-bnd}
  follow. \qedhere
\end{proof}

\subsection{Topological classes of rough connections}
Here, we verify the results stated in
Section~\ref{subsec:top-class}. Our first goal is to prove homotopy
equivalence of $O_{(\infty)}$ of different good representations of the
same connection (Proposition~\ref{prop:goodrep-homotopy-0}). We need a
few lemmas.
\begin{lemma} \label{lem:homotopy} Let $O \in
  \calG^{2,\frac{d}{2}}(A)$, where $A = \set{x \in \bbR^{d} : R_{1} <
    \abs{x} < R_{2}}$ is an annulus. For almost every $R \in (R_{1},
  R_{2})$, $O \restriction_{\rd B_{R}}$ is continuous, which are all
  homotopic to each other.
\end{lemma}
By this lemma, we may define $[O]$ to be the homotopy class (as
continuous maps $\bbS^{d-1} \to \G$) of the restriction of $O$ to $\rd
B_{R}$ for almost every $R$. We refer to such $R$'s as \emph{generic
  radii}.
\begin{proof}
  Since the boundary of $A$ is smooth, we may approximate $O$ by
  $O^{n} \in C^{\infty}(A; \G)$ in the $W^{2, \frac{d}{2}}(A; \bbR^{N
    \times N})$-topology \cite{MR710054, MR815194}. After passing to a
  subsequence, for almost every $R \in (R_{1}, R_{2})$, we have
  \begin{equation*}
    O^{n} \restriction_{\rd B_{R}} \to O^{n} \restriction_{\rd B_{R}} \quad \hbox{ in } W^{2, \frac{d}{2}}(\rd B_{R}; \bbR^{N \times N}).
  \end{equation*}
  The lemma now follows from the observation that $W^{2,
    \frac{d}{2}}(\rd B_{R}; \bbR^{N \times N}) \hookrightarrow
  C^{0}(\rd B_{R}; \bbR^{N \times N})$, due to the Sobolev embedding
  on spheres.
\end{proof}

\begin{lemma} \label{lem:homotopy-0} Let $\dlt > 0$ and let $O \in
  \calG^{2,\frac{d}{2}}(\tilde{A})$, where $\tilde{A} = \set{x \in
    \bbR^{d} : R_{1} -\dlt < \abs{x} < R_{2}}$. Then there exists an
  extension $\tilde{O} \in \calG^{2, \frac{d}{2}}(B_{R_{2}})$ such
  that $\tilde{O} \restriction_{A} = O \restriction_{A}$ if and only
  if $[O] = [id]$.
\end{lemma}
In this lemma, $[O]$ is defined by viewing $O$ as defined on either
the annulus $\tilde{A}$ or $A$; both give the same answer by
Lemma~\ref{lem:homotopy}. Our proof is qualitative, in that we make no
claim regarding the size of $\tilde{O} \in \calG^{2,
  \frac{d}{2}}(B_{R_{2}})$.
\begin{proof}
  We first prove the ``only if'' part. By Lemma~\ref{lem:part-approx}
  (with $V = \0$ and $U = W = B_{R_{2}}$), there exists an
  approximating sequence $O^{n} \in C^{\infty}(B_{R_{2}}; \G)$, which
  approaches $O$ in the $W^{2, \frac{d}{2}}(B_{R_{2}}; \bbR^{N \times
    N})$-topology. Recalling the proof of Lemma~\ref{lem:homotopy}, we
  see that $[O]$ is the homotopy class of $O^{n} \restriction_{\rd
    B_{R}}$ for any $\rd B_{R} \subseteq B_{R_{2}}$, provided that $n$
  is sufficiently large. Now, the whole map $O^{n} : B_{R_{2}} \to \G$
  provides a homotopy from $O^{n} \restriction_{\rd B_{R_{2}}}$ to the
  constant map $O^{n} \restriction_{\set{0}}$, which in turn is
  homotopic to the identity map.

  Next, we prove the ``if'' part. First, by
  Lemma~\ref{lem:part-approx}, there exists $O' \in \calG^{2,
    \frac{d}{2}}(R_{1} - \frac{4}{3} \dlt < \abs{x} < R_{2})$ such
  that $O' \in C^{\infty}(R_{1} - \frac{4}{3} \dlt < \abs{x} <R_{1} -
  \frac{1}{4} \dlt; \G)$, $O' \restriction_{A} = O
  \restriction_{A}$. By Lemma~\ref{lem:homotopy}, $[O'] = [O] =
  [id]$. Working in the smooth category, we may find $\tilde{O} \in
  \calG^{2, \frac{d}{2}}(B_{R_{2}})$ such that $\tilde{O}
  \restriction_{A} = O' \restriction_{A}$ while $\tilde{O} \in
  C^{\infty}(\abs{x} < R_{2} - \frac{1}{2} \dlt)$. \qedhere
\end{proof}

\begin{lemma} \label{lem:homotopy-infty} Let $O \in \calG^{2,
    \frac{d}{2}}(\bbR^{d} \setminus \overline{B})$. Then $[O] = [id]$.
\end{lemma}
In this lemma, $[O]$ is defined by viewing $O$ as defined on an
annulus $A \subseteq \bbR^{d} \setminus B$.
\begin{proof}
  Without loss of generality, let $U = \bbR^{d} \setminus
  \overline{B}_{1}$. We also observe that it suffices to prove $[O] =
  [const]$.  As before, by Lemma~\ref{lem:part-approx} (more
  precisely, a slight variant for the exterior domain) there exists an
  approximating sequence $O^{n} \in C^{\infty}(U; \G)$, which
  approaches $O$ in the $W^{2, \frac{d}{2}}(U; \bbR^{N \times
    N})$-topology, where $[O]$ is the homotopy class of $O^{n}
  \restriction_{\rd B_{R}}$ for any $\rd B_{R} \subseteq U$, provided
  that $n$ is sufficiently large.

  By Sobolev embedding, note that
  \begin{equation*}
    \int_{U} \abs{O^{n}_{;x}}^{d} \, \ud x < \infty \quad \hbox{ for all } n.
  \end{equation*}
  In the polar coordinates $(r, \Tht) \in (0, \infty) \times
  \bbS^{d-1}$, it follows that
  \begin{equation*}
    \int_{1}^{\infty} \int_{\bbS^{d-1}} \abs{\rd_{\Tht} O^{n} (r, \Tht)}^{d} \, \ud V_{\bbS^{d-1}}(\Tht)  \, \frac{\ud r}{r} < \infty \quad \hbox{ for all } n,
  \end{equation*}
  which implies that $\liminf_{r \to \infty} \nrm{\rd_{\Tht} O^{n} (r,
    \Tht)}_{L^{d}(\bbS^{d})} = 0$.  The desired conclusion $[O] =
  [const]$ now follows.
\end{proof}

We are ready to prove Proposition~\ref{prop:goodrep-homotopy-0}.
\begin{proof}[Proof of Proposition~\ref{prop:goodrep-homotopy-0}]
  By suitably replacing $\chi$, we may assume that $1-\chi$ vanishes
  outside the unit ball $B$.

  \pfstep{Proof of (1)} By equivalence of $(O_{(\infty)}, B_{x})$ and
  $(O'_{(\infty)}, B_{x}')$, there exists $O \in \calG_{loc}^{2,
    \frac{d}{2}}(\bbR^{d})$ such that
  \begin{equation*}
    - O_{(\infty);x } + B_{x} = - Ad(O) O_{(\infty); x}' - O_{;x} + Ad(O) B_{x}'.
  \end{equation*}
  From simple computation, it follows that
  \begin{equation*}
    (O_{(\infty)}^{-1} O O'_{(\infty)})_{;x} = Ad(O_{(\infty)}^{-1} O) B_{x}' - Ad(O_{(\infty)}^{-1}) B_{x},
  \end{equation*}
  which implies that $O_{(\infty)}^{-1} O O'_{(\infty)} \in \calG^{1,
    d}(\bbR^{d} \setminus \overline{B})$. Applying
  Lemmas~\ref{lem:homotopy-0} and ~\ref{lem:homotopy-infty} to $O$ and
  $O_{(\infty)}^{-1} O O'_{(\infty)}$, respectively, it follows that
  \begin{equation*}
    [id] = [O] = [O_{(\infty)}^{-1} O O'_{(\infty)}].
  \end{equation*}
  Therefore, $[O_{(\infty)}] = [O'_{(\infty)}]$, as desired.

  \pfstep{Proof of (2)} Since $[O_{(\infty)}' O_{(\infty)}^{-1}] =
  [id]$, by Lemma~\ref{lem:homotopy-0} there exists a gauge transform
  $P \in \calG^{2, \frac{d}{2}}(2B)$ such that $P = O_{(\infty)}'
  O_{(\infty)}^{-1}$ in $2B \setminus \overline{B}$. Extend $P$ as a
  $0$-homogeneous map outside $2B$; we abuse the notation and refer to
  the extension again by $P$ (thus, $P = O_{(\infty)}'
  O_{(\infty)}^{-1}$ in $\bbR^{d} \setminus \overline{B}$). Apply the
  gauge transform $P$ to $A_{x} = - \chi O_{(\infty); x} + B_{x}$, and
  define $B'_{x}$ by the decomposition $Ad(P) A_{x} - P_{;x} = - \chi
  O'_{(\infty); x} + B'_{x}$. From $P \in \calG^{2, \frac{d}{2}}(2B)$,
  it follows that $B'_{x} \in L^{d} \cap \dot{W}^{1,
    \frac{d}{2}}(2B)$. Moreover, outside $2B$,
  \begin{equation*}
    B'_{x} = Ad(P) B_{x}.
  \end{equation*}
  Observe that $0$-homogeneity of $P$ is sufficient to ensure
  $Ad(P)B_{x} \in L^{d} \cap \dot{W}^{1, \frac{d}{2}}(\bbR^{d}
  \setminus \overline{2B})$. Hence $(O'_{(\infty); x}, B'_{x})$ is
  also a good representation, as desired.
%
%
  \qedhere
\end{proof}

Finally, we prove Proposition~\ref{prop:top-class-outer}.
\begin{proof}[Proof of Proposition~\ref{prop:top-class-outer}]
  By scaling, we may set $R = 1$. Arguing as in the proof of
  Theorem~\ref{thm:goodrep}, we find local gauge potentials
  $A_{(\infty)}$ and $A'_{(\infty)}$ in $\bbR^{d} \setminus
  \overline{B}$ satisfying \eqref{eq:goodrep-extr}. By construction,
  there exist $O, O' \in \calG^{2, \frac{d}{2}}(5B \setminus
  \overline{B})$ such that
  \begin{equation*}
    A = Ad(O) A_{(\infty)} - O_{;x}, \quad 
    A' = Ad(O') A'_{(\infty)} - O'_{;x} \quad \hbox{ in } 5B \setminus \overline{B}.
  \end{equation*}
  From the proof of Theorem~\ref{thm:goodrep}, as well as
  Definition~\ref{def:top-class}, note that the topological classes
  $[A]$ and $[A']$ are determined by the homotopy classes $[O]$ and
  $[O']$, respectively, as defined in Lemma~\ref{lem:homotopy}. In
  particular, it suffices to prove that $O \restriction \rd B_{r}$ and
  $O' \restriction \rd B_{r}$ are homotopic to each other for a
  generic $1 < r < 5$, in the sense of Lemma~\ref{lem:homotopy}.

  Since $\nrm{A - A'}_{L^{d}(5B}) \leq \eps_{\ast}$, the difference
  $O_{;x} - O'_{;x}$ obeys the bound
  \begin{equation*}
    \nrm{O_{;x} - O'_{;x}}_{L^{d}(5B \setminus \overline{B})} \aleq \eps_{\ast},
  \end{equation*}
  which holds independently of possible additional constant gauge
  transformations for $O$ or $O'$.  By the pigeonhole principle, the
  following bound holds some generic $1 < r < 5$:
  \begin{equation*}
    \nrm{O_{;x} - O'_{;x}}_{L^{d}(\rd B_{r})} \aleq \eps_{\ast}.
  \end{equation*}
  After a suitable constant gauge transformation (which does not
  change the homotopy class), it follows that $O$ and $O'$ are close
  in $C^{\frac{1}{d}}(\rd B_{r})$, and therefore belong to the same
  homotopy class. \qedhere
\end{proof}

\section{Excision, gluing and extension of Yang--Mills initial data
  sets} \label{sec:excise} In this section, we provide proofs of the
results stated in Section~\ref{subsec:excise} concerning the
Yang--Mills initial data sets.
\subsection{Solvability results for the inhomogeneous Gauss equation}
In this subsection, we address the question of solvability for
divergence equations
\begin{equation}\label{ext-div} (\covD^{(a)})^{\ell} e_{\ell} = h
\end{equation}
in exterior of a convex domain.

To quantify the constants, we need to quantify the geometry of a
convex domain. Let $K$ be a convex domain with barycenter $x_{K}$. By
convexity, for each $\Tht \in \bbS^{d-1}$, there exists a unique
intersection $f_{K}(\Tht)$ of $\rd K$ and the ray in the direction
$\Tht$ emanating from $x_{K}$. Define the \emph{radius} of $K$ by
$R(K) = \sup_{x, y \in K} \abs{x-y}$, and the \emph{Lipschitz
  constant} of $K$ by
\begin{align}
  L(K) =& \sup_{\Tht, \Tht' \in \bbS^{d-1}} \frac{\abs{f_{K}(\Tht) -
      f_{K}(\Tht')}}{R(K) \abs{\Tht - \Tht'}} \label{eq:lip-K}.
\end{align}
Clearly, $R(K)$ is $1$-homogeneous and $L(K)$ is scaling-invariant, in
the sense that $R(\lmb K) = \lmb R(K)$ and $L(\lmb K) = L(K)$ for
$\lmb > 0$.

We begin with a general solvability result for the usual divergence
equation (i.e., $a = 0$).
\begin{proposition} \label{prop:gauss-zero} For any convex domain $K$,
  there exists a solution operator $T_{0}$ for the equation
  $\rd^{\ell} e_{\ell} = h$ with the following properties:
  \begin{enumerate}
  \item (Boundedness) For $1 < p < \infty$ and $1-\frac{d}{p} < \sgm <
    1 + \frac{d}{p}$,
    \begin{equation} \label{eq:gauss-zero} \nrm{T_{0}
        h}_{\dot{W}^{\sgm, p}} \aleq_{L(K), \sgm, p}
      \nrm{h}_{\dot{W}^{\sgm-1, p}}.
    \end{equation}
  \item (Exterior support) If $h = 0$ in $\lmb K$, then $T_{0} h = 0$
    in $\lmb K$.
  \item (Higher regularity) If $h$ is smooth, so is $T_{0} h$.
  \end{enumerate}
\end{proposition}
\begin{proof}
  In the case $K$ is a ball, this was considered in our prior work
  \cite{OT1}, where $T_{0}$ is constructed as a pseudodifferential
  operator of order $-1$.  Here we will use a slightly different but
  closely related solution operator.

  First, we claim that given a unit vector $\omega \in \bbS^{d-1}$, we
  can construct an exact solution operator $T_\omega$ with smooth
  homogeneous symbol of order $1$, and kernel supported in a small
  conic neighborhood of $\omega$. Our starting point is the simple
  observation that the following operator solves the divergence
  equation (say for $h \in C^{\infty}_{c}(\bbR^{d})$):
  \begin{equation*}
    \tilde{T}_{e_{1}} h(x) = \int_{-\infty}^{x^{1}} e_{1} h(y^{1}, x^{2}, \ldots, x^{d}) \, \ud y^{1},
  \end{equation*}
  where $e_{1}$ is the unit vector $(1, 0, \ldots, 0)$. This operator
  is translation-invariant with kernel
  \begin{equation*}
    e_{1} 1_{(0, \infty)}(x^{1}) \dlt_{0}(x^{2}) \cdots \dlt_{0}(x^{d}), 
  \end{equation*}
  which is supported on the ray $\set{r e_{1} : r > 0}$. By rotation,
  for any unit vector $\omg \in \bbS^{d-1}$, we obtain an analogous
  translation-invariant solution operator $\tilde{T}_{\omg}$ whose
  kernel is supported on the ray $\set{r \omg : r > 0}$. Moreover,
  given a smooth function $\tilde{\chi}_{\omg}(\omg')$ on $\bbS^{d-1}$
  supported on a neighborhood $\hat{C}_{\omg} \subseteq \bbS^{d-1}$,
  the smooth average
  \begin{equation*}
    T_{\omg} h = \int \tilde{T}_{\omg'} (h) \tilde{\chi}_{\omg}(\omg') \, \ud \omg'
  \end{equation*}
  is a translation-invariant solution operator, whose kernel is smooth
  outside the origin, homogeneous of degree $-d+1$ and supported in
  the conic neighborhood $C_{\omg} = \set{x \in \bbR^{d}:
    \frac{x}{\abs{x}} \in \hat{C}_{\omg}}$, as desired.

  We now turn to the issue of insuring the exterior support
  property. If one were to work with the operators $\tilde{T}_{\omg}$,
  then it is easy to produce such an solution operator $T$: We simply
  decompose the input into each angle $\omg$ and apply
  $\tilde{T}_{\omg}$, i.e., $T = \int_{\bbS^{d-1}} \tilde{T}_{\omg}
  \dlt_{\omg}(\omg') \, \ud \omg'$. Then (formally) the exterior
  support property holds for any convex set $K$.

  To use the operators $T_{\omg}$ with ``fattened'' kernel, we use a
  uniform conical partition of unity in the physical space $1 = \sum
  \chi_\omega$ (centered at the origin) and define our solution
  operator $T_{0}$ to be
  \[
  T_{0} = \sum T_\omega \chi_\omega.
  \]
  Making the angular support of each $\chi_{\omega}$ sufficiently
  narrow (which, of course, increases the number of partitions)
  depending on $L(K)$, we may insure the exterior support property of
  $T$.

  Multiplication by each $\chi_{\omg}$ is bounded on $\dot{W}^{\sgm-1,
    p}$ thanks to Hardy's inequality, which holds since $\abs{\sgm-1}
  < \frac{d}{p}$; hence \eqref{eq:gauss-zero} follows. The higher
  regularity property follows by differentiation. \qedhere
\end{proof}

Next, we generalize Proposition~\ref{prop:gauss-zero} to the
inhomogeneous covariant Gauss equation \eqref{ext-div} when
$\nrm{a}_{\dot{H}^{\frac{d-2}{2}}}$ is small by a perturbative
argument.
\begin{proposition} \label{prop:gauss-small} Let $\covD = \ud + a \in
  \calA^{\frac{d-2}{2}, 2}(\bbR^{d})$ satisfy
  $\nrm{a}_{\dot{H}^{\frac{d-2}{2}}} \leq \eps_{\ast}$. For any convex
  domain $K$, there exists a solution operator $T_{a}$ for the
  equation $\covD^{\ell} e_{\ell} = h$ with the following properties:
  \begin{enumerate}
  \item (Boundedness) For $2 \leq p < \infty$ and $1-\frac{d}{p} <
    \sgm < \frac{d}{2}$,
    \begin{equation} \label{eq:gauss-small} \nrm{T_{a}
        h}_{\dot{W}^{\sgm, p}} \aleq_{L(K), \sgm, p}
      \nrm{h}_{\dot{W}^{\sgm-1, p}}.
    \end{equation}
  \item (Exterior support) If $h = 0$ in $\lmb K$, then $T_{a} h = 0$
    in $\lmb K$.
  \item (Higher regularity) If $a$ and $h$ is smooth, so is $T_{a} h$.
  \end{enumerate}
\end{proposition}

\begin{proof}
  We proceed in two steps.

  \pfstep{Step~1: Definition of $T_{a}$} To define $T_{a}$, we solve
  the fixed point problem is
  \[
  e = T( h - [a^{\ell},e_{\ell}]).
  \]
  Let us abbreviate $[a^{\ell}, e_{\ell}] = ad(a) e$. Under the
  conditions for $p$ and $\sgm$, multiplication by $a$ takes
  $\dot{W}^{\sgm, p}$ into $\dot{W}^{\sgm-1, p}$ (this may be proved
  by the usual Littlewood--Paley trichotomy), so that we can estimate
  \[
  \| T ad(a) \|_{\dot{W}^{\sgm, p} \to \dot{W}^{\sgm, p}} \lesssim
  \|a\|_{\dot{H}^{\frac{d-2}{2}}}.
  \]
  Therefore, for $\nrm{a}_{\dot{H}^{\frac{d-2}{2}}}$ sufficiently
  small, we find $T_{a}$ which clearly satisfies the boundedness and
  exterior support properties.

  \pfstep{Step~2: Higher regularity} Here we assume that $\rd^{(m)} a
  \in \dot{H}^{\frac{d-2}{2}}$ and $\rd^{(m)} h \in \dot{W}^{\sgm-1,
    p}$ for $0 \leq m \leq n$, then we prove that $\rd^{(n)} e \in
  \dot{W}^{\sgm, p}$. We consider the case $n=1$; higher values of $n$
  are dealt with in a similar manner. Differentiating our fixed point
  problem we get
  \begin{equation} \label{eq:T-diff} \rd e = T( \rd h - [a^{\ell},\rd
    e_{\ell}]) - T( [\rd a^{\ell},e_{\ell}]) + [\rd, T] ( h -
    [a^{\ell},e_{\ell}])
  \end{equation}
  where we can estimate
  \[
  \| T( [\rd a,e]) + [T,\rd] ( h - [a,e])\|_{\dot{W}^{\sgm, p}}
  \lesssim \| e\|_{\dot{W}^{\sgm, p}} + \|h\|_{\dot{W}^{\sgm, p}}
  \]
  with an implicit constant depending on the $\dot{H}^{\frac{d-2}{2}}$
  norms of $\rd a$ and $a$. Then we have a fixed point problem for
  $\rd e$, which is solved in $\dot{W}^{\sgm, p}$ to obtain the bound
  \[
  \| \rd e\|_{\dot{W}^{\sgm, p}} \lesssim \| h \|_{\dot{W}^{\sgm, p}
    \cap \dot{W}^{\sgm-1, p}}
  \]
  One minor issue here is that we do not a-priori know that $ \rd e
  \in \dot{W}^{\sgm, p}$. But this can be easily circumvented by
  replacing the gradient with the appropriated divided
  difference. \qedhere
\end{proof}

Finally, we prove Theorem~\ref{thm:gauss-0}, where the smallness
assumption for $a$ is removed. For simplicity, we restrict to the
critical space $h \in \dot{H}^{\frac{d-6}{2}}$ where $d \geq 4$, which
suffices for our main applications.
\begin{proof}[Proof of Theorem~\ref{thm:gauss-0}]
  We work from the case when $h$ is not differentiated (i.e., $h \in
  L^{\frac{d}{2}}$), and gradually move up to higher regularity
  spaces. In the proof, we omit the dependence of constants on $L(K)$.

  \pfstep{Step~1: Construction of $T_{a}: \dot{W}^{-1, p} \to L^{p}$
    $(1 < p < d)$} We compensate for the lack of smallness of $a$ by
  adding a weight $w = 2^{-\phi}$ where $\phi$ is a smooth bounded
  increasing radial function. The goal is to insure that
  \[
  \| T ad(a) \|_{L^{p}_{w} \to L^p_w} \ll 1
  \]
  We denote
  \[
  A_k = \{ x \in \bbR^{d}: k \leq \phi(x) \leq k+1\} .
  \]
  Then for $j \geq k$, by H\"older's inequality, the embedding $L^{q}
  \hookrightarrow \dot{W}^{-1, p}$ (where $q^{-1} = p^{-1} + d^{-1}$)
  and Proposition~\ref{prop:gauss-zero} we have
  \[
  \| 1_{A_j} T ad(a) 1_{A_k} \|_{L^{p}_{w} \to L^{p}_w} \lesssim
  2^{k-j} \|a\|_{L^{d}(A_k)} .
  \]
  On the other hand, the LHS vanishes when $j < k$ by the exterior
  support property. After summation, we obtain
  \[
  \| T ad(a) \|_{L^{p}_{w} \to L^{p}_w} \lesssim \sup_k \|
  a\|_{L^{d}(A_k)}.
  \]
  Thus to insure the desired smallness, it suffices to choose $w$ so
  that the RHS is small, which is easily done.

  \pfstep{Step~2: Boundedness into $\dot{H}^{\frac{d-4}{2}}$} Let $n$
  be the least integer greater than or equal to $\frac{d-4}{2}$. The
  strategy is to commute $\rd$ for up to order $n$ (as in Step~2 in
  the proof of Proposition~\ref{prop:gauss-small}), and inductively
  prove boundedness of $T_{a}: \dot{W}^{m-1, \frac{d}{m+2}} \to
  \dot{W}^{m, \frac{d}{m+2}}$ for $m=1, \ldots, n$; this would
  directly imply \eqref{eq:gauss-0-bnd} for even $d$, and after
  interpolation for odd $d$.

  For simplicity, as in Step~2 of the proof of
  Proposition~\ref{prop:gauss-small}, we only consider the case $n=1$;
  the general case is dealt with by induction in a similar manner. Our
  starting point is \eqref{eq:T-diff}:
  \begin{equation} \label{eq:T-diff'} \rd e = T( \rd h - [a^{\ell},\rd
    e_{\ell}]) - T( [\rd a^{\ell},e_{\ell}]) + [\rd, T] ( h -
    [a^{\ell},e_{\ell}]).
  \end{equation}
  The strategy is to use $\nrm{e}_{L^{d}}$, which is already under
  control, to estimate the last two terms, and use an iteration
  argument in $L^{p}_{w}$ as in Step~1 with $p = \frac{d}{3}$ to
  estimate\footnote{As in Step~2 of
    Proposition~\ref{prop:gauss-small}, to be rigorous one should work
    with divided differences, but the argument is essentially the
    same.} $\rd e$. By Proposition~\ref{prop:gauss-zero}, Sobolev and
  H\"older, we have
  \begin{equation*}
    \nrm{T([\rd a^{\ell}, e_{\ell}])}_{L^{\frac{d}{3}}} 
    \aleq \nrm{\rd a}_{L^{\frac{d}{2}}} \nrm{e}_{L^{\frac{d}{2}}}
    \aleq \nrm{a}_{\dot{H}^{\frac{d-2}{2}}} \nrm{e}_{L^{\frac{d}{2}}}.
  \end{equation*}
%
%
  On the other hand, note that $[T, \rd] = \sum_{\omg} T_{\omg} \rd
  \chi_{\omg}$ (cf. proof of Proposition~\ref{prop:gauss-zero}), where
  $\chi_{\omg}$ is $0$-homogeneous. Thus by $T_{\omg} : \dot{W}^{-1,
    p} \to L^{p}$, Hardy, Sobolev and H\"older,
  \begin{equation*}
    \nrm{[\rd, T](h - [a^{\ell}, e_{\ell}])}_{L^{\frac{d}{3}}} 
    \aleq \nrm{h - [a^{\ell}, e_{\ell}]}_{L^{\frac{d}{3}}}
    \aleq \nrm{h}_{L^{\frac{d}{3}}} + \nrm{a}_{\dot{H}^{\frac{d-2}{2}}}\nrm{e}_{L^{\frac{d}{2}}}.
  \end{equation*}
  By Step~1 with $p = \frac{d}{2}$, we have $\nrm{e}_{L^{\frac{d}{2}}}
  \aleq \nrm{h}_{\dot{W}^{-1, \frac{d}{2}}} \aleq
  \nrm{h}_{L^{\frac{d}{3}}}$. Then finding the fixed point $\rd e$ of
  \eqref{eq:T-diff'} as in Step~1, the desired estimate $\nrm{\rd
    e}_{L^{\frac{d}{3}}} \aleq_{\nrm{a}_{\dot{H}^{\frac{d-2}{2}}}}
  \nrm{h}_{L^{\frac{d}{3}}}$ follows.

  \pfstep{Step~3: Higher regularity} This step is analogous to Step~2
  of Proposition~\ref{prop:gauss-small}, where the iteration is done
  in $L^{p}_{w}$. \qedhere
\end{proof}

\subsection{Initial data surgery}

Now we explore consequences of the previous result in terms of
excising and extending initial data sets. The aim of this subsection
is to prove Theorems~\ref{thm:ext-id} and \ref{thm:excise}.
 
Before we turn to the proofs, a few remarks about Sobolev extension
are in order. For any domain $K$ with locally Lipschitz boundary,
Stein's extension theorem \cite[\S VI.3]{MR0290095} says that there
exists a universal linear extension operator $\mathfrak{E}$ for all
Sobolev spaces $W^{\sgm, p}(K) \to W^{\sgm, p}(\bbR^{d})$. When $K$ is
convex with $R(K) = 1$ (which we may insure by scaling), it can be
checked that the constant depends only on $\sgm, p$ and the Lipschitz
constant $L(K)$. In particular, we have
\begin{equation} \label{eq:ext-K} \nrm{\mathfrak{E} u}_{\dot{W}^{\sgm,
      p}} \aleq_{L(K), p, \sgm} \nrm{u}_{\dot{W}^{\sgm, p}(K)} \quad
  \hbox{ where } \sgm \geq 0, \ q, p \in (1, \infty), \ \frac{d}{p} -
  \sgm = \frac{d}{q}.
\end{equation}
The same bound holds for general $R(K)$ by scaling-invariance of the
both sides. Similarly, for an annular region $4K \setminus
\overline{K}$ (with general $R(K)$), there exists a universal linear
extension operator $\mathfrak{E}$ such that
\begin{equation} \label{eq:ext-ann-K} \nrm{\mathfrak{E}
    u}_{\dot{W}^{\sgm, p}} \aleq_{L(K), p, \sgm}
  \nrm{u}_{\dot{W}^{\sgm, p}(4 K \setminus \overline{K})} \quad \hbox{
    where } \sgm \geq 0, \ q, p \in (1, \infty), \ \frac{d}{p} - \sgm
  = \frac{d}{q}.
\end{equation}

Now we prove Theorem~\ref{thm:ext-id}, concerning truncation of
Yang--Mills initial data sets.
\begin{proof}[Proof of Theorem~\ref{thm:ext-id}]
  Let $(a, e)$ be the given $\calH^{\frac{d-2}{2}}$ Yang--Mills
  initial data set on $2K \setminus \overline{K}$. In this proof, we
  use the shorthands
  \begin{equation*}
    A = \nrm{a}_{\dot{H}^{\frac{d-2}{2}}(2K \setminus \overline{K})}, \quad
    E = \nrm{e}_{\dot{H}^{\frac{d-4}{2}}(2K \setminus \overline{K})}.
  \end{equation*}

  First, we use the universal extension $\mathfrak{E}$ to extend $a,
  e$ to $\ba, \be'$ on $\bbR^{d}$, respectively. Clearly, restriction
  of $\ba$ satisfies \eqref{eq:ext-id-a}. On the other hand, $\be'$
  obeys a favorable bound, but violates the Gauss equation outside $2K
  \setminus \overline{K}$. More precisely,
  \begin{equation*}
    (\covD^{(\ba)})^{\ell} \be'_{\ell} = h
  \end{equation*}
  where $h = 0$ in $2K \setminus \overline{K}$ since $(\ba, \be) = (a,
  e)$ there. Let $\chi_{out}$ be a smooth cutoff which equals zero in
  $K$ and $1$ outside $2K$, then let $h_{out} = \chi_{out} h$. Note
  that
  \begin{equation*}
    \nrm{h_{out}}_{\dot{H}^{\frac{d-6}{2}}} \aleq \nrm{\rd \be}_{\dot{H}^{\frac{d-6}{2}}}+\nrm{\ba}_{\dot{H}^{\frac{d-2}{2}}}\nrm{\be}_{\dot{H}^{\frac{d-4}{2}}} \aleq_{A} E.
  \end{equation*}
  Hence, by Theorem~\ref{thm:gauss-0}, we find $d_{\ell}$ such that
  $(\covD^{(\ba)})^{\ell} d_{\ell} = - h_{out}$, $d = 0$ in $2K$ and
  $\nrm{d}_{\dot{H}^{\frac{d-4}{2}}} \aleq_{A} E$. The desired $\be$
  is then given by the restriction of $\be' + d$ to $\bbR^{d}
  \setminus \overline{K}$.

  To conclude the proof, note that the higher regularity and local
  Lipschitz properties are obvious by construction. Finally,
  equivariance under constant gauge transformations can be insured by
  fixing a particular construction, conjugating by elements of $\G$,
  and then averaging over $\G$.  \qedhere
\end{proof}

Combined with Uhlenbeck's lemma (Theorem~\ref{thm:uhlenbeck-ball}), we
may now prove the final excision-and-extension result
(Theorem~\ref{thm:excise}).

\begin{proof}[Proof of Theorem~\ref{thm:excise}]
  We only treat the case when $d \geq 4$ is even and $X = B_{R}$. The
  other cases are simpler and thus are left to the reader (when $d$ is
  odd, Uhlenbeck's lemma is not needed, and when $X = \bbR^{d}$, the
  extension procedure is unnecessary).

  \pfstep{Step~1: Application of Uhlenbeck's lemma} As in the proof of
  Proposition~\ref{prop:goodrep-ball-key}, we first set $a_{r} = 0$ by
  Lemma~\ref{lem:Ar=0}, and extend $a$ outside $B_{R}$ by
  Lemma~\ref{lem:ext-simple}. Then the $L^{d}$-concentration radius of
  $a$ does not vary much, and Uhlenbeck's lemma
  (Theorem~\ref{thm:uhlenbeck-ball}) is applicable on any ball $B_{2
    r}(x)$ with $r < 10 \rc$ and $x \in B_{R}$. We claim that
  \begin{equation*}
    \nrm{\ta}_{\dot{H}^{\frac{d-2}{2}}(B_{r}(x) \cap X)} \aleq \nrm{\covD^{(\frac{d-4}{2})} F[a]}_{L^{2}(B_{r}(x) \cap X)} + \nrm{F[a]}_{L^{\frac{d}{2}}(B_{r}(x) \cap X)}.
  \end{equation*}
  For interior balls (i.e., $B_{2r}(x) \cap \rd B = \0$), this bound
  follows directly from Lemma~\ref{lem:div-curl-A-intr}. For boundary
  balls (i.e., $B_{2r}(x) \cap \rd B \neq \0$), we obtain angular
  regularity (with respect to the center of $B_{R}$) by
  Lemma~\ref{lem:div-curl-A-tang}, then radial regularity by
  \eqref{eq:div-curl-rad}. We note that the implicit constant is
  controlled thanks to the smallness of $\eps_{\ast}$.

  Next, by the formula $O_{;x} = Ad(O) a - \ta)$, we obtain
  \eqref{eq:excise-O}. Then it also follows that
  \begin{equation*}
    \nrm{\te}_{\dot{H}^{\frac{d-4}{2}}(B_{r}(x) \cap X)} \aleq \nrm{\covD^{(\frac{d-4}{2})} e}_{L^{2}(B_{r}(x) \cap X)} + \nrm{e}_{L^{\frac{d}{2}}(B_{r}(x) \cap X)}.
  \end{equation*}

  \pfstep{Step~2: Application of Theorem~\ref{thm:ext-id}} We apply
  Theorem~\ref{thm:ext-id} to $(\ta, \te)$ and obtain an extended
  Yang--Mills initial data set outside the convex domain $K = B_{r}(x)
  \cap B_{R}$, which we still denote by $(\ta, \te)$. We note for
  domains of the form $K = B_{r}(x) \cap B_{R}$, we have the universal
  bound $R(K) \aeq r$ and $L(K) \aeq 1$. Therefore, by
  \eqref{eq:ext-id-a}, \eqref{eq:ext-id-e}, and the preceding bounds
  for $(\ta, \te)$ on $K$, we obtain \eqref{eq:excise-a}.

  \pfstep{Step~3: Completion of proof} It remains to prove
  Theorem~\ref{thm:excise}.(2). We begin by clarifying the ambiguity
  of the construction so far.  In Step~1, the triple $(\ta, \te, O)
  \restriction_{K}$ is determined up to a constant gauge
  transformation, as in Uhlenbeck's lemma
  (Theorem~\ref{thm:uhlenbeck-ball}). Since Theorem~\ref{thm:ext-id}
  is equivariant under such operations, the corresponding extensions
  $(\ta, \te)$ are also constant gauge transformations of each other.

  As a result, in order to prove (2), it suffices to show that we can
  enforce strong convergence of $(\ta^{n}, \te^{n}, O^{n})$ to $(\ta,
  \te, O)$ in $H^{\frac{d-2}{2}} \times H^{\frac{d-4}{2}} \times
  H^{\frac{d}{2}}(K)$, after passing to a subsequence and conjugating
  the sequence with a constant gauge transformation. Proceeding as in
  Theorem~\ref{thm:uhlenbeck-ball}.(2), we may first insure
  convergence of a suitable subsequence up to a constant gauge
  transformation in $W^{1, \frac{d}{2}} \times L^{\frac{d}{2}} \times
  W^{2, \frac{d}{2}}(K)$. Then by Remark~\ref{rem:uhlenbeck-cont},
  strong convergence in the desired topology (of the same sequence)
  may be proved. We omit the straightforward details. \qedhere
\end{proof}

\section{The local theory for the hyperbolic Yang--Mills
  equation} \label{sec:local} In this section, we consider the
local-in-time theory for the hyperbolic Yang--Mills equation for data
in an arbitrary topological class.

\subsection{Gauge equivalent classes of connections}
We start by verifying that the gauge-equivalent class of
$\calH^{\frac{d-2}{2}}$ connections is closed, as asserted in
Section~\ref{subsec:local}.
\begin{proposition}\label{p:closed-class}
  Let $A$ be an $\calH^{\frac{d-2}{2}}_{loc}$ connection in $\calO
  \subseteq \bbR^{1+d}$. Then $[A]$ is closed in the corresponding
  topology.
\end{proposition}

\begin{proof}
  Suppose that $O^{(n)}$ is a sequence of admissible gauge
  transformations so that the gauge equivalent connections $A^{(n)}$
  given by
  \begin{equation} \label{eq:closed-An} A^{(n)} = Ad(O^{(n)}) A -
    O^{(n)}_{; t,x}
  \end{equation}
  converge to an $\calH^{\frac{d-2}{2}}_{loc}$ connection $B$. Then we
  need to show that $A$ and $B$ are gauge equivalent.

  We first consider the corresponding gauge transformations
  $O^{(n)}$. By the relation \eqref{eq:closed-An}, it follows that
  these are uniformly bounded on compact sets. Hence, by compact
  Sobolev embeddings we obtain a limiting gauge transformation $O$, so
  that on a subsequence we have
  \begin{enumerate}[label=(\roman*)]
  \item $O$ satisfies the bounds
    \[
    \nb O \in L^\infty H^{\frac{d-2}{2}}_{loc}
    \]
  \item Convergence in weaker topologies:
    \[
    \nb O^{(n)} \to \nb O \qquad \text{ in $L^p
      W^{\frac{d-2}{2},2-}_{loc}$,\ \ $p < \infty$.}
    \]

  \item Pointwise a.e convergence:
    \[
    O^{(n)}(t,x) \to O(t,x) \ \ \text{a.e.}, \qquad \nb O^{(n)} (t,x)
    \to \nb O(t,x) \ \ \text{a.e.}.
    \]
  \end{enumerate}
  These properties allow us to pass to the limit and obtain
  \[
  B = O A O^{-1} - O_{;t,x}
  \]
  as well as the similar relation for the curvatures.

  It remains to improve the first property (i) above to continuity in
  time. This cannot come from weak convergence, instead it is a
  consequence of the corresponding continuity property for $A$ and
  $B$. We start from property (ii), which guarantees that $O(t,x)$ is
  continuous in $t$ for almost every $x$. Since $A,B \in C_t
  L^{d}_{loc}$, so is $Ad(O) A $ and thus $\nb O$.  We now
  differentiate and repeat the process for $\rd \nb O$ in
  $L^{\frac{d}{2}}$, and so on.
\end{proof}

\subsection{Local theory at optimal regularity 
for dimensions \texorpdfstring{$d \geq 4$}{d>3}}
We begin by recalling the temporal gauge small data global
well-posedness result proved\footnote{In \cite{OTYM2}, this theorem is
  stated an proved in the most difficult case $d = 4$. Nevertheless,
  its proof may be extended to $d > 4$.} in \cite{OTYM2}.
\begin{theorem}[{\cite[Theorem~1.17]{OTYM2}}] \label{thm:small-temp}
  If the $\dot{H}^{\frac{d-2}{2}} \times \dot{H}^{\frac{d-4}{2}}$ norm
  of the initial data set $(a, e)$ is smaller than some universal
  constant $\eps_{\ast}$, then the corresponding solution $(A_{t,x},
  \rd_{t} A_{t,x})$ in the temporal gauge $A_{0} = 0$ exists globally
  in $C_{t}(\bbR; \dot{H}^{\frac{d-2}{2}} \times
  \dot{H}^{\frac{d-4}{2}})$, and obeys the a-priori bound
  \begin{equation*}
    \nrm{\nb A_{x}}_{L^{\infty} \dot{H}^{\frac{d-4}{2}}} \aleq \nrm{(a, e)}_{\dot{H}^{\frac{d-2}{2}} \times \dot{H}^{\frac{d-4}{2}}}.
  \end{equation*}
  The solution is unique among the local-in-time limits of smooth
  solutions, and it depends continuously on the data $(a, e) \in
  \dot{H}^{\frac{d-2}{2}} \times \dot{H}^{\frac{d-4}{2}}$.
\end{theorem}
 
We now derive Theorem~\ref{thm:local-temp} from
Theorems~\ref{thm:excise} and \ref{thm:small-temp}.
\begin{proof}[Proof of Theorem~\ref{thm:local-temp}]
  The idea is to construct the local-in-spacetime solutions using
  Theorems~\ref{thm:excise} and \ref{thm:small-temp}, and then patch
  up by finite speed of propagation (i.e., local-in-spacetime
  uniqueness) in the temporal gauge.  \pfstep{Step~1: Construction of
    local-in-spacetime solutions} Consider a ball $B_{r}(x)$ with $r <
  10 \rc$ and $x \in X$; we introduce the abbreviation $K = B_{r}(x)
  \cap X$. Let $(\ta, \te)$ and $O$ be the global Yang--Mills initial
  data and the gauge transformation associated with $(a, e)$ by
  Theorem~\ref{thm:excise}.(1); recall that $(a, e)$ is gauge
  equivalent to $(\ta, \te)$ via $O$ on $K$. Choosing $\eps_{\ast}$
  sufficiently small, Theorem~\ref{thm:small-temp} produces a unique
  $C_{t} \calH^{\frac{d-2}{2}}$ temporal-gauge solution $\tA$
  corresponding to $(\ta, \te)$. We define $A$ on $\calD(K)$ by
  \begin{equation*}
    A_{\mu}(t,x) = Ad(O^{-1}(x)) \tA_{\mu}(t,x) - O^{-1}_{; \mu}(x).
  \end{equation*}
  Note that $(\ta, \te, O)$ in Theorem~\ref{thm:excise}.(1) is
  determined up to a constant gauge transformation, but any choice
  leads to the same solution $A$. By \eqref{eq:excise-a},
  \eqref{eq:excise-O} and Theorem~\ref{thm:small-temp}, it follows
  that
  \begin{equation} \label{eq:local-temp-A} \nrm{\nb A_{x}}_{L^{\infty}
      \dot{H}^{\frac{d-4}{2}}(\calD(K))} \aleq \nrm{(a,
      e)}_{\dot{H}^{\frac{d-2}{2}} \times \dot{H}^{\frac{d-4}{2}}(K)}.
  \end{equation}

  \pfstep{Step~2: Continuous dependence and uniqueness} We claim that
  the mapping
  \begin{equation*}
    \calH^{\frac{d-2}{2}}(K) \ni (a, e) \mapsto (A, \rd_{t} A) \in C_{t} (H^{\frac{d-2}{2}} \times H^{\frac{d-4}{2}}) (\calD(K))
  \end{equation*}
  is continuous. Indeed, for the purpose of contradiction, suppose
  that there is a sequence of $\calH^{\frac{d-2}{2}}$ Yang--Mills
  initial data sets on $K$ such that $(a^{n}, e^{n}) \to (a, e)$,
  while $(A^{n}, \rd_{t} A^{n}) \not \to (A, \rd_{t} A)$ in $C_{t}
  (H^{\frac{d-2}{2}} \times H^{\frac{d-4}{2}})(\calD(K))$. By passing
  to a subsequence, we may assume that no further subsequence of
  $(A^{n}, \rd_{t} A^{n})$ converges to $(A, \rd_{t} A)$ in the same
  topology. However, by Theorem~\ref{thm:excise}.(2) and the
  continuity statement in Theorem~\ref{thm:small-temp}, there exists a
  subsequence for which $(\tA^{n}, \rd_{t} \tA^{n}) \to (\tA, \rd_{t}
  \tA)$ in $C_{t} (\dot{H}^{\frac{d-2}{2}} \times
  \dot{H}^{\frac{d-4}{2}})$. By the convergence $O^{n} \to O$ and
  $(O^{n})^{-1} \to O^{-1}$ in $\calG^{\frac{d}{2}, 2}(K)$, it follows
  that $(A^{n}, \rd_{t} A^{n}) \to (A, \rd_{t} A)$ in the above
  topology, which is a contradiction.

  From continuous dependence and persistence of regularity in
  Theorems~\ref{thm:excise} and \ref{thm:small-temp} it follows that
  $(A, \rd_{t}A)$ defined in Step~1 is approximated by smooth
  (temporal gauge) solutions, i.e., it is a solution to \eqref{eq:ym}
  in the sense of Definition~\ref{def:ym-sol-rough}. Therefore,
  uniqueness of the solution on $\calD(K)$ in the sense of
  Definition~\ref{def:ym-sol-rough} in the temporal gauge follows.

  \pfstep{Step~3: Conclusion of the proof} Consider now a family of
  balls $\set{B_{2\rc}(x)}_{x \in X}$, and the corresponding family
  of temporal gauge solutions in each $\calD(B_{2\rc}(x) \cap
  X)$. By the local-in-spacetime uniqueness that we just proved, these
  solutions coincide on the intersections, and therefore define a
  unique temporal gauge solution (in the sense of
  Definition~\ref{def:ym-sol-rough}) in $\calD_{[0, \rc)}(X)
  \subseteq \cup_{x \in X} \calD(B_{2\rc}(x) \cap X)$.

  Properties (1) and (2) claimed in Theorem~\ref{thm:local-temp}
  follow from the construction. For the a-priori bound in (3), we
  repeat the above steps to the data restricted to uniformly spaced
  balls $B$ of radius $2 \rc$ that cover $B_{R'}(x)$. By
  local-in-spacetime uniqueness, the result coincides with $A$ in
  $\calD_{[0, \rc)} (B_{R'}(x))$. Moreover,
  \eqref{eq:local-temp-bnd} follows by summing up the a-priori bounds
  in Theorem~\ref{thm:small-temp} for the local-in-spacetime
  solutions. \qedhere
\end{proof}

Next, we also show that all $\calH^{\frac{d-2}{2}}_{loc}$ solutions
(in the sense of Definition~\ref{def:ym-sol-rough}) are gauge
equivalent to the corresponding temporal solutions.

\begin{proof}[Proof of Theorem~\ref{thm:equiv-temp}]
  Let $A^{(n)}$ be a sequence of smooth solutions which converge to
  $A$ in the norm $C_{t} (H^{\frac{d-2}{2}}_{loc} \times
  H^{\frac{d-4}{2}}_{loc})$. Let $\tA^{(n)}$, respectively $\tA$, be
  the corresponding temporal solutions. We know that $\tA^{(n)}$ and
  $A^{(n)}$ are gauge equivalent; denote by $O^{(n)}$ the
  corresponding gauge transformations.

  We know that in the $H^1$ topology
  \[
  Ad(O^{(n)}) \tA^{(n)} - O^{(n)}_{;t,x} = A^{(n)} \to A
  \]
  but also that
  \[
  \tA^{(n)} \to \tA
  \]
  Thus, the gauge transformations $O^{(n)}$ satisfy uniform bounds
  locally. Then it follows that (up to a subsequence)
  \[
  Ad(O^{(n)}) \tA - O^{(n)}_{;t,x} \to A .
  \]
  But now we can use Proposition~\ref{p:closed-class} to conclude that
  $\tA$ and $A$ are gauge equivalent.
\end{proof}

Continuity of $A_{x}(t)$ in $\calH^{\frac{d-2}{2}}_{loc, \rc}$, as
stated in Theorem~\ref{thm:local-temp}, is in general insufficient to
conclude invariance of the topological class. However, combined with
finite speed of propagation and
Proposition~\ref{prop:top-class-outer}, we may nevertheless prove that
the topological class of $A_{x}(t)$ is conserved under the hyperbolic
Yang--Mills evolution.
\begin{proof}[Proof of Proposition~\ref{prop:top-class-ym}]
  Thanks to Theorem~\ref{thm:equiv-temp}, it suffices to consider a
  temporal gauge solution $A_{x}(t)$. By a usual continuous induction
  in $t$ (as well as time reversibility of \eqref{eq:ym}), it suffices
  to show that the $[A_{x}(t)] = [A_{x}(0)]$ for all $t > 0$
  sufficiently close to $0$.

  Since $\calE_{\bbR^{d}}^{\frac{d-2}{2}}(a, e) < \infty$, there
  exists $R > 0$ such that $\calE_{\bbR^{d} \setminus
    \overline{B_{R}}}^{\frac{d-2}{2}}(a, e) \ll \eps_{\ast}$. By
  Uhlenbeck's lemma (when $d$ is even) and the local-in-spacetime
  a-priori estimate \eqref{eq:local-temp-bnd}, it follows that
  \begin{equation*}
    \sup_{t \in [0, \rc)}\calE_{\bbR^{d} \setminus \overline{B_{R + t}}}^{\frac{d-2}{2}} (A_{x}(t), \rd_{t} A_{x}(t))
    \aleq \calE_{\bbR^{d} \setminus \overline{B_{R}}}^{\frac{d-2}{2}}(a, e) \ll \eps_{\ast}.
  \end{equation*}
  In particular, choosing $R$ large enough, we may insure that
  \begin{equation*}
    \nrm{F[A_{x}(t)]}_{L^{\frac{d}{2}}(\bbR^{d} \setminus \overline{B_{2R}})} < \eps_{\ast},
  \end{equation*}
  where $\eps_{\ast}$ is as in
  Proposition~\ref{prop:top-class-outer}. For $t > 0$ sufficiently
  close to $0$, by the continuity property \eqref{eq:local-temp-cont},
  we may also insure that
  \begin{equation*}
    \nrm{A_{x}(t) - A_{x}(0)}_{L^{d}(B_{2R})} < \eps_{\ast}.
  \end{equation*}
  By Proposition~\ref{prop:top-class-outer}, it follows that
  $[A_{x}(t)] = [A_{x}(0)]$. \qedhere
\end{proof}

Finally, we turn to the proof of Theorem~\ref{thm:imp-reg}. The main
ingredient is the caloric gauge small data well-posedness theorem from
\cite{OTYM2}:
\begin{theorem}[{\cite[Corollary~1.13]{OTYM2}}] \label{thm:small-cal}
  Let $(a, e)$ be an Yang--Mills initial data set with the property that its
  $\dot{H}^{\frac{d-2}{2}} \times \dot{H}^{\frac{d-4}{2}}$ norm is
  smaller than some universal constant $\eps^{2}_{\ast}$.  Then there
  exists a gauge transformation $O \in \dot{H}^{\frac{d}{2}}(\bbR^{d};
  \G)$ of $(a, e)$ to a caloric gauge data $(\ta, \te)$, which is
  unique up to a constant gauge transformation. Moreover, the
  corresponding caloric gauge solution $(\tA_{t,x}, \rd_{t}
  \tA_{t,x})$ exists globally in time, and obeys the a-priori bound
  \begin{equation} \label{eq:small-cal-S}
    \nrm{\tA_{x}}_{S^{\frac{d-2}{2}}} \aleq \nrm{(a,
      e)}_{\dot{H}^{\frac{d-2}{2}}\times\dot{H}^{\frac{d-4}{2}}}.
  \end{equation}
\end{theorem}
We refer the reader to \cite{OTYM1, OTYM2} for the precise definition
of the caloric gauge and the $S^{\frac{d-2}{2}}$ norm. For our
purposes, all we need to know is that
\begin{equation} \label{eq:small-cal-H} \nrm{\nb \tA_{x}}_{L^{\infty}
    \dot{H}^{\frac{d-4}{2}}} + \nrm{\tA_{x}}_{L^{2} L^{2d}} \aleq
  \nrm{\tA_{x}}_{S^{\frac{d-2}{2}}}
\end{equation}
and that the a-priori bound of the $S^{\frac{d-2}{2}}$ norm implies
the following additional control of the solution $\tA_{t,x}$
\cite[Theorem~5.1]{OTYM2}:
\begin{equation} \label{eq:small-cal-ell} \nrm{\Box \tA_{x}}_{\ell^{1}
    L^{2} \dot{H}^{\frac{d-5}{2}}} + \nrm{\rd^{\ell}
    \tA_{\ell}}_{\ell^{1} L^{2} \dot{H}^{\frac{d-3}{2}}} + \nrm{\nb
    \tA_{0}}_{\ell^{1} L^{2} \dot{H}^{\frac{d-3}{2}}}
  \aleq_{\nrm{\tA_{x}}_{S^{\frac{d-2}{2}}}}
  \nrm{\tA_{x}}_{S^{\frac{d-2}{2}}}^{2}.
\end{equation}
Combined with the initial data surgery technique
(Theorem~\ref{thm:excise}) and the patching procedure in
Section~\ref{subsec:patching}, we may now prove
Theorem~\ref{thm:imp-reg}.
\begin{proof}[Proof of Theorem~\ref{thm:imp-reg}]
  On the one hand, we have a global $\calH^{\frac{d-2}{2}}_{loc}$
  solution $A$ in $\calD_{[0, \rc)}(B_{R})$ by
  Theorem~\ref{thm:local-temp}. On the other hand, we can cover $[0,
  \rc) \times B_{R - 4 \rc}$ with cylinders $[0, \rc) \times
  B_{\rc}(x_{\alp})$, each of which is contained in a truncated cone
  $\calD_{[0, \rc)}(B_{4 \rc}(x_{\alp}))$ whose base is contained
  in $B_{R}$, i.e., $B_{4 \rc}(x_{\alp}) \subseteq B_{R}$. In each
  $\calD_{[0, \rc)}(B_{4 \rc}(x_{\alp}))$, by
  Theorem~\ref{thm:small-cal}, we have a gauge-equivalent caloric
  solution $\tA_{(\alp)}$ satisfying
  \begin{equation} \label{eq:imp-reg-tA} \nrm{\nb \tA_{(\alp)
        x}}_{L^{\infty} \dot{H}^{\frac{d-4}{2}}} + \nrm{\Box
      \tA_{(\alp) x}}_{\ell^{1} L^{2} \dot{H}^{\frac{d-5}{2}}} +
    \nrm{\rd^{\ell} \tA_{(\alp) \ell}}_{\ell^{1} L^{2}
      \dot{H}^{\frac{d-3}{2}}} + \nrm{\nb \tA_{(\alp) 0}}_{\ell^{1}
      L^{2} \dot{H}^{\frac{d-3}{2}}} \aleq \eps_{\ast}.
  \end{equation}
  In the remainder of the proof, we restrict each solution
  $\tA_{(\alp)}$ to the cylinder $[0, \rc) \times
  B_{\rc}(x_{\alp})$.

  We need to compute the regularity of the gauge transformation
  $O_{(\alp \bt)}$ between two such solutions $\tA_{(\alp)}$ and
  $\tA_{(\bt)}$. We build up the regularity of $O_{(\alp \bt)}$ in
  several stages, depending on the formula
  \begin{equation*}
    \tA_{(\alp)} = Ad(O_{(\alp \bt)}) \tA_{(\bt)} - O_{(\alp \bt); t,x} \quad \hbox{ in } [0, \rc) \times (B_{\rc}(x_{\alp}) \cap B_{\rc}(x_{\bt}))
  \end{equation*}
  In what follows, all norms are over $[0, \rc) \times
  (B_{\rc}(x_{\alp}) \cap B_{\rc}(x_{\bt}))$, and we omit the
  subscripts $(\alp)$, $(\bt)$ and $(\alp \bt)$.

  \begin{enumerate}[label=(\roman*)]
  \item {\it $L^p$ regularity.} It immediately follows that
    \[
    O_{;x}, O_{;t} \in L^\infty L^{d} \cap L^2 L^{2d}.
    \]
    Reiterating this, we also obtain
    \[
    O_{;x},O_{;t} \in L^\infty \dot H^{\frac{d-2}{2}}.
    \]
  \item{\it $\ell^1$ Besov structure for $O_{;x}$.}  Here we obtain
    \[
    O_{;x} \in \ell^1 (L^\infty \dot H^\frac{d-2}{2} \cap L^2 \dot
    H^{\frac{d-1}{2}}).
    \]
    which follows from the div-curl system\footnote{In order to appeal
      to interior regularity, we may in fact start with local data on
      slightly larger balls $B_{2 \rc}(x_{\alp})$, then shrink their
      radii to $\rc$ at this stage. We omit this minor technical
      detail.} for $O_{;x}$ (cf. Lemma~\ref{lem:div-curl-O}).

  \item {\it $\ell^1$ Besov structure for $O_{;t}$.}  Next, we obtain
    \[
    O_{;t} \in \ell^1 (L^\infty \dot H^\frac{d-2}{2} \cap L^2 \dot
    H^{\frac{d-1}{2}}).
    \]
    which is obtained by differentiating in $x$ in the $O_{;t}$
    relation. Differentiating instead in $t$, we also obtain
    \[
    \rd_{t} O_{;t} \in \ell^1 (L^\infty \dot H^\frac{d-4}{2} \cap L^2
    \dot H^{\frac{d-3}{2}}).
    \]
  \item {\it $\Box O_{;x} \in \ell^1 L^2 \dot H^{\frac{d-5}{2}}$.}
    This requires a similar bound for $[O_{;\alpha}, \partial^\alpha
    \tA]$ and for $[\partial^\alpha O_{;\alpha},\tA]$.  Both of them
    follow from the previous bounds.
  \end{enumerate}

  To summarize, we have the regularity properties:
  \begin{equation} \label{eq:imp-reg-O} O_{;x} \in \ell^1 L^2 \dot
    H^{\frac{d-1}{2}} , \qquad \rd_{t}^{2} O_{;x} \in \ell^1 L^2 \dot
    H^{\frac{d-5}{2}}, \qquad \nb O_{;t} \in \ell^1 L^2 \dot
    H^{\frac{d-3}{2}},
  \end{equation}
  where $\rd_{t}^{2} O_{;x} \in L^{2} \dot{H}^{\frac{d-5}{2}}$ follows
  by combining (ii) and (iii). These in particular imply that each $O$
  is continuous, and is close to a constant in $L^\infty$. Hence, the
  operations of pointwise multiplication, inversion, adjoint action on
  $\g$ etc. are all well-behaved for $O$ (in contrast to the general
  situation in Section~\ref{subsec:rough-gt}).

  Next step is to patch up the local gauges. Taking only the balls
  $B_{\rc}(x_{\alp})$ which cover $B_{R-4 \rc}$ and which are
  uniformly separated, Scenario~(2) in Section~\ref{subsec:patching}
  is applicable to each fixed time $\set{t} \times B_{R-4
    \rc}$. Note that the diffeomorphisms and the smooth cutoffs
  involved in the patching procedure in Scenario~(1) in
  Section~\ref{subsec:patching} all depend trivially on $t$. It
  follows that on each $[0, \rc) \times B'_{\alp}$, the gauge
  transformations $P_{(\alp)}$ obey
  \begin{equation} \label{eq:imp-reg-P} P_{;x} \in \ell^1 L^2 \dot
    H^{\frac{d-1}{2}} , \qquad \rd_{t}^{2} P_{;x} \in \ell^1 L^2 \dot
    H^{\frac{d-5}{2}}, \qquad \nb P_{;t} \in \ell^1 L^2 \dot
    H^{\frac{d-3}{2}},
  \end{equation}
  where the bound depends only on $R / \rc$ and $\eps_{\ast}$.

  It remains to verify the bound \eqref{eq:imp-reg} for the global
  gauge potential $A$, which is a consequence of
  \eqref{eq:imp-reg-tA}, \eqref{eq:imp-reg-O} and the formula
  \eqref{eq:patching-A-ball} (it is easily extended to the $0$-th
  component). Here, we only sketch the proof of $\Box A_{x} \in
  \ell^{1} L^{2} \dot{H}^{\frac{d-5}{2}}$, which is the trickiest, and
  leave the remaining cases to the reader.

  Recalling the formula \eqref{eq:patching-A-ball}, we have
  \begin{equation*}
    \Box A_{x} = \sum \chi_{\alp} \left( Ad(P_{(\alp)})  \Box \tA_{(\alp) x} - \Box P_{;x} + h.o.t.\right).
  \end{equation*}
  The higher order terms, whose precise expression is omitted, are
  estimated by \eqref{eq:imp-reg-P} and
  \eqref{eq:imp-reg-tA}. Moreover, $\Box P_{;x} = - \rd_{t}^{2} P_{;x}
  + \lap P_{;x} \in \ell^{1} L^{2} \dot{H}^{\frac{d-5}{2}}([0, \rc)
  \times B'_{\alp})$ by \eqref{eq:imp-reg-P}. Thanks to
  \eqref{eq:imp-reg-P}, $Ad(P_{(\alp)})$ may be easily removed in
  $\ell^{1} L^{2} \dot{H}^{\frac{d-5}{2}}([0,\rc) \times
  B'_{\alp})$. Then finally, $\Box \tA_{(\alp) x} \in \ell^{1} L^{2}
  \dot{H}^{\frac{d-5}{2}}([0, \rc) \times B'_{\alp})$ by
  \eqref{eq:imp-reg-tA}. \qedhere
\end{proof}

\subsection{Local theory in dimension \texorpdfstring{$d = 3$}{d=3}}
Here we sketch the proofs of Theorems~\ref{thm:local-temp-sub} and
\ref{thm:KM}. The key result is the following subcritical
initial data surgery result (cf. Theorems~\ref{thm:ext-id} and
\ref{thm:excise}):

\begin{theorem} \label{thm:excise-3} Let $\frac{1}{2} < \sgm <
  \frac{5}{2}$, and let $(a, e)$ be an $\calH^{\sgm}$ Yang--Mills
  initial data set on a convex domain $K$ in $\bbR^{3}$ satisfying
  \begin{equation} \label{eq:excise-3-hyp}
    \nrm{a}_{\dot{H}^{\frac{1}{2}}(K)} \leq \eps.
  \end{equation}
  If $\eps > 0$ is sufficiently small depending on $L(K)$, then there
  exists an $\calH^{\sgm}$ Yang--Mills initial data set $(\ba, \be)$
  in $\bbR^{3}$ that coincides with $(a, e)$ on $K$ and obeys
  \begin{align}
    \nrm{\ba}_{\dot{H}^{\sgm} \cap R(K)^{-\sgm} L^{2}} +
    \nrm{\be}_{\dot{H}^{\sgm-1} + R(K)^{\sgm-1} L^{2}} \aleq_{L(K)}
    \nrm{a}_{\dot{H}^{\sgm} \cap R(K)^{-\sgm} L^{2} (K)} +
    \nrm{e}_{\dot{H}^{\sgm-1} + R(K)^{\sgm-1}
      L^{2}(K)}. \label{eq:excise-3}
  \end{align}
  It can be arranged so that the association $(a, e) \mapsto (\ba,
  \be)$ is equivariant under constant gauge transformations, and so
  that $(a, e) \mapsto (\ba, \be)$ is locally Lipschitz
  continuous. Moreover, if $(a, e)$ is smooth, then so is $(\ba,
  \be)$.
\end{theorem}
\begin{proof}
  By rescaling, we set $R(K) = 1$ so that $\dot{H}^{\sgm} \cap
  R(K)^{-\sgm} L^{2} \aeq H^{\sgm}$ and $\dot{H}^{\sgm-1} +
  R(K)^{\sgm-1} L^{2} \simeq H^{\sgm-1}$. As in the proof of
  Theorem~\ref{thm:ext-id}, we apply the universal extension operator
  $\mathfrak{E}$ to $(a, e)$ to first obtain $(\ba, \be') \in H^{\sgm}
  \times H^{\sgm-1}(\bbR^{3})$. Then the error for the Gauss equation
  $h = (\covD^{(\ba)})^{\ell} \be'$ is supported outside $K$ and obeys
  $\nrm{h}_{H^{\sgm-2}} \aleq_{\nrm{\ba}_{\dot{H}^{\frac{1}{2}}}}
  \nrm{e}_{H^{\sgm-1}(K)}$. Since
  \begin{equation*}
    \nrm{\ba}_{\dot{H}^{\frac{1}{2}}} \aleq_{L(K)} \nrm{a}_{\dot{H}^{\frac{1}{2}}(K)} \leq \eps, 
  \end{equation*}
  Proposition~\ref{prop:gauss-small} is applicable if $\eps > 0$ is
  chosen sufficiently small. Thus $d = - T_{\ba} h$ satisfies
  \begin{equation*} (\covD^{(\ba)})^{\ell} d_{\ell} = - h, \qquad
    \nrm{d}_{H^{\sgm-1}} \aleq \nrm{h}_{H^{\sgm-2}} \aleq
    \nrm{\be'}_{H^{\sgm-1}},
  \end{equation*}
  and vanishes in $K$. It follows that $(\ba, \be = \be' + d)$ is a
  Yang--Mills initial data set obeying the desired bound
  \eqref{eq:excise-3}. The higher regularity and local Lipschitz
  properties are obvious by construction. Finally, equivariance under
  constant gauge transformations can be insured by fixing a particular
  construction, conjugating by elements of $\G$, and then averaging.
\end{proof}

Next, we recall the temporal gauge small data local well-posedness of
Tao.
\begin{theorem} [{\cite{TaoYM}}]\label{thm:small-temp-sub}
  Let $\sgm > \frac{3}{4}$. If the $\calH^{\sgm}$ norm of the initial
  data set $(a, e)$ is sufficiently small, then the corresponding
  solution $(A_{t,x}, \rd_{t} A_{t,x})$ in the temporal gauge $A_{0} =
  0$ exists in $C_{t}((-1, 1); H^{\sgm} \times H^{\sgm-1})$, and obeys
  the a-priori bound
  \begin{equation*}
    \nrm{(A_{x}, \rd_{t} A_{x})}_{L^{\infty} (H^{\sgm} \times H^{\sgm-1})} \aleq \nrm{(a, e)}_{H^{\sgm} \times H^{\sgm-1}}.
  \end{equation*}
  The solution is unique among the local-in-time limits of smooth
  solutions, and it depends in a locally Lipschitz manner on the data
  $(a, e) \in H^{\sgm} \times H^{\sgm-1}$.
\end{theorem}

Now we are ready to prove Theorem~\ref{thm:local-temp-sub}.
\begin{proof}[Sketch of Proof of Theorem~\ref{thm:local-temp-sub}]
  As in the proof of Theorem~\ref{thm:local-temp}, the idea is to
  patch together the small local-in-spacetime solutions constructed
  using Theorems~\ref{thm:excise-3} and \ref{thm:small-temp-sub} in
  the temporal gauge.

  It suffices to consider $\frac{3}{4} < \sgm < \frac{5}{2}$. Observe
  that, by subcriticality, the $\calH^{\sgm}_{loc}$ norm obeys the
  following one-sided scaling property:
  \begin{equation*}
    \nrm{(a^{(\lmb)}, e^{(\lmb)})}_{\calH^{\sgm}_{loc}}
    \aleq \lmb^{\sgm - \frac{1}{2}} \nrm{(a, e)}_{\calH^{\sgm}_{loc}}	 \quad \hbox{ for } \lmb \leq 1.
  \end{equation*}
  Here $(a^{(\lmb)}, e^{(\lmb)})(x) = (\lmb a, \lmb^{2} e)(\lmb x)$ is
  the invariant scaling. Choosing
  \begin{equation*}
    \lmb \aeq \left( \eps_{\ast} \nrm{(a, e)}_{\calH^{\sgm}_{loc}}^{-1} \right)^{\frac{2}{\sgm-1}} ,
  \end{equation*}
  we may insure that $\nrm{(a^{(\lmb)},
    e^{(\lmb)})}_{\calH^{\sgm}_{loc}} \ll \eps_{\ast}$. Choosing
  $\eps_{\ast} > 0$ sufficiently small, we may apply
  Theorem~\ref{thm:excise-3} to each $(a^{(\lmb)}, e^{(\lmb)})
  \restriction_{B_{2}(x)}$ to find an extension $(\ba^{(\lmb)},
  \be^{(\lmb)})$, and then Theorem~\ref{thm:small-temp-sub} to this
  global-in-space small data to obtain a temporal gauge solution
  $A^{(\lmb)}$ on the time interval $(-1, 1)$. Proceeding as in the
  proof of Theorem~\ref{thm:local-temp}, we obtain a well-posed
  temporal gauge solution for $(a^{(\lmb)}, e^{(\lmb)})$ on $(-1,
  1)$. By rescaling back, the theorem follows with an explicit lower
  bound $T \ageq \nrm{(a,
    e)}_{\calH^{\sgm}_{loc}}^{-\frac{2}{\sgm-1}}$.
\end{proof}

Finally, Theorem~\ref{thm:KM} is an easy corollary of Uhlenbeck's
lemma (at subcritical regularity) and
Theorem~\ref{thm:local-temp-sub}.
\begin{proof}[Sketch of Proof of Theorem~\ref{thm:KM}]
  By conservation of energy, it suffices to prove that the temporal
  gauge solution given by Theorem~\ref{thm:local-temp-sub} exists on a
  interval of length $T(\nrm{(F[a], e)}_{L^{2}_{loc}})$, where
  $\nrm{\cdot}_{L^{2}_{loc}} = \sup_{x \in \bbR^{3}}
  \nrm{\cdot}_{L^{2}(B_{1}(x))}$. As before, we have the one-sided
  scaling property
  \begin{equation*}
    \nrm{(F[a^{(\lmb)}], e^{(\lmb)})}_{L^{2}_{loc}}
    \aleq \lmb^{\frac{1}{2}} \sup_{x \in \bbR^{3}} \nrm{(F[a], e)}_{L^{2}_{loc}}	 \quad \hbox{ for } \lmb \leq 1.
  \end{equation*}
  Choosing $\lmb \aeq \eps_{\ast} \nrm{(F[a], e)}_{L^{2}_{loc}}^{-2}$,
  we may insure that the LHS is $\aleq \eps_{\ast}$. In what follows,
  we work with the rescaled data $(a^{(\lmb)}, e^{(\lmb)})$; we omit
  the superscript $(\lmb)$ for simplicity. For the rescaled data, we
  wish to show that the corresponding temporal gauge solution given by
  Theorem~\ref{thm:local-temp-sub} exists on the unit time interval
  $[0, 1)$.

  Fix a unit ball $B = B_{1}(x_{0})$. Applying Uhlenbeck's lemma
  \cite[Theorem~1.3]{MR648356} (which is possible if we take
  $\eps_{\ast}$ sufficiently small), we find $O \in \calG^{2, 2}(2B)$
  such that
  \begin{equation*}
    \nrm{O}_{H^{2}(B)} \aleq \nrm{a}_{H^{1}(2B)},
  \end{equation*}
  and $(\ta, \te) = (Ad(O) a - O_{;x}, Ad(O) e)$ obeys
  \begin{equation*}
    \nrm{(\ta, \te)}_{H^{1} \times L^{2}(2B)} \aleq \nrm{(F[a], e)}_{L^{2}(2B)} \aleq \eps_{\ast}. 
  \end{equation*}
  By Theorem~\ref{thm:small-temp-sub} (taking $\eps_{\ast}$ even
  smaller if necessary), we find a temporal gauge solution $\tA$ with
  data $(\ta, \te)$ on $(-1, 1)$. Applying the $H^{2}(2B)$ gauge
  transformation $O^{-1}$, we obtain a temporal gauge solution $A =
  Ad(O^{-1}) \tA + O^{-1} O_{;t, x}$ in $\calD_{[0, 1)}(2B)$. It can
  be easily verified that this solution is the limit of smooth
  temporal gauge solutions; hence it coincides with the solution given
  by Theorem~\ref{thm:local-temp-sub} in $\calD_{[0, 1)}(2B)$. Since
  this procedure can be applied to any unit ball $B \subseteq
  \bbR^{3}$, it follows that the temporal gauge solution exists on the
  time interval $[0, 1)$, as desired. \qedhere
\end{proof}
\section{{Harmonic Yang--Mills connections with compact structure
    group}} \label{sec:threshold} The goal of this section is to prove
Theorem~\ref{thm:thr}.  We proceed in two steps, in increasing
generality.

\pfstep{Step~1: $\G$ is simple, compact and simply connected} Assume
that $\G$ is compact and simply connected, and also that $\g$ is
\emph{simple}, i.e., it is nonabelian ($[\g, \g] \neq 0$) and there is
no nonzero proper ideal. As we will see, this case turns out to be
completely analogous to the model case $\G = SU(2)$.

We need some algebraic preliminaries on compact simple Lie algebras
over $\bbR$. We only sketch the part of the theory that is needed for
us; for a more comprehensive treatment, see \cite[Chapters~II and
IV]{Knapp}.

A maximal abelian subalgebra $\h$ of $\g$ is called a \emph{Cartan
  subalgebra}. Given such a $\h$, consider $\set{ad(H) : \g \to \g}_{H
  \in \h}$, which is a family of commuting anti-self-adjoint
operators. Thus, viewed as linear operators on the complexification
$\g_{\bbC} = \g \otimes_{\bbR} \bbC$, they are simultaneously
diagonalizable with purely imaginary (or zero) eigenvalues. A nonzero linear
functional $\alp \in \h^{\ast}$ is called a \emph{root}\footnote{A
  more standard definition (used in \cite{Knapp}) is to define roots
  as $\alp \in \h_{\bbC}^{\ast}$ such that $\cap_{H \in \h_{\bbC}}
  ker(ad(H) - \alp(H)) \neq \set{0}$. This differs from our definition
  by a factor of $i$.} if the simultaneous eigenspace (called the
\emph{root space})
\begin{equation*}
  \g_{\bbC, \alp} = \set{A \in \g_{\bbC} : ad(H) A = i \alp(H) A, \ \forall H \in \h}
\end{equation*}
is nonzero. We write $\Dlt$ for the space of all roots. By the
preceding discussion, we see that
\begin{equation*}
  \g_{\bbC} = \h_{\bbC} \oplus \bigoplus_{\alp \in \Dlt} \g_{\bbC, \alp}
\end{equation*}
as vector spaces. In particular, $\Dlt \neq \set{0}$; in fact, it
spans $\h^{\ast}$.
%
%
%
It is a fundamental result of Cartan that all Cartan subalgebras are
related to each other by an $Ad(O)$-action; thus $\Dlt$ is independent
of the choice of $\h$.

To each $\alp \in \Dlt$, we use the inner product $\brk{\cdot, \cdot}$
to associate $H_{\alp} \in \h$ such that
\begin{equation*}
  \alp(H) = \brk{H_{\alp}, H}, \qquad H \in \h,
\end{equation*}
and define the induced inner product on $\Dlt$ by $\brk{\alp, \bt} =
\brk{H_{\alp}, H_{\bt}}$. The roots with the largest norm are called
the \emph{highest roots}.

Clearly, if $\alp \in \Dlt$, then $- \alp \in \Dlt$ with $\g_{\bbC,
  -\alp} = \overline{\g_{\bbC, \alp}}$. For any $E_{\alp} \in
\g_{\bbC, \alp}$, by definition,
\begin{equation*}
  [H_{\alp}, E_{\alp}] = i \alp(H_{\alp}) E_{\alp} = i \brk{\alp, \alp} E_{\alp}, \qquad
  [H_{\alp}, \overline{E_{\alp}}] = - i \alp(H_{\alp}) \overline{E_{\alp}} = - i \brk{\alp, \alp} \overline{E_{\alp}}.
\end{equation*}
Moreover, $\dim_{\bbC }\g_{\bbC, \alp} = 1$ and for any $E_{\alp} \in
\g_{\bbC ,\alp}$, we have
\begin{equation*}
  \brk{E_{\alp}, E_{\alp}}= 0, \qquad [E_{\alp}, \overline{E_{\alp}}] = i \brk{E_{\alp}, \overline{E_{\alp}}} H_{\alp},
\end{equation*}
where $\brk{\cdot, \cdot}$ is extended to $\g_{\bbC}$ in a
$\bbC$-bilinear fashion. For the proofs of the last properties, see
\cite[Section~II.4]{Knapp}.

Every root generates an embedding of $su(2)$ into $\g$. More
precisely, given a root $\alp \in \Dlt$, normalize $E_{\alp}$ so that
\begin{equation*}
  \brk{E_{\alp}, \overline{E_{\alp}}} = \frac{2}{\brk{\alp, \alp}},
\end{equation*}
and consider $\bfi_{\alp}, \bfj_{\alp}, \bfk_{\alp} \in \g$ defined by
\begin{equation*}
  \bfi_{\alp} =  (E_{\alp} + \overline{E_{\alp}}), \quad
  \bfj_{\alp} = i (E_{\alp} - \overline{E_{\alp}}), \quad
  \bfk_{\alp} = \frac{2}{\brk{\alp, \alp}} H_{\alp}.
\end{equation*}
Then it is straightforward to verify that $\set{\bfi_{\alp},
  \bfj_{\alp}, \bfk_{\alp}}$ generate an $su(2)$-subalgebra, i.e.,
\begin{equation} \label{eq:lie-alg-su2} [\bfi_{\alp}, \bfj_{\alp}] = 2
  \bfk_{\alp}, \qquad [\bfj_{\alp}, \bfk_{\alp}] = 2 \bfi_{\alp},
  \qquad [\bfk_{\alp}, \bfi_{\alp}] = 2 \bfj_{\alp}.
\end{equation}
%
Indeed, \eqref{eq:lie-alg-su2} are precisely the Lie bracket relations
satisfied the following standard basis of $su(2)$:
\begin{equation*}
  \bfi = \left(\begin{array}{cc} 0 & 1 \\ -1 & 0 \end{array}\right), \qquad
  \bfj = \left(\begin{array}{cc} 0 & i \\ i & 0 \end{array}\right), \qquad
  \bfk = \left(\begin{array}{cc} i & 0 \\ 0 & -i \end{array}\right).
\end{equation*}
Note also that $\bfi_{\alp}, \bfj_{\alp}, \bfk_{\alp}$ obeys
\begin{equation} \label{eq:lie-norm-alp} \abs{\bfi_{\alp}}^{2} =
  \abs{\bfj_{\alp}}^{2} = \abs{\bfk_{\alp}}^{2} = \frac{4}{\brk{\alp,
      \alp}}.
\end{equation}

By simplicity, all symmetric $Ad$-invariant bilinear functions on $\g$
(of which $\brk{\cdot, \cdot}$ is an example) are constant multiples
of each other \cite[Corollary~4.9]{Knapp}. Multiplying $\brk{\cdot,
  \cdot}$ by a suitable constant, which does not change the conclusion
of Theorem~\ref{thm:thr}, we may assume that:
\begin{equation} \label{eq:lie-norm} \hbox{The highest roots in $\g$
    have $\brk{\alp, \alp} = 2$.}
\end{equation}
When $\G = SU(n)$, this amounts to taking $\brk{A, B} = -\tr(AB)$. We
now recall the following well-known result of Bott \cite{MR0087035}
concerning the third homotopy group $\pi_{3}(\G)$ of $\G$:
\begin{theorem} \label{thm:bott} Let $\G$ be a simple, compact, simply
  connected Lie group. Then $\pi_{3}(\G) \simeq \bbZ$. Any Lie group
  homomorphism $\varphi : SU(2) \to \G$, induced by the Lie algebra
  homomorphism
  \begin{equation*}
    \ud \varphi : su(2) \to \g, \ (\bfi, \bfj, \bfk) \mapsto (\bfi_{\alp}, \bfj_{\alp}, \bfk_{\alp})
  \end{equation*}
  for a highest root $\alp$ in $\g$, induces an isomorphism
  $\pi_{3}(SU(2)) \to \pi_{3}(\G)$.
\end{theorem}
The identification $\pi_{3}(\G) \simeq \bbZ$ is due to Bott
\cite{MR0087035}. For the proof that such a $\varphi$ induces an
isomorphism, see Atiyah--Hitchin--Singer \cite[Section~8]{AHS}. By
our normalization \eqref{eq:lie-norm}, $\ud \varphi$ is isometric.

Our goal now is to prove an analogue of
Theorem~\ref{thm:instanton-SU2} concerning topological classes,
characteristic numbers and instantons. Let $a$ be a $\calA^{1,
  2}_{loc}$ connection on $\bbR^{4}$ with finite energy, and let
$O_{(\infty)}$ be a gauge at infinity for $a$ (which exists thanks to
Theorem~\ref{thm:goodrep}). By Theorem~\ref{thm:bott}, $[O_{(\infty)}]
= - \kpp [\varphi]$ for some $\kpp \in \bbZ$. We claim that:
\begin{claim}
  We have $\ch = - 8 \pi^{2} \kpp$. Moreover, there exists an
  instanton for each $\kpp$ with energy $8 \pi^{2} \abs{\kpp}$.
\end{claim}

To prove the claim, note that each self-dual (resp. anti-self-dual)
$SU(2)$-connection $\ta_{\kpp}$ with second Chern number $c_{2} =
-\kpp$ where $\kpp > 0$ (resp. $\kpp < 0$) induces a self-dual
(resp. anti-self-dual) $\G$-connection $a_{\kpp} = \ud
\varphi(\ta_{\kpp})$ by the Lie algebra homomorphism $\ud \varphi :
su(2) \to \g$. Since $\ud \varphi$ preserves the normalized
$Ad$-invariant inner product, which equals $- \tr(AB)$ on $su(2)$, we
have
\begin{align*}
  \ch =& \int_{\bbR^{4}} - \brk{\ud \varphi(F[\ta_{\kpp}]), \ud \varphi(F[\ta_{\kpp}])} = \int_{\bbR^{4}} \tr (F[\ta_{\kpp}] \wedge F[\ta_{\kpp}]) = 8 \pi^{2} c_{2} \\
  \spE(a_{\kpp}) =& \frac{1}{2} \int_{\bbR^{4}} \brk{\ud
    \varphi(F_{jk}[\ta_{\kpp}]), \ud \varphi(F^{jk}[\ta_{\kpp}]) } =
  \frac{1}{2} \int_{\bbR^{4}} - \tr (F_{jk}[\ta_{\kpp}]
  F^{jk}[\ta_{\kpp}]) = 8 \pi^{2} \abs{c_{2}}.
\end{align*}
Moreover, by a standard computation, the degree of a gauge at infinity
$\tilde{O}_{\kpp (\infty)}$ for $\ta_{\kpp}$, viewed as a map
$\bbS^{3} \to SU(2) \simeq \bbS^{3}$, is equal to $c_{2} = \kpp$ (with
the appropriate choices of the orientations). Correspondingly,
$O_{\kpp (\infty)} = \varphi \circ \tilde{O}_{(\kpp (\infty)}$ is a
gauge at infinity for $a_{\kpp}$, and since $\varphi$ induces the
isomorphism $\pi_{3}(SU(2)) \to \pi_{3}(\G)$, we have
$[O_{\kpp(\infty)}] = - \kpp [\varphi]$. Since $\ch$ is dependent only
on the topological class, the claim follows.


Next, analogous to Theorem~\ref{thm:GKS-SU2}, we claim that:
\begin{claim}
  Let $a$ be a finite energy harmonic Yang--Mills connection, which is
  not an instanton. Then
  \begin{equation*}
    \spE(a) \geq \abs{\ch} + 16 \pi^{2}.
  \end{equation*}
\end{claim}
In essence, this is \cite[Corollary~1.2]{GKS}. However, to insure that
we obtain the sharp bound, we need to verify that the proof goes
through for our choice of $\brk{\cdot, \cdot}$, without relying on an
embedding $\g \subset so(n)$ to normalize $\brk{\cdot, \cdot}$ as in
\cite{GKS}. For this purpose, we have the following replacement of
\cite[Lemma~2.1]{GKS}:
\begin{lemma} \label{lem:GKS-inner} Under our normalization
  \eqref{eq:lie-norm}, we have
  \begin{equation*}
    \abs{[A, B]} \leq \sqrt{2} \abs{A} \abs{B} \qquad \hbox{ for any } A, B \in \g.
  \end{equation*}
  with equality if and only if, up to an $Ad(O)$-action, $A$ and $B$
  are proportional to two of $\set{\bfi_{\alp}, \bfj_{\alp},
    \bfk_{\alp}}$ for some highest root $\alp$.
\end{lemma}
\begin{proof}
  Consider a maximal abelian subalgebra $\h$ containing
  $A$. Eigenvalues of $ad(A)$ are $\set{0, i \alp(A)}_{\alp \in
    \Dlt}$. By \eqref{eq:lie-norm}, $\abs{\alp} \leq \sqrt{2}$. Thus,
  \begin{equation*}
    \abs{[A, B]} = \abs{ad(A) B} \leq \sup_{\alp \in \Dlt} \abs{\alp(A)} \abs{B} \leq \sup_{\alp \in \Dlt} \abs{\alp} \abs{A} \abs{B} \leq \sqrt{2} \abs{A} \abs{B}.
  \end{equation*}
  In order for the equalities to hold, $\alp$ must be a highest root,
  $A = \abs{A} H_{\alp} = \abs{A} \bfk_{\alp}$, and $B \in
  span(\bfi_{\alp}, \bfj_{\alp})$. Since $Ad(\exp(s \bfk_{\alp}))$
  simply rotates the plane $span(\bfi_{\alp}, \bfj_{\alp})$, and
  leaves $\bfk_{\alp}$ invariant, we see that $Ad(\exp(s \bfk_{\alp}))
  B$ is parallel to $\bfi_{\alp}$ for an appropriate choice of $s \in
  \bbR$. Finally, the converse is easy to verify. \qedhere
\end{proof}
The proof in \cite{GKS} now goes through for a $\G$-bundle with
normalization \eqref{eq:lie-norm} with the parameters $\gmm_{0} =
\sqrt{2}$ and $\gmm_{1} = \frac{2}{\sqrt{3}} \gmm_{0} =
\frac{4}{\sqrt{6}}$. The $SU(2)$-instanton with $\kpp = 1$, which we
constructed above, saturates the inequalities in \cite{GKS}, exactly
as in \cite[Remark~2.7 and Section~3.2]{GKS}.

\pfstep{Step~2: $\G$ is a general nonabelian compact Lie group}
Finally, we consider a general nonabelian compact Lie group $\bfG$ and
prove Theorem~\ref{thm:thr}.

Observe that the $Ad$-invariant inner product on $\g$ can be used to
define the orthogonal complement $\h^{\perp}$ of an ideal $\h
\subseteq \g$, which is also an ideal. Thus $\g$ admits the direct-sum
splitting
\begin{equation*}
  \g = \tilde{\g}_{1} \oplus \cdots \oplus \tilde{\g}_{\tilde{n}}
\end{equation*}
as Lie algebra ideals, where each summand has no proper nonzero
ideal. In fact, it is either $1$-dimensional, and thus abelian, or
simple. Since $\G$ is assumed to be nonabelian, at least one summand
is simple. Thus, we arrive at the decomposition
\begin{equation*}
  \g = \g_{1} \oplus \cdots \oplus \g_{n} \oplus \mathfrak{a}.
\end{equation*}
where $n \geq 1$, each $\g_{i}$ is simple, and $\mathfrak{a}$ is
abelian. As a result, the universal cover $\tilde{\G}$ of $\G$ splits
into
\begin{equation*}
  \tilde{\G} = \Pi_{i} \G_{i} \times \bbR^{r}
\end{equation*}
where $\G_{i}$ is the simply connected Lie group corresponding to
$\g_{i}$, and $r = \dim \mathfrak{a}$.  Denote by
$\boldsymbol{\pi}_{i}$ the projection $\G \to \G_{i}$, and by $\ud
\boldsymbol{\pi}_{i}$ the corresponding projection $\g \to \g_{i}$,
with the convention $\G_{n+1} = \bbR^{r}$, $\g_{n+1} = \mathfrak{a}$.

As we are working with global gauge potentials on $\bbR^{4}$, the
splitting allows us to decompose any $a$ into components $\ud
\boldsymbol{\pi}_{i} (a)$, which are completely decoupled from each
other. We have the splitting
\begin{align}
  \ch = & \int_{\bbR^{4}} - \brk{F[a], F[a]} = \sum_{i}
  \int_{\bbR^{4}} - \brk{\ud \boldsymbol{\pi} (F[a]), \ud
    \boldsymbol{\pi} (F[a])}
  = \sum_{i} \ch(\ud \boldsymbol{\pi}_{i}(a)), \label{eq:cpt-ch} \\
  \spE(a) = & \frac{1}{2} \int_{\bbR^{4}} \brk{F_{jk}[a], F^{jk}[a]} =
  \sum_{i} \frac{1}{2} \int_{\bbR^{4}} \brk{\ud \boldsymbol{\pi}
    (F_{jk}[a]), \ud \boldsymbol{\pi} (F^{jk}[a])} = \sum_{i} \spE(\ud
  \boldsymbol{\pi}_{i}(a)). \label{eq:cpt-en}
\end{align}
Moreover, $a$ is a harmonic Yang--Mills connection if and only if each
$\ud \boldsymbol{\pi}_{i}(a)$ is. In this case, $\ud
\boldsymbol{\pi}_{n+1}(a) = 0$, since no nontrivial finite energy
harmonic $2$-form exists on $\bbR^{4}$.

For each compact simple $\G_{i}$, let $E_{i}$ be the energy of a
first instanton; from Step~1, we know that $E_{i} =
\frac{16}{\brk{\alp, \alp}} \pi^{2}$, where $\alp$ is a highest root
in $\g_{i}$. Reordering the factors if necessary, we may arrange so
that $E_{1} \leq E_{2} \leq \ldots \leq E_{n}$. In
particular, $E_{1}$ coincides with the infimum in
Theorem~\ref{thm:thr}, and part (1) follows.

To prove part (2), note that if $a$ is a finite energy harmonic
Yang--Mills connection with energy $< 2 E_{1} \leq 2 E_{i}$,
then by Step~1, each $\ud \boldsymbol{\pi}_{i}(a)$ is either zero or a
first instanton. Immediately by \eqref{eq:cpt-en}, we also see that
exactly one of $\ud \boldsymbol{\pi}_{i}(a)$ is nonzero. Thus
$\abs{\ch} = \abs{\ch(\ud \boldsymbol{\pi}_{i}(a))} = \spE(\ud \boldsymbol{\pi}_{i}(a)) =
\spE(a)$, as desired.

\bibliographystyle{ym} 
\bibliography{ym}

\end{document}